\documentclass[a4paper,10pt]{article}
\usepackage{amsmath, amssymb, amsfonts, amscd}
\usepackage{amsthm}
\usepackage{eucal, xspace}
\usepackage{epic}
\usepackage{eepic}
\usepackage[matrix,line,arrow]{xy}

\theoremstyle{definition}
\newcounter{foo}[section]

\newtheorem{defi}[foo]{Definition}
\newtheorem{rem}[foo]{Remark}
\newtheorem{ex}[foo]{Example}
\newtheorem{thm}[foo]{Theorem}
\newtheorem{lem}[foo]{Lemma}
\newtheorem{prop}[foo]{Proposition}

\newtheorem{cor}[foo]{Corollary}

\newtheorem*{thm*}{Theorem}

\newcommand{\rel}{rel.\@\xspace}

\newcommand{\ie}{i.e.\@\xspace}

\newcommand{\etc}{etc.\@\xspace}

\newcommand{\fg}{f.g.\@\xspace}
\newcommand{\Wlog}{W.l.o.g.\@\xspace}
\newcommand{\co}{\colon}
\newcommand{\fat}{
}

\newcommand{\ol}[1]{l_{2q+1}(#1)}

\newcommand{\oL}[1]{L^{}_{2q+1}(#1)}
\newcommand{\ols}[1]{l_{2q+1}(#1)}
\newcommand{\oels}[1]{\mathcal{E}\ols{#1}}
\newcommand{\oLs}[1]{L^s_{2q+1}(#1)}

\newcommand{\eL}[1]{L^{}_{2q}(#1)}

\newcommand{\bN}{\mathbb N}
\newcommand{\bZ}{\mathbb Z}

\newcommand{\svec}[1]{\left( \begin{smallmatrix} #1 \end{smallmatrix} \right)}
\newcommand{\mat}[1]{\begin{pmatrix} #1 \end{pmatrix}}

\newcommand{\mt}{\longmapsto}
\newcommand{\ra}{\rightarrow}
\newcommand{\lra}{\longrightarrow}
\newcommand{\ras}{\rightarrow}
\newcommand{\isora}{\stackrel{\cong}{\lra}}

\newcommand{\dlra}[1]{\stackrel{#1}{\longrightarrow}}

\newcommand{\hQ}{\widehat{Q}}
\newcommand{\del}{\partial}
\newcommand{\hra}{\hookrightarrow}

\providecommand{\id}{\mathop{\rm id}\nolimits}
\providecommand{\im}{\mathop{\rm Im}\nolimits}
\providecommand{\rk}{\mathop{\rm rk}\nolimits}

\providecommand{\cha}{\mathop{\rm char}\nolimits}
\providecommand{\coker}{\mathop{\rm cok}\nolimits}
\providecommand{\Hom}{\mathop{\rm Hom}\nolimits}

\providecommand{\Rad}{\mathop{\rm Rad}\nolimits}

\providecommand{\asr}{\mathop{\rm asr}\nolimits}
\providecommand{\ind}{\mathop{\rm ind}\nolimits}

\providecommand{\Kdim}{\mathop{\rm Kdim}\nolimits}
\newcommand{\Wh}{\mathop{\rm Wh}\nolimits}

\providecommand{\Iso}{\mathop{\rm Iso}\nolimits}
\providecommand{\bIso}{\mathop{\rm bIso}\nolimits}
\providecommand{\sbIso}{\mathop{\rm sbIso}\nolimits}
\providecommand{\Aut}{\mathop{\rm Aut}\nolimits}
\providecommand{\bAut}{\mathop{\rm bAut}\nolimits}

\providecommand{\sbAut}{\mathop{\rm sbAut}\nolimits}

\newcommand{\eq}{$\epsilon$-quadratic quasi-form\-ation}
\newcommand{\eqs}{$\epsilon$-quadratic quasi-form\-ations}
\newcommand{\eqf}{$\epsilon$-quadratic form\-ation}
\newcommand{\eqfs}{$\epsilon$-quadratic form\-ations}

\newcommand{\on}[2]{\mathop{\null#2}\limits^{#1}}
\newcommand{\oover}[1]{\on{\circ}{#1}}

\title{Stably diffeomorphic manifolds and $\ol{\bZ[\pi]}$}

\author{Diarmuid Crowley and J\"org Sixt}

\begin{document}

\maketitle


\begin{abstract}
The Kreck monoids $l_{2q+1}(\bZ[\pi])$ detect $s$-cobordisms amongst
certain bordisms between stably diffeomorphic $2q$-dimensional manifolds and
generalise the Wall surgery obstruction groups,
$L_{2q+1}^s(\bZ[\pi]) \subset \ol{\bZ[\pi]}$.   In this paper we
identify $\ol{\bZ[\pi]}$ as the edge set of a directed graph with
vertices a set of equivalence classes of quadratic forms on
finitely generated free $\bZ[\pi]$ modules.  Our main theorem
computes the set of edges $l_{2q+1}(v, v') \subset \ol{\bZ[\pi]}$
between the classes of the forms $v$ and $v'$ via an exact
sequence
\[
\oLs{\bZ[\pi]}\stackrel{\rho}{\lra} l_{2q+1}(v, v') \stackrel{\delta}{\lra} \sbIso(v, v') \stackrel{{\kappa }}{\lra} L^s_{2q}(\bZ[\pi]).
\]
Here ${\rm sbIso}(v, v')$ denotes the set of ``stable boundary isomorphisms" between the algebraic boundaries of $v$ and $v'$.  As a consequence we deduce new classification results for stably diffeomorphic manifolds.
\end{abstract}

\section{Introduction} \label{Intro}
Let $M_0$ and $M_1$ be connected, compact $2q$-dimensional, smooth manifolds ($q \geq 2$) with (possibly empty) boundary and the same Euler characteristic.   A stable diffeomorphism from $M_0$ to $M_1$ is a diffeomorphism
\[h \co M_0 \sharp_k(S^q \times S^q) \cong M_1 \sharp_k(S^q \times S^q).\]
The cancellation problem is to classify stably diffeomorphic manifolds and we briefly mention its history in section \ref{sdmsec}.  The most systematic approach to date is via the surgery obstruction monoids $\ol{\bZ[\pi]}$ \cite{Kre99} which depend upon the twisted group ring $\bZ[\pi]$ defined by the orientation character of the fundamental group of $M_0$ and the parity of $q$.  Henceforth we assume that $q \neq 3, 7$ (see Remark \ref{6and14rem} for these dimensions).  A stable diffeomorphism $h$ gives rise to an element $\Theta(h)$ in $l_{2q+1}(\bZ[\pi])$ (see Lemma \ref{SDcondlem}) and a fundamental theorem of \cite{Kre99} states that there is a submonoid $\mathcal{E}\ol{\bZ[\pi]} \subset \ol{\bZ[\pi]}$ such that:
\[\text{if $\Theta(h) \in \mathcal{E}\ol{\bZ[\pi]}$ then there is an $s$-cobordism between $M_0$ and $M_1$.}\]
Despite its topological significance, no computation of $\ol{\bZ[\pi]}$ appears in the literature for any group $\pi$.  In this paper we give exact sequences which allow us to compute $\ol{\bZ[\pi]}$ as a set and also its Grothendieck group.  First we give some topological applications.

Recall that a finitely presented group is polycyclic-by-finite if it has a subnormal series where the quotients are either cyclic or finite (see Definition \ref{polydefi}).  The number of infinite cyclic quotients is an invariant of $\pi$ called the Hirsch number $h(\pi)$.  We define $h'(\pi,q)$ to be $0$ ({\em resp.} 1) if $\pi$ is trivial and $q$ is odd ({\em resp.} even), $2$ if $\pi$ is finite but non-trivial and $h(\pi)+3$ if $\pi$ is infinite.

\begin{thm}\label{thm1}
Suppose that the fundamental group $\pi$ of $M$ is polycyclic-by-finite and that $M \cong N \sharp_k(S^q \times S^q)$ where $k \geq h'(\pi, q)$.  Then every manifold stably diffeomorphic to $M$ is $s$-cobordant to $M$.
\end{thm}

\begin{rem}
For finite $\pi$ Theorem \ref{thm1} is \cite{Kre99}[Corollary 4] which is based on algebraic results of Bass \cite{Bas73}.  Bass's results were first applied to the topological cancellation problem for $4$-manifolds by Hambleton and Kreck \cite{HaKr88} who later showed that the above theorem holds in the topological category for closed $4$-dimensional manifolds with finite fundamental group when $M \cong N \sharp(S^q \times S^q)$ \cite{HaKr93b}.

Recently Khan \cite{Kh04} has also proven cancellation results for closed topological $4$-dimensional manifolds with infinite fundamental group.  While Khan's bound is sometimes one better than ours, his methods do not apply for all polycyclic-by-finite groups: for example certain semi-direct products $(\bZ \times \bZ) \times_\alpha \bZ$ where $\alpha \in GL_2(\bZ)$ has infinite order.
\end{rem}

Let $B\pi$ be an aspherical space with fundamental group $\pi$.
Our next theorem concerns the representation of elements in the
$(2q+1)$-dimensional oriented bordism group of $B\pi$ via mapping
tori.

\begin{thm} \label{thm2}
Suppose that $M$ is an oriented manifold with polycyclic-by-finite fundamental group
$\pi$, that $M = N \sharp_k(S^q \times S^q)$ for $k \geq h'(\pi,
q)$ and that the canonical map $M \ra BSO \times B\pi$ classifying
the tangent bundle of $M$ and the universal cover of $M$ is a
$q$-equivalence.  Then every element of the oriented bordism
group, $\Omega_{2q+1}^{SO}(B\pi)$, can be represented by the
mapping torus of an orientation preserving diffeomorphism $f\co M
\cong M$ which induces the identity on $\pi$.
\end{thm}

Whereas the above theorems concern $2q$-dimensional manifolds with appropriately
large intersection forms our next theorem considers $2q$-dimensional manifolds
with small intersection forms.  Let $K \subset H_q(M)$ (where
$\bZ[\pi]$ coefficients are understood) be the submodule of elements
which evaluate to zero when paired with all decomposable elements of
$H^q(M)$.  Further let $\lambda_M|_K$ be the restriction of the
equivariant intersection form of $M$, $\lambda_M \co  H_q(M) \times
H_q(M) \ra \bZ[\pi]$ to $K \times K$.  Let also $\pi = \pi_1(M)$ and let ${\rm UWh}(\pi)$ ({\em resp}. ${\rm U'Wh}(\pi))$ be the subgroup of the Whitehead group of $\pi$ given by torsions arising from automorphisms of quadratic ({\em resp}. symmetric) hyperbolic forms (see subsection \ref{linearandsimplesubsec}).

\begin{thm}\label{thm3}
Suppose that $\lambda_{M_0}|_K$ is identically zero and that ${\rm UWh}(\pi) = {\rm U'Wh}(\pi)$ for $\pi = \pi_1(M_0)$.   Then any manifold $M_1$ which is stably diffeomorphic to $M_0$ is homotopy equivalent to $M_0$.
\end{thm}

\begin{rem}
In fact more is true.  For example if $M_0$ is simply connected we may conclude that $M_1$ is $h$-cobordant to $M_0$.  We refer to Theorem \ref{linsimthm} for a more general statement.
\end{rem}

\begin{rem}
The intersection form $\lambda_{M_0}|_K$ has a quadratic refinement $\mu$ and if $\mu$ is identically zero then the above theorem follows easily from results in \cite{Kre99} so the novelty lies  in covering the case where $\mu$ is nonzero.
\end{rem}


\subsection{The structure of $\ol{\bZ[\pi]}$} \label{olsubsec}
We start by giving the topological context for our algebraic results and quickly recall the modified surgery setting in which $\ol{\bZ[\pi]}$ arises.  For details we refer the reader to section \ref{sdmsec} and \cite{Kre99}[\S 2].

Let $B = \gamma\co  B \ra BO$ be a fibration where $B$ has the homotopy type of a finite type $CW$-complex and let $\pi = \pi_1(B)$.  We work in the category of $B$-manifolds $(M, \bar \nu)$ which are compact, smooth manifolds $M$ together with an equivalence class of maps $\bar \nu \co  M \ra B$ which factors the stable normal bundle $\nu \co  M \ra BO$ up to homotopy.  A $B$-manifold is called a $(k-1)$-smoothing if $\bar \nu$ is $k$-connected.

We consider the directed graph $G^B_{2q}$ whose vertices,
$V^B_{2q}$, are the set of $B$-diffeomorphism classes of closed
$2q$-dimensional $(q-1)$-smoothings in $B$ and whose edges,
$E^B_{2q}$, are the set of \rel boundary $B$-bordism classes of
$B$-bordisms between closed $(q-1)$-smoothings with the same Euler
characteristic.  An edge in $E^B_{2q}$ is represented by a
$B$-bordism $(W, \bar \nu; M_0, M_1)$ from $(M_0, \bar
\nu|_{M_0})$ to $(M_1, \bar \nu|_{M_1})$ and if such a bordism
exists \cite{Kre99}[Theorem 2] states that $M_0$ and $M_1$ are
stably diffeomorphic.  If $B = B^{q-1}(M_0)$ then the converse
holds by Lemma \ref{SDcondlem}.

We now write $\Lambda$ for any weakly finite\footnote
{That is, the rank of a free \fg $\Lambda$-modules is well-defined.
See also \cite {Coh89}[p.~143f].},
unital ring with involution, for example $\Lambda = \bZ[\pi]$, and let $\epsilon = (-1)^q$.  The graph $G^B_{2q}$ has an algebraic analogue $G^{\Lambda}_{2q}$ whose edge set is the monoid $l_{2q+1}(\Lambda)$.  The vertex set of $G^{\Lambda}_{2q}$ is $\mathcal{F}^{\rm zs}_{2q}(\Lambda)$, the unital abelian monoid of $0$-stabilised $\epsilon$-quadratic forms.  These are equivalence classes of $\epsilon$-quadratic forms  $v = (V, \theta)$, defined on finitely generated, free, based $\Lambda$-modules $V$ (see Definition \ref{formdefi}) where two forms are equivalent if they become isometric after the addition of zero forms on such modules.  We write $[v] \in \mathcal{F}^{\rm zs}_{2q}(\Lambda)$ for the $0$-stabilised form defined by $v$.  A $(q-1)$-smoothing $(M, \bar \nu)$ defines a zero stabilised form $[v(\bar \nu)]$ (see Example \ref{0-stabformdefi}).

The elements of $\ol{\Lambda}$ are algebraic models of bordisms $(W, \bar \nu, M_0, M_1)$.  They are defined as equivalence classes $[x]$ of {\bf quasi-formations} which are triples
\[ x = (H,\psi; L,V)\]
consisting of a quadratic form $(H, \psi)$ 
together with a simple Lagrangian $L$ (see
Definition \ref{formdefi}) and some other half-rank, based direct
summand $V \subset H$.  For the present we omit the precise
details of the equivalence relation on quasi-formations which
defines $\ol{\Lambda}$ but refer the reader to subsection
\ref{olquasidef}.  Addition in $\ol{\Lambda}$ is the operation
induced by the direct sum of quasi-formations.  The
quasi-formation $x$ defines induced quadratic forms $v$ and
$v^{\perp}$ on $V$ and its annihilator, $V^{\perp}$.  It turns out
that we obtain a monoid map
\[ b\co  \ol{\Lambda} \ra \mathcal{F}^{\rm zs}_{2q}(\Lambda) \times \mathcal{F}^{\rm zs}_{2q}(\Lambda),~~~
[x] \mapsto ([v], [-v^{\perp}])
 \]
and we view $[x]$ as an algebraic bordism from $[v]$ to $[-v^{\perp}]$.

A $(2q+1)$-dimensional bordism $W = (W, \bar \nu; M_0, M_1)$ defines an element $\Theta(W, \bar \nu) \in \ol{\bZ[\pi}$ (see the proof of Lemma \ref{algtopstrictcanlem}) such that $b(\Theta(W, \bar \nu)) = ([v(\bar \nu_0)], [v(\bar \nu_1)])$.  We wish to know when $W$ is bordant \rel boundary to an $s$-cobordism and $\Theta(W, \bar \nu)$ tells us: elements of $\ol{\bZ[\pi]}$ which are represented by a quasi-formation $(H, \psi; L, V)$ for which $H = L \oplus V$ are called elementary and $\Theta(W, \bar \nu)$ is elementary if and only if $W$ is bordant \rel boundary to an s-cobordism.

The elementary elements of $\ol{\Lambda}$ play the role of
algebraic bordism classes of $s$-cobordisms and form a submonoid
$\mathcal{E}\ol{\Lambda}$.  Writing $b_{\mathcal{E}}$ for
$b|_{\mathcal{E}\ol{\Lambda}}$ it is easy to see that
$b_{\mathcal{E}}([x]) = ([v], [v])$ lies on the diagonal, $\Delta(
\mathcal{F}_{2q}^{\rm zs}(\Lambda))$, for every $[x] \in
\mathcal{E}\ol{\Lambda}$ and we prove

\begin{thm*}[Corollary \ref{bcor} (ii)]
For each $0$-stabilised form $[v]$ there is a unique elementary
element, denoted $e([v])$, with $b_{\mathcal{E}}(e([v])) = ([v],
[v])$.  There are thus monoid isomorphisms
\[  \mathcal{E}\ol{\Lambda} \dlra{b_{\mathcal{E}}} \Delta( \mathcal{F}_{2q}^{\rm zs}(\Lambda)) \dlra{\cong} \mathcal{F}_{2q}^{\rm zs}(\Lambda).\]
\end{thm*}

Given quadratic forms $v$ and $v'$ of equal rank we define
\[l_{2q+1}(v, v') := b^{-1}([v], [v']) \subset \ol{\Lambda}\]
to be the set of edges between fixed vertices in
$G^{\Lambda}_{2q}$.  The algebraic analogue of the fact that edges
of $G^B_{2q}$ occur only between stably diffeomorphic manifolds is
that $l_{2q+1}(v, v')$ is empty unless $[v] \oplus [H_\epsilon(\Lambda^k)] = [v']
\oplus [H_\epsilon(\Lambda^k)]$ for some $k$.  In this case $[v]$ and $[v']$ are
called stably equivalent and we write $[v] \sim [v']$ and $v \sim
v'$.  To begin describing $l_{2q+1}(v, v')$, let $0$ be the zero
form.  Kreck \cite{Kre99}[\S 6] proved that $L_{2q+1}(\Lambda) :=
l_{2q+1}(0, 0)$ is the group of units of $\ol{\Lambda}$ and that
$L_{2q+1}^s(\Lambda)$ can be identified with a subgroup of
$L_{2q+1}(\Lambda)$ (see Remarks \ref{Lunitrem} and
\ref{Ltorsionrem} for more details).

As a subgroup of the units $L_{2q+1}^s(\Lambda)$ acts on $\ol{\Lambda}$ and one easily sees that this action restricts to $l_{2q+1}(v, v')$ and that the orbits of this action are appropriate equivalence classes of embeddings of $v = (V, \theta) \hra (H, \psi) = H_\epsilon(\Lambda^r)$.  The central idea of this paper is to use the theory of algebraic surgery to define a complete invariant of these embeddings.

Algebraic surgery allows us to treat an embedding $j\co  v = (V, \theta) \hra (H, \psi) = H_\epsilon(\Lambda^k)$ like an embedding of a co-dimension zero manifold with boundary into a closed manifold.  Specifically, the quadratic forms $v$ and $v^{\perp}$ have algebraic boundaries $\del v$ and $\del v^{\perp}$ which are generalisations of boundary quadratic linking forms.   The embedding $j$ defines an isomorphism $f_j\co  \del v \cong -\del v'$ such that $H_\epsilon(\Lambda^k) \cong v \cup_{f_j} - v'$ where we have glued $v$ to $-v'$ along $f_j$, a procedure defined in algebraic surgery.  Indeed for any $f \in {\rm Iso}(\del v, \del v')$, the set of isomorphisms from $\del v$ to $\del v'$, we may construct the nonsingular form $\kappa(f) := v \cup_f -v'$ and so obtain an embedding of $v$ into $\kappa(f)$.  Defining
\[{\rm bIso}(v, v') := {\rm Iso}(\del v, \del v')/({\rm Aut}(v) \times {\rm Aut}(v')),\]
where ${\rm Aut}(v)$ and ${\rm Aut}(v')$ are the groups of
isometries of $v$ and $v'$ which act respectively by pre and post
composition with the induced isometry of the boundary, one shows
that two embeddings $j_0, j_1\co  v \ra h$ are equivalent if and only
if $[f_{j_0}] = [f_{j_1}] \in {\rm bIso}(v, v')$ (see Proposition
\ref{orthoglueprp}).

For quadratic forms $v \sim v'$, we define a `$0$-stabilised boundary isomorphism set' ${\rm sbIso}(v, v')$ (see Definition \ref{sbIsodefi}) and a map
\[
\delta \co  l_{2q+1}(v, v') \ra {\rm sbIso}(v, v'), \quad [H,\psi;L,V] \mapsto [f_j]
\]
where $f_j \co \partial v \stackrel{}{\rightarrow}  \partial v'$ is induced by $j\co v =  (V,\theta) \hra (H,\psi)$.  Not every form $\kappa(f)$ above is hyperbolic and indeed there is a further map
\[
\kappa\co  {\rm sbIso}(v, v') \ra L_{2q}^s(\Lambda),\quad [f]\mt [v\cup_f-v']
\]
where $L_{2q}^s(\Lambda)$ is the usual even dimensional Wall group.  The maps $\kappa$ and $\delta$ and the action $\rho$ of $L_{2q+1}^s(\Lambda)$ on $l_{2q+1}(v, v')$ are related in our main theorem.

\begin{thm*}[Theorem \ref{generalmainthm}]
Let $v$ and $v'$ be $\epsilon$-quadratic forms with $v \sim v'$.  There is an ``exact'' sequence of sets
\[
\oLs{\Lambda}\stackrel{\rho}{\lra} l_{2q+1}(v, v') \stackrel{\delta}{\lra} \sbIso(v, v') \stackrel{{\kappa }}{\lra} L^s_{2q}(\Lambda)
\]
by which we mean that the orbits of $\rho$ are the fibres of $\delta$ and $\im(\delta)={\kappa}^{-1}(0)$.
\end{thm*}

In the case where $v = v'$ the set ${\rm sbIso}(v, v) =: {\rm sbAut}(v)$ is the set of stable boundary automorphisms and contains $1$, the equivalence class of the identity.

\begin{thm*}[Corollary \ref{maincor}]
For every $\epsilon$-quadratic form $v$ there is an exact sequence
\[
\oLs{\Lambda}\stackrel{\rho}{\lra} \ols{v, v} \stackrel{\delta}{\lra} \sbAut(v) \stackrel{{\kappa}}{\lra} L^s_{2q}(\Lambda)
\]
where the orbits of the action $\rho$ are precisely the fibres of $\delta$ and $\im(\delta)={\kappa}^{-1}(0)$.  Moreover $\delta([x]) = 1 \in \sbAut(v)$ if and only if $[x]$ is elementary modulo the action of $L_{2q+1}^s(\Lambda)$.
\end{thm*}

To discuss our main theorems we define
\[ l_{2q+1}(v) := \bigcup_{v' \sim v} l_{2q+1}(v, v'),\]
the set of edges in $G_{2q}^{\Lambda}$ leaving a given vertex $[v]$.  Consider the problem of determining whether $[x] \in l_{2q+1}(v)$ is elementary:  the theorems above reveals three obstacles.  Firstly we must have $b([x]) = ([v], [v])$.  Secondly, if $[x] \in l_{2q+1}(v, v)$ we need $\delta([x]) = 1 \in {\rm sbAut}(v)$.  Finally the transitive action of $L_{2q+1}^s(\Lambda)$ on $\delta^{-1}(1)$ must be taken into account.  Up until now the role of $b([x])$ and the action of $L_{2q+1}^s(\Lambda)$ have been understood and so our main achievement is to identify the role of the set $\sbAut(v)$ and more generally ${\rm sbIso}(v, v')$.  We point out that these sets were already in the literature for $l_{1}(\bZ)$:  on the algebraic side in \cite{Nik79} and on the topological in explicitly in \cite{Boy87} and implicitly in \cite{Vog82}.

To apply Corollary \ref{maincor} we wish to calculate the set ${\rm sbAut}(v)$.  If $v$ becomes nonsingular in some localisation of $\Lambda$ then $\del v$ is a quadratic linking form and ${\rm sbAut}(v)$ is readily identified and often calculable (see Proposition \ref{bIsosprop} for a general statement).  A simple but instructive example of this is the following: if $\Lambda = \bZ$ and $\epsilon = +1$, then for the quadratic form $v = (\bZ, n)$ where $n = p_1 \dots p_k$ is a product of distinct odd primes, then ${\sbAut}([v]) \cong (\bZ/2)^{k-1}$ (see Example \ref{sbautex}).  More generally, we prove

\begin{thm*}[Proposition \ref{l_1(Z)finprop}]
For each $+$-quadratic form $v$ over $\bZ$ the set $l_1(v)$ is finite but there are $v$ for which $\{ [v'] \,| \, [v'] \sim [v] \}$ or ${\rm sbAut}(v)$ is arbitrarily large.
\end{thm*}

We also show that ${\rm sbAut}(v)$ is small if $v$ is the sum or a linear and a simple form and if the torsion hypothesis of Thereom \ref{thm3} holds.   Here $v = (V, \theta)$ is linear if $\theta + \theta^* = 0$ and simple if $\theta + \theta^* \co  V \cong V^*$ is a simple isomorphism.

\begin{thm*}[Proposition \ref{simplelinearprop}]
If $v = (N, \eta) \oplus (M, \psi)$ is the sum of a linear form $(N, \eta)$ and simple form $(M, \psi)$ and if ${\rm UWh}(\Lambda) = {\rm U'Wh}(\Lambda)$ then $L_{2q+1}(\Lambda)$ acts transitively on $l_{2q+1}(v, v)$.
\end{thm*}


\subsection{Algebraic cancellation} \label{algcansubsec}

Given a $0$-stabilised form $[v]$ it is customary to say that cancellation holds for $[v]$ if $[v'] \sim [v]$ entails that $[v] = [v']$.  Generalising this, we say that {\bf strict cancellation holds for $[v]$} if $l_{2q+1}(v) = \{ e([v]) \}$.  The topological significance of strict cancellation is primarily the following: if $[v(\bar \nu)]$ is the $0$-stabilised form of a $(q-1)$-smoothing $(M, \bar \nu)$ in $B^{q-1}(M)$ and if strict cancellation holds for $[v(\bar \nu)]$ then every manifold stably diffeomorphic to $M$ is $s$-cobordant to $M$.  We show that strict cancellation holds in a variety of algebraic circumstances.  In order of increasing complexity the result are as follows.

\begin{thm*}[Corollary \ref{fieldcor}]
Let $\Lambda$ be a field of a characteristic different from $2$ or let
$\Lambda=\bZ/2\bZ$.  Then all elements of $\ol{\Lambda}$ are elementary.
\end{thm*}

\begin{thm*}[Proposition \ref{leml_3Zelem}]
Every element of $l_{3}(\bZ)$ is elementary.
\end{thm*}

Now let $\rk (V)$ denote the rank of the free abelian group $V$, let $G$ be a finite abelian group, let $l(G)$ denote the minimal number of generators of $G$, and $l_p(G) = l(G_p)$ where $G_p$ is the $p$-primary component of $G$, $p$ a prime.   The first two parts of the next theorem are translations of results from \cite{Nik79}.

\begin{thm*}[Proposition \ref{cancellationforVZ+1prop}]
Let $v=(V, \theta)$ be a nondegenerate quadratic form and let $(G,\phi)$ be the associated symmetric boundary (Definition \ref{Sboundarydef}). Then strict cancellation holds for $v$ if
any of the following conditions hold.

\begin{enumerate}
\item
The symmetric form $(V, \theta + \theta^*)$ is indefinite and satisfies

\begin{enumerate}
\item
$\rk(V) \geq l_p(G) +2$ for all primes $p \neq 2$,

\item
if $\rk(V) = l_2(G)$ then the symmetric boundary associated to
$\left( \bZ^2, \svec{0 & 2 \\ 0 & 0} \right)$
is a summand of the $2$-primary component of $(G, \phi)$.
\end{enumerate}

\item
The symmetric form $(V, \theta + \theta^*)$ is isomorphic to one of the classical lattices $E_8$, $E_7$, $E_6$, $D_5$ or $A_4$.
\item
The quadratic form $v$ is isomorphic to $(\bZ, p)$ for any prime $p$.

\end{enumerate}
\end{thm*}

We next consider group rings $\bZ[\pi]$ for polycyclic-by-finite groups $\pi$. Recall $h'(\pi,q)$ from Theorem \ref{thm1}.
\begin{thm*}[Corollary \ref{polycyccor}]
Let $\Lambda = \bZ[\pi]$ be the group ring of a polycyclic-by-finite group $\pi$ and let $[v]$ be a $0$-stabilised form.  If $[v] = [w] \oplus [H_\epsilon(\Lambda^k)]$ for $k \geq h'(\pi,q)$, then strict cancellation holds for $[v]$.
\end{thm*}

\begin{rem}
It is very likely that the bound $h(\pi) +3$ is not optimal for all infinite polycyclic-by-finite groups.  As we noted before, Kahn \cite{Kh04} has recently obtained cancellation results for topological $4$-dimensional manifolds with certain infinite fundamental groups.   When translated to the context of $l_{5}(\bZ[\pi])$ Khan's results should give strict cancellation for the group rings he considers when $[v]$ splits off $h(\pi) + 2$ hyperbolic planes.
\end{rem}

Finally, we remark that we know of no example where strict cancellation does not hold where $[v] = [w] \oplus [H_\epsilon(\Lambda)]$ splits off a single hyperbolic plane.

\subsection{The Grothendieck group of $\ol{\Lambda}$} \label{grothgpsubsec}
Our aim in this paper has not been to compute the monoid
$\ol{\Lambda}$ but to understand the subsets $l_{2q+1}(v, v')$.
However, a key stabilization property of $\ol{\Lambda}$ allows us
to compute its Grothendieck group.

Recall that Wall's original definition of $L_{2q+1}^s(\Lambda)$ was by isometries of hyperbolic forms and that any isometry $\alpha\co  (H, \psi) \cong (H, \psi)$ defines the element
\[ [z(\alpha)] := [(H, \psi; L, \alpha(L)] \in L_{2q+1}^s(\Lambda).\]

\begin{thm*}[Lemma \ref{L-actionlem} and Corollary \ref{L-actioncor}]
Let $[x] \in \ol{\Lambda})$.  If $\alpha\co  (H, \psi) \cong (H, \psi)$ restricts to an isometry of $V$ in some representative $(H, \psi; L, V)$ of $[x]$ then $[z(\alpha)]$ acts trivially on $[x]$.  In particular, if $(V, \theta) \cong (W, \sigma) \oplus H_\epsilon(\Lambda^k)$ then any $[z] = [H_\epsilon(\Lambda^k), J, K] \in L_{2q+1}^s(\Lambda)$ acts trivially on $[x]$.
\end{thm*}

Combining this result and the key Proposition \ref{bdryextprp} we prove

\begin{thm*}[Proposition \ref{stablyelemprop}]
For any $[x] \in \ol{\Lambda}$, there exists a natural number $k$ such that $[x] + e([H_\epsilon(\Lambda^k)])$ is elementary.
\end{thm*}
\noindent Now, for an abelian monoid $A$ let ${\rm Gr}(A)$ denote the Grothendieck group of $A$.  It is a simple matter to obtain the following

\begin{thm*}[Corollary \ref{grothgrpcor}]
The sequence $\mathcal{F}_{2q}^{\rm zs}(\Lambda) \cong \mathcal{E}\ol{\Lambda} \hra \ol{\Lambda}$ induces isomorphisms of Grothendieck groups
\[ {\rm Gr}(\mathcal{F}_{2q}^{\rm zs}(\Lambda)) \cong {\rm Gr}(\mathcal{E}\ol{\Lambda}) \cong {\rm Gr}(\ol{\Lambda}).\]
\end{thm*}


The remainder of the paper is organised as follows: in section 2 we review the application of the surgery obstruction monoids $\ol{\bZ[\pi]}$ to the cancellation problem and prove our topological results.  In section 3 we begin the algebra: we recall the \cite{Kre99} definition of $\ol{\Lambda}$ and give an equivalent but slightly more flexible definition closer to the spirit of algebraic surgery.  In section 4 we elaborate some results from algebraic surgery on the gluing and splitting of $\epsilon$-quadratic forms.  In section 5 we apply the ideas from section 4 to prove our main theorem.  In section 6 we apply the main theorem in certain situations: we consider $l_{2q+1}(v,v)$ when $v$ is the sum of a simple and a linear forms, we show how the $\bAut$-sets and $\sbIso$-sets can often be calculated using automorphisms of linking forms and we calculate $\ol{\bZ}$.  In section 7 we consider strict cancellation when $\Lambda$ has finite asymptotic stable rank: e.g. $\Lambda = \bZ[\pi]$ and $\pi$ is a polycyclic by finite group.

{\bf Acknowledgements:}  We would like to thank Matthias Kreck and Andrew Ranicki for their long term support.  We would also like to thank Jim Davis, Ian Hambleton and Qayum Khan for helpful discussions.

\section{Stably diffeomorphic manifolds}
\label{sdmsec}
The cancellation problem and the study of stable diffeomorphisms and stably diffeomorphic manifolds has a rich history which we shall not attempt to summarise here but the highlights include \cite{Wa64, CS71, Te92, HaKr88, HaKr93b}.  Instead we shall focus on the modified surgery setting of \cite{Kre99} where the monoids $\ol{\bZ[\pi]}$ play a key role.   In this section we give the proofs of all of the topological results from the introduction assuming the algebraic results about $\ol{\bZ[\pi]}$ to which the remaining sections of the paper are devoted.  We also prove the useful Lemma \ref{SDcondlem} which extends the scope of results in \cite{Kre99}.  


\subsection{The proofs of topological applications}
\label{sdmneq37subsec}

We work in the category of compact, smooth manifolds but appropriate translations of our statements continue to hold for any category, dimension and fundamental group where one can do surgery.  Let $BO$ be the classifying space for stable real vector bundles.  We let $B$ denote $\gamma\co  B \ra BO$, a fibration whose domain is a $CW$-complex of finite type with fundamental group $\pi_1(B)  = \pi$.  

A compact manifold $M$ embedded in a high dimensional Euclidean space has a stable normal bundle classified by a map $\nu\co  M \ra BO$.   A {\bf $B$-manifold} $(M, \bar \nu)$ is a manifold $M$ together with an appropriate equivalence class of lift of its the stable normal bundle through $\gamma$:
\[
\nu \co M \dlra{\bar \nu} B \dlra{\gamma} BO.
\]
There are notions of $B$-diffeomorphism, $B$-bordism and
$B$-bordism groups of closed, $n$-dimensional $B$-manifolds
$\Omega_{n}(B)$.  For all positive integers $k$, a $B$-manifold
$(M, \bar \nu)$ is called a {\bf $k$-smoothing} if $\bar \nu$ is
$(k+1)$-connected and for every manifold $M$ the $k$-th Postnikov
factorisation of $\nu\co  M \ra BO$ defines a fibration $B^{k}(M) \ra
BO$, the {\bf normal $k$-type} of $M$.  The fibre homotopy type of
$B^{k}(M)$ is a diffeomorphism invariant of $M$ and there are
always $k$-smoothings $\bar \nu \co  M \ra B^{k}(M)$.

In this subsection we consider $2q$-dimensional $B$-manifolds for
$q \neq 3, 7$.  We also assume for technical reasons that $\pi_{q+1}(B)$ is a finitely generated $\bZ[\pi]$-module.  The following definition, adapted slightly from
\cite{Kre99}[Theorem 4], identifies the bordisms which arise in
modified surgery at the $(q-1)$-type and which are relevant to the
cancellation problem.

\begin{defi}  Let $B \ra BO$ be a fibration.  A $(2q+1)$-dimensional modified surgery problem over $B$ $(W, \bar \nu \co M_0, M_1, f)$ consists of the following data:
\begin{enumerate}
\item $(M_0, \bar \nu_0)$ and $(M_1, \bar \nu_1)$, two compact, connected $2q$-dimensional $(q-1)$-smoothings in $B$ with the same Euler characteristic,
\item  a diffeomorphism $f \co  \del M_0\isora \del M_1$ compatible with $\bar \nu_0$ and $\bar \nu_1$,
\item   a compact $(2q+1)$-dimensional $B$-manifold $(W, \bar \nu)$ with boundary $\del W = M_0 \cup_f M_1$ such that $\bar \nu |_{M_i} = \bar \nu_i$.
\end{enumerate}
Sometimes we shall write simply $(W, \bar \nu)$ for $(W, \bar \nu; M_0, M_1, f)$.
\end{defi}

The relevance of  modified surgery problems for the cancellation problem is made clear in the following lemma.

\begin{lem} \label{SDcondlem}
Let $M_0$ and $M_1$ be compact, connected $2q$-dimensional manifolds of equal Euler characteristic and let $\bar \nu_0\co  M \ra B=B^{q-1}(M_0)$ be a $(q-1)$-smoothing of $M_0$ in its $(q-1)$-type.  Then the following are equivalent.
\begin{enumerate}
\item
$M_0$ and $M_1$ are stably diffeomorphic.
\item $M_1$ admits a $(q-1)$-smoothing $\bar \nu_1\co  M_1 \ra B$ such that there is a modified surgery problem $(W, \bar \nu; M_0, M_1, f)$ with $\bar \nu|_{M_i} = \bar \nu_i$, $i = 0,1$.
\end{enumerate}
Moreover, in the case that $(i)$ and $(ii)$ hold then each stable diffeomorphism $h\co M_0 \sharp_k(S^q \times S^q) \cong M_1 \sharp_k(S^q \times S^q)$ defines element $\Theta(h, \bar \nu_0) \in l_{2q+1}(\bZ[\pi])$.
\end{lem}

\begin{proof}
If $(W, \bar \nu; M_0, M_1, f)$ exists as in $(ii)$ then \cite{Kre99}[Theorem 2] states in part that $M_0$ and $M_1$ are stably diffeomorphic.  In the other direction, given a diffeomorphism $h\co  M_0 \sharp_k(S^q \times S^q) \cong M_1 \sharp_k (S^q \times S^q)$ we can build a bordism from $M_0$ to zero $M_1$ as follows.  Let $i = 0$ or $1$ and let $W_i \cong (M_i \times [i,i + 1] \natural_k (S^q \times D^{q+1}))$ be the trace of $k$ trivial surgeries on trivially embedded $(q-1)$-spheres in the interior of $M_i$  where the boundary connected sums all take place in $M_i \times \{ 1 \}$.  The boundary of $W_i$ consists of codimension $0$ submanifolds: $\del W^+_i := M_i\times\{2i\}$, $\del W^c_i := \del M_i \times [i, i+1]$ and $\del W^-_i := M_i \sharp_k(S^q \times S^q)\times\{1\}$.  We form $W_h := W_0 \cup_h W_1$ by gluing $\del W^-_0$ to $\del W^-_1$ along $h$.  By construction $\del W$ has a decomposition by codimension $0$ submanifolds $\del W \cong M_0 \cup \del W^c_0 \cup_{\del h} \del W^c_1 \cup M_1$.

We must now put a $B$-structure, $\bar \nu\co  W \ra B$, on $W$ such
that $\bar \nu|_{M_0} = \bar \nu_0$ and this is easy on $W_0$:
just extend $\bar \nu_0$ trivially over the trace of the trivial
surgeries to obtain $B$-manifolds $(W_0, \bar \nu_2)$ and $(\del
W_0^-, \bar \nu_2^-)$ the restriction of $(W_0, \bar \nu_2)$ to
$\del W_0^-$.  To extend the $B$-structure $(W_0, \bar \nu_2)$ to
all of $W$ we use $h$ to transport $(\del W_0^-, \bar \nu_2^-)$ to
$\del W_1^-$ and obtain the $B$-structure $(\del W_1^-, \bar
\nu_2^- \circ h)$ which we must now extend to all of $W_1$.  This
is a homotopy lifting problem for the fibration $B \ra BO$ and the
pair $(W_1, \del W_1^-)$.  Since $W_1$ is the trace of $k$
$q$-surgeries on $\del W_1^-$, up to homotopy $W_1 \simeq \del
W_1^- \cup (\cup_k e^{q+1})$ and we obtain an single obstruction
to extending the lift $\bar \nu_2^- \circ h$ to all of $W_1$ which
lies in $H^{q+1}(W_1, \del W_1^-; \pi_q(F))$ where $F$ is the
fibre of $B \ra BO$.  But by the definition of $B = B^{q-1}(M)$,
$\pi_q(F) = 0$, the obstruction vanishes and there is a unique,
$B$ structure $(W, \bar \nu)$ extending $(W_0, \bar \nu_2)$ to all
of $W$.  We define $\bar \nu_1 \co  M_1 \ra B$ to be the restriction
of $\bar \nu$ to $M_1 \subset \del W \subset W$.

We must show that $(M_1, \bar \nu_1)$ is a $(q-1)$-smoothing in $B$.  By construction $(W, \bar \nu)$ is a $(q-1)$-smoothing and $W$ is homotopy equivalent to $(M_1 \vee_k S^q) \cup (\cup_k e^{q+1})$ where $\bar \nu$ is homotopically trivial when restricted to $\vee_k S^q$ and so $\bar \nu_1 \co  M_1 \ra B$ is indeed a $q$-equivalence.

To finish the proof, we define $\Theta(h, \bar \nu_0) = \Theta(W, \bar \nu) \in \ol{\bZ[\pi_1(B)]}$.
\end{proof}

The following fundamental theorem of Kreck identifies the key role of the monoids $\ol{\bZ[\pi]}$ for the cancellation problem.

\begin{thm} [\cite{Kre99}{[Theorem 4]}] \label{kreckmainthm}
Let $(W, \bar \nu; M_0, M_1, f)$ be a  $(2q+1)$-dimensional modified surgery problem.  Then there is a well-defined surgery obstruction $\Theta(W, \bar \nu) \in l_{2q+1}(\bZ[\pi])$ depending only on the \rel boundary $B$-bordism class of $(W, \bar \nu; M_0, M_1, f)$ and a submonoid $\mathcal{E}\ol{\bZ[\pi]}$ such that $\Theta(W, \bar \nu) \in \mathcal{E}\ol{\bZ[\pi]} $ if and only if $(W, \bar \nu;   M_0, M_1, f)$ is bordant \rel boundary to an $s$-cobordism.
\end{thm}
\noindent

It is customary to say that cancellation holds for $M$ if every
manifold stably diffeomorphic to $M$ is diffeomorphic to $M$.
In the light of Lemma \ref{SDcondlem} we make the following

\begin{defi}
Let $(M, \bar \nu)$ be a $(q-1)$-smoothing in $B$.  We say that strict cancellation holds for $(M, \bar \nu)$ if every  modified surgery problem over $B$ is bordant \rel boundary to an $s$-cobordism.  We say that strict cancellation holds for $M$ if it holds for every $(q-1)$-smoothing of $(M, \bar \nu)$ in $B^{q-1}(M)$.
\end{defi}

We next identify a key invariant of a $(q-1)$-smoothing $(M, \bar
\nu)$ in $B$ which will allow us to  show that strict cancellation
holds for many classes of manifolds.  As is explained in
\cite{Kre99}[\S 5] subtle but by now standard surgery techniques
define a quadratic form $\mu$ on ${\rm Ker}(\bar
\nu\co  \pi_q(M) \ra \pi_q(B))$.  As $\bar \nu$ is $q$-connected and $\pi_{q+1}(B)$ is assumed to be finitely generated, it follows that ${\rm Ker}(\bar \nu\co  \pi_q(M) \ra \pi_q(B))$ is a finitely presented $\bZ[\pi]$-module.  Now let $j \co  V \ra {\rm Ker}(H_q(M) \ra H_q(B))$ be a surjection with $V$ a finitely generated free based $\bZ[\pi]$-module and let $\theta = j^*\mu$ be the induced quadratic form on $V$.  From the proof of \cite{Kre99}[Proposition 8 (ii)] we deduce that there is a well-defined $0$-stabilised form $[V, \theta]$.

\begin{defi} \label{0-stabformdefi}
Let $(M,  \bar \nu)$ be a $(q-1)$-smoothing in $B$.  The {\bf
$0$-stabilised quadratic form of $(M, \bar \nu)$} is the
$0$-stabilised form $[v(\bar \nu)] := [V, \theta]$ defined above.
\end{defi}



\begin{rem} \label{[v(nu)]rem}
We observe that if $(M, \bar \nu)$ is a $(q-1)$-smoothing in $B^{q-1}(M)$ then the module ${\rm Ker}(\bar \nu \co  H_q(M) \ra H_q(B))$ is independent of $\bar \nu$.  This is due to the uniqueness of Postnikov decompositions \cite{Bau77}[Corollary 5.3.8]: a point observed in \cite{Kre85}.  It follows that the choice of $\bar \nu$ can only effect the sign of $[v(\bar \nu)]$.  Moreover, if strict cancellation holds for any $0$-stabilised form $[v]$ then it holds for $[-v]$ since the automorphism $T \co  l_{2q+1}(\Lambda) \cong l_{2q+1}(\Lambda)$ of Remark \ref{minusrem} gives a bijection from $l_{2q+1}(v)$ to $l_{2q+1}(-v)$.
\end{rem}

The following lemma relates strict algebraic cancellation and strict topological cancellation.

\begin{lem} \label{algtopstrictcanlem}
Let $(M, \bar \nu)$ be a $(q-1)$-smoothing of $M$ in $B^{q-1}(M)$.  If strict cancellation holds for $[v(\bar \nu)]$, then strict cancellation holds for $M$.
\end{lem}

\begin{proof}
We recall some further facts from the \cite{Kre99} analysis of  $(2q+1)$-dimensional modified surgery problems $(W, \bar \nu; M_0, M_1, f)$ over a general $B$.  After surgery below the middle dimension on the interior of $(W, \bar \nu)$ we may assume that $\bar \nu$ is a $q$-equivalence.
Let $U$ be the union of $k$ disjoint embeddings $S^q\times D^{q+1} \hra W$ representing a set of generators for $\im (d\co \pi_{q+1}(B,W) \ras \pi_q(W))$ and let $L=H_q(U)$.  The surgery obstruction $\theta(W, \bar \nu) \in l_{2q+1}(\bZ[\pi])$ is represented by the quasi-formation $(H_\epsilon(L); L, V)$ where
\begin{eqnarray*}
V=H_{q+1}(W - \oover{U},\del U\cup M_0) \lra H_q(\del U) = H_\epsilon(L).
\end{eqnarray*}
Moreover, the induced form $(V, \theta)$ and the induced form on
the annihilator of $V$, $(V^{\perp}, \theta^{\perp})$, are related
to the zero-stable forms of $(M_0, \bar \nu_0)$ and $(M_1, \bar
\nu_1)$ via
\[[V, \theta] = [v(\bar \nu_0)] \text{~~and~~}[V^{\perp}, -\theta^{\perp}] = [v(\bar \nu_1)].\]
Assume now that $M_0 = M$ and that $B = B^{q-1}(M)$.  We have that the surgery obstruction $\Theta(W, \bar \nu )$, is represented by a quasi-formation $(H_\epsilon(L); L, V)$ where $[V, \theta] = [v(\bar \nu_0)] = \pm [v(\bar \nu)]$ and so satisfies strict algebraic cancellation by Remark \ref{[v(nu)]rem}.  By definition, this means that $\Theta(W, \bar \nu)$ is elementary and by Theorem \ref{kreckmainthm} we conclude that $(W, \bar \nu; M_0, M_1, f)$ is bordant \rel boundary to an $s$-cobordism.
\end{proof}


We now prove our main topological results. 
Recall $h'(\pi,q)$ from Theorem \ref{thm1}. 
\begin{cor}\label{strcancor}
Let $M$ be a compact, connected, smooth $2q$-dimensional manifold with polycyclic-by-finite fundamental group $\pi$.  Assume that either of the following hold:
\begin{enumerate}
\item  $M \cong N \sharp_k(S^q \times S^q)$ where $k \geq h'(\pi, q)$,
\item $q$ is even, $M$ is simply connected and admits a $(q-1)$-smoothing $\bar \mu \co  M \ra B^{q-1}(M)$ such that $[v(\bar \nu)]$ satisfies any of the conditions of Proposition  \ref{cancellationforVZ+1prop}.
\end{enumerate}
Then strict cancellation holds for $M$.
\end{cor}

\begin{proof}
In the first case, let $c\co  M \ra N$ be the collapse map induced by a decomposition $M \cong N \sharp_k(S^q \times S^q)$  and let $\bar \nu \co  N \ra B$ be a $(q-1)$-smoothing in $B=B^{q-1}(N)$.  One checks that $\bar \nu \circ c \co  M \ra B^{q-1}(N)$ is a $(q-1)$-smoothing and that $B$ is also the normal $(q-1)$-type for $M$.  It follows that $[v(\bar \nu \circ c)]$ splits off $H_\epsilon(\bZ[\pi]^k)$ (see Definition \ref{formmonoiddef}) and so by Corollary \ref{polycyccor} or Proposition \ref{cancellationforVZ+1prop} strict algebraic cancellation holds for $[v(\bar \nu \circ c)]$.  Both cases now follow from Lemma \ref{algtopstrictcanlem}.
\end{proof}

Theorem \ref{thm1} now follows from Lemma \ref{SDcondlem} and Corollary \ref{strcancor}.

\begin{rem}
We note that Corollary \ref{strcancor} (i) shows that the bordisms which fall under the assumptions of \cite{Kre99}[Theorem 5] are already bordant \rel boundary to an $s$-cobordism.
\end{rem}


We next turn to the proof of Theorem \ref{thm2} and the
representation of bordism classes by mapping tori.  Recall that
every fibration $B$ defines bordism groups $\Omega_n(B)$ of closed
$B$-manifolds up to $B$-bordism.

\begin{thm} \label{maptorthm}
Suppose that strict cancellation holds for $(M_0, \bar \nu_0)$, a
closed, $2q$-dimensional, $(q-1)$-smoothing in $B$.  Then every
element of $\Omega_{2q+1}(B)$ is represented by some $B$-structure on the mapping torus
of some $B$-diffeomorphism of $(M_0, \bar \nu_0)$.
\end{thm}

\begin{proof}
Let $(W, \bar \nu) = (M_0 \times [0, 1], \bar \nu_0 \times {\rm
Id})$ be the trivial $s$-cobordism and let $(Y, \phi)$ be a closed
$(2q+1)$-dimensional $B$-manifold representing $[Y, \phi] \in
\Omega_{2q+1}(B)$.  Then the disjoint union $(W \sqcup Y, \bar
\nu_W \sqcup \bar \nu_Y)$ is a  modified surgery
problem and so by assumption there is a  $B$-bordism \rel boundary
$(X, \bar \nu_X)$ from $(W \sqcup Y, \bar \nu_Y \sqcup \phi)$ to
an $s$-cobordism $(Z, \bar \nu_Z; M_0, M_0)$.  This $s$-cobordism
defines, up to pseudo isotopy, a $B$-diffeomorphism $g\co  M_0 \cong
M_0$.  Let $T_g$ be the mapping torus of $g$.  By definition, $(X,
\bar \mu_X)$ yields a $B$-bordism from $(Y, \bar \nu_Y)$ to $(T_g,
\bar \nu_T)$, where $\bar \nu_T$ is the $B$-structure induced on
$T_g$ by $(X, \bar \nu_X)$.  Hence $[T_g, \bar \nu_T] = [Y, \bar
\nu_Y] \in \Omega_{2q+1}(B)$.
\end{proof}

Theorem \ref{thm2} now follows by combining Theorem \ref{maptorthm} and Corollary \ref{strcancor} (i).

Finally we move from cancellation up to diffeomorphism to
cancellation up to homotopy.   We first require some preliminary
remarks: the group of units of $\ol{\bZ[\pi}$,
$L_{2q+1}(\bZ[\pi])$, is an extension of the classical Wall group
$L_{2q+1}^s(\bZ[\pi])$ by a subgroup of the Whitehead group of
$\pi$.  Assuming that $q \geq 3$ or that $q = 2$, $\pi$ is good
and we are in the topological category, the group
$L_{2q+1}(\bZ[\pi])$ acts on $s$-cobordism classes of manifolds
homotopy equivalent to a fixed $2q$-dimensional manifold $M$.  To prove this one only observes
that it is no harder to realise a general formation by a bordism
than a simple formation.  Recall also that a form $n$ is linear if
its symmetrisation is zero and that a form $w$ is nonsingular if
its symmetrisation is a simple isomorphism.

\begin{thm} \label{linsimthm}
Let $(W, \bar \nu \co M_0, M_1, f)$ be a  modified surgery problem between $(q-1)$-smoothings $(M_0, \bar \nu_0)$ and $(M_1, \bar \nu_1)$ such that $[v(\bar \nu_0)] = [v(\bar \nu_1)] = [n+w]$ where $n$ is linear and $w$ is simple.  Assume further that  ${\rm UWh}(\pi) = {\rm U'Wh}(\pi)$ for $\pi = \pi_1(M_0)$ as in Theorem \ref{thm3}.  Then for some $[z] \in L_{2q+1}(\bZ[\pi])$, $\Theta(W, \bar \nu) + [z]$ is elementary.  In particular, $M_0$ is homotopy equivalent to $M_1$.
\end{thm}

\begin{proof}
By assumption $b(\Theta(W, \bar \nu)) \in l_{2q+1}(n + w, n + w)$ and so we apply Proposition \ref{simplelinearprop} to obtain $[z] \in L_{2q+1}(\bZ[\pi])$ such that $[z] + \Theta(W, \bar \nu)$ is elementary.  Realising $[z]$ by a bordism $(W', \bar \nu'; M_1, M_2)$ and forming $(W'', \bar \nu'') = (W, \bar \nu) \cup (W', \bar \nu')$ we have that $\Theta(W'', \bar \nu'')$ is elementary and thus is bordant to an $s$-cobordism between $M_0$ and $M_2$.  But by construction, $M_2$ is homotopic to $M_1$.
\end{proof}

Now let $M_0$ satisfy the hypotheses of Theorem \ref{thm3}.  It follows for any $(q-1)$-smoothing $\bar \nu_0 \co  M_0 \ra B^{q-1}(M_0)$ that $[v(\bar \nu)]$ is linear.  By Lemma \ref{SDcondlem} if $M_1$ is stably diffeomorphic to $M_0$, then there is a  surgery problem $(W, \bar \nu; M_0, M_1, f)$.  Now $[v(\bar \nu_0)] = [-v(\bar \nu_0)^{\perp}] = [v(\bar \nu_1)]$: the first equality holds since $[v(\bar \nu_0)]$ is linear and the second by definition.  Now Theorem \ref{thm3} follows from Theorem \ref{linsimthm}.

\begin{rem}\label{6and14rem}
In dimension $6$ and $14$ the surgery obstruction $\Theta(W, \bar \nu)$ of Theorem \ref{kreckmainthm} sometimes lies in the slightly altered monoid $\widetilde l_{2q+1}(\bZ[\pi])$ \cite{Kre99}[\S 6].  We are confidant that appropriate analogues of all our results for $\ol{\bZ[\pi]}$ continue to hold for $\widetilde l_{2q+1}(\bZ[\pi])$ and that the same is true for our topological results.  In particular,  we expect that Theorems \ref{thm1}, \ref{thm2} and \ref{thm3} also hold when $q = 3$ or $7$ but leave the details to the reader.
\end{rem}

\section{Forms, quasi-formations and $\ol{\Lambda}$}
\label{form&qformsec}

Let $q$ be a positive integer, $\epsilon=(-1)^q$ and $\Lambda$ a weakly finite
unital ring with an involution $x\mt\bar x$.  Important examples are the group rings $\bZ[\pi]$
with involution
$
\sum_{g\in\pi} x_g g \mt \sum_{g\in\pi} w(g) \overline{x_g} g^{-1}
$
where $w\colon\pi\ras \bZ/2\bZ \cong \{\pm 1\}$ is a homomorphism.

\subsection{Based modules and simple isomorphisms}

The following definition reminds the reader of some basic concepts from
the theory of Whitehead torsion which can be found, for example, in \cite{Mil66}[\S1-4].

\begin{defi}
\label{baseddef}
\begin{enumerate}

\item
The {\bf $*$-operation} on the {\bf reduced $K$-group $\widetilde{K}_1(\Lambda) = \coker(K_1(\bZ)\ras K_1(\Lambda))$}
is the isomorphism:
\[
    *\co  \widetilde{K}_1(\Lambda) \isora \widetilde{K}_1(\Lambda),\quad [f] \mt [f^*]
\]
A subgroup $Z\subset \widetilde{K}_1(\Lambda)$ is {\bf $*$-invariant} if $Z^*\subset Z$.

\item
Let $M$ be a {\bf stably \fg free left module $M$ over $\Lambda$} \ie a \fg left module $M$
for which $n,m\in\bN_0$ exist such that $M\oplus\Lambda^m\cong\Lambda^n$.
An {\bf $s$-basis of $M$} is a basis $\{b_1,\dots,b_{r+m}\}$ of some \fg free left module $M\oplus \Lambda^m$.

Let $Z\subset \widetilde{K}_1(\Lambda)$ be a $*$-invariant subgroup. Two $s$-bases
$\{b_1,\dots,b_{r+m}\} \subset M\oplus \Lambda^m$ and
$\{b'_1,\dots,b'_{r+m'}\} \subset M\oplus \Lambda^{m'}$ are {\bf $Z$-equivalent} if there is a
$k\geq\max(m,m')$ such that the transformation matrix in regard to the bases
$\{b_1,\dots,b_{r+m}, e_{m+1},\dots, e_k\}$ and
$\{b_1',\dots,b'_{r+m'}, e_{m'+1},\dots, e_k\}$
represents an element in $Z$. (Here $\{e_1,\dots, e_k\}$ denotes the standard basis of $\Lambda^k$).

\item
A {\bf $Z$-based module $(M,\mathcal{B})$} is a stably \fg free left module $M$ over $\Lambda$
together with a $Z$-equivalence class of $s$-bases $\mathcal{B}=[b_1,\dots,b_n]$. Any representative
of $\mathcal{B}$ is called a {\bf preferred $s$-basis}.



\item
Let $(M,\mathcal{B})$ and $(N,\mathcal{C})$ be two $Z$-based modules. Then
$(M,\mathcal{B}) \leq (N,\mathcal{C})$ if there is a $Z$-based module $(P,\mathcal{D})$
such that $(M,\mathcal{B}) \oplus (P,\mathcal{D}) = (N,\mathcal{C})$.

\item
Let $(M,\mathcal{B})$ and $(M',\mathcal{B'})$ be $Z$-based modules and $f\co M\isora M'$
an isomorphism of
the underlying $\Lambda$-modules. Let $A \in M_{n}(\Lambda)$ be the matrix of $f\oplus\id$ in respect to
some $s$-bases  of cardinality $n$ representing $\mathcal{B}$ and $\mathcal{B'}$.
The {\bf torsion of $f$} is given by  $\tau(f)=[A] \in \widetilde{K}_1(\Lambda)$.
The isomorphism $f$ is called {\bf $Z$-simple} if $\tau(f) \in Z$.


\end{enumerate}
\end{defi}

\begin{rem} \label{torsionrem}
In the following we will no longer mention the $s$-bases
    explicitly in the notation of based modules and we will
    fix $Z$, defining $\Wh(\Lambda) =
    \widetilde{K}_1(\Lambda)/Z$.  For group rings
    $\Lambda=\bZ[\pi]$ we shall use $Z=\{\pm g|g\in\pi\}$ and
    so $\Wh(\Lambda) = \Wh(\pi)$ is the usual Whitehead group.

\end{rem}

\subsection{Forms}
\label{formsubsec}
In this subsection we recall definitions for $\epsilon$-quadratic forms and the even dimension $L$-groups as well as introducing the notions ``zero stable forms''.

\begin{defi} \label{formdefi}
Let $M$ be a based module.
\begin{enumerate}
\item

The {\bf $\epsilon$-duality involution map}
\begin{eqnarray*}
T_\epsilon \co \Hom_\Lambda(M,M^*) \ra\Hom_\Lambda(M,M^*),\quad
\phi \mt (x \mapsto (y\mapsto \epsilon \overline{\phi(y)(x)}))
\end{eqnarray*}
leads to the abelian groups
$Q^\epsilon(M)={\rm Ker}(1-T_\epsilon)$ and
$Q_\epsilon(M)=\coker(1-T_\epsilon)$.

\item
The {\bf hyperquadratic groups} $\hQ^{-\epsilon}(M)$ are defined via the exact sequence
\begin{equation*}
0 \lra \hQ^{-\epsilon}(M) \lra Q_\epsilon(M) \stackrel{1+T_\epsilon}
{\lra} Q^\epsilon(M) \lra \hQ^\epsilon(M) \lra 0.
\end{equation*}

\item An {\bf asymmetric form $(M,\rho)$} is a pair with $\rho \in \Hom_\Lambda(M,M^*)$.

\item
An {\bf $\epsilon$-symmetric form $(M,\phi)$} is a pair with $\phi \in Q^\epsilon(M)$.
It is called {\bf nondegenerate} if $\phi\co M \ras M^*$ is injective, {\bf nonsingular} if
$\phi$ is an isomorphism and {\bf simple} if $\phi$ is a simple isomorphism.

\item
An {\bf $\epsilon$-quadratic form} $(M,\psi)$ is a pair with $\psi \in Q_\epsilon(M)$.
Its {\bf symmetrisation} is the $\epsilon$-symmetric form $(M, (1 + T_\epsilon)\psi)$.
The form $(M, \psi)$ is {\bf nondegenerate}, {\bf nonsingular} or {\bf simple} if its
symmetrisation has this property. It is {\bf linear} if its symmetrisation is the zero form.

\item
An $\epsilon$-symmetric form $(M,\phi)$ is {\bf even} if there is a $\psi\in Q_\epsilon(M)$
such that $(1+T_\epsilon)\psi=\phi$. A choice of $\psi$ is a {\bf quadratic refinement} of $\phi$.

\item
The {\bf annihilator of a submodule $j \co L \hookrightarrow M$} of an $\epsilon$-quadratic
form $(M, \psi)$ is the (unbased) submodule $L^\perp:={\rm Ker}(j^*(1+T_\epsilon)\psi\co M \ras L^*)$.
The {\bf radical} of $(M,\psi)$ is the (unbased) submodule $\Rad(M,\psi) := M^\perp$.

\item
A {\bf Lagrangian $L$ of an $\epsilon$-quadratic form $(M,\psi)$} is a submodule
$L\subset M$ that is both a direct summand and a based module such that
$L=L^\perp$ (as unbased modules) and $j^*\psi j=0 \in Q_\epsilon(L)$.

\item
A {\bf Hamiltonian $s$-basis} of a nonsingular $\epsilon$-quadratic form $(M,\psi)$ induced by a Lagrangian $L$ is an $s$-basis of $M$ such that
\[
0 \lra L \lra M \stackrel{j^*\phi}{\lra} L^*\lra 0
\]
is a based exact sequence where $\phi=(1+T_\epsilon)\psi$.

\item
A {\bf simple Lagrangian $L$ of a simple $\epsilon$-quadratic form $(M,\psi)$} is a Lagrangian such that any Hamiltonian $s$-basis defined by $L$ is a preferred $s$-basis of $M$.

\item
An {\bf isometry $f\co(M,\psi) \isora (M',\psi')$ of $\epsilon$-quadratic forms} is an
isomorphism $f\co M\isora M'$ such that $f^*\psi' f=\psi\in Q_\epsilon(M)$.  Unless stated otherwise, we shall assume that all isometries $f$ are simple and we denote the
{\bf group of simple self-isometries of $(M,\psi)$} by {\bf $\Aut(M,\psi)$}.

\end{enumerate}
\end{defi}

\begin{rem}
The hyperquadratic groups have exponent $2$ and satisfy the equality $\hQ^\epsilon(M\oplus N) = \hQ^\epsilon(M)\oplus \hQ^\epsilon(N)$ for two based modules $M$ and $N$.
\end{rem}

\begin{defi}
\label{formmonoiddef}
\begin{enumerate}
\item
For any based module $L$ the {\bf hyperbolic $\epsilon$-quadratic}
and {\bf  $\epsilon$-symmetric forms} are defined as follows
\begin{eqnarray*}
H_\epsilon(L)=\left(L \oplus L^*, \svec{0&{\rm Id} \\0&0}\right)
,\quad
H^\epsilon(L)=\left(L \oplus L^*, \svec{0&{\rm Id}\\\epsilon {\rm Id}&0}\right).
\end{eqnarray*}
If the rank of $L$ is one, these forms are call {\bf hyperbolic planes}.

\item Two $\epsilon$-quadratic forms $(M,\psi)$ and $(M',\psi')$ are {\bf stably isometric}
if there is a hyperbolic $H_\epsilon(L)$ such that $(M,\psi)\oplus H_\epsilon(L) \cong (M',\psi')\oplus H_\epsilon(L)$.

\item
The stable isometry classes of nonsingular $\epsilon$-quadratic forms form a group under direct sum with $-[(M, \psi)] = [(M, -\psi)]$.  This group is denoted $\eL{\Lambda}$ with $L_{2q}^s(\Lambda)$ the subgroup of classes represented by  simple forms.

\item A {\bf $0$-stabilised} $\epsilon$-quadratic form is an
    equivalence class of forms where two forms $(M,\psi)$ and
    $(M',\psi')$ are considered equivalent if there are
    modules $P$ and $Q$ and an isometry $(M,\psi)\oplus (P,0)
    \cong (M',\psi')\oplus (Q, 0)$.  The $0$-stabilised form
    defined by $(V, \theta)$ is denoted $[V, \theta]$.

\item
We let $\mathcal{F}^{\rm zs}_{2q}(\Lambda)$ denote the abelian {\bf monoid of $0$-stabilised $\epsilon$-quadratic forms} with addition induced by the direct sum of $\epsilon$-quadratic forms and unit $0 := [P, 0]$ for any module $P$.


\end{enumerate}
\end{defi}

\begin{rem}
\label{splitformrem}
\begin{enumerate}
\item
An $\epsilon$-quadratic form $(M,\psi)$ defines an $\epsilon$-quadratic form $(M,\phi, \nu)$ in the classical sense where $\phi$ is the
symmetrisation of $\psi$ and $\nu\co M\ras Q_{\epsilon}(\Lambda)$ is the quadratic refinement given by $\nu(x) := \psi(x,x)$.
Conversely, every $\epsilon$-quadratic form $(M,\phi, \nu)$ in the
classical sense gives rise to an $\epsilon$-quadratic form $(M,\psi)$ (\cite{Ran02}[\S11]).

\item
The group $L_{2q}^s(\bZ[\pi])$ is Wall's surgery obstruction group.
\end{enumerate}
\end{rem}


\subsection{The original definition of $\ol{\Lambda}$}
\label{origolsec}
We first recall Wall's original definition of the odd-dimensional simple
$L$-groups (\cite{Wal99}[\S 6]).  Let $SU_k(\Lambda,\epsilon)=\Aut(H_\epsilon(\Lambda^k))$.
Let $TU_k(\Lambda,\epsilon)$ be the subgroup of those isometries preserving the
Lagrangian $\Lambda^k\times \{0\}$ and inducing a simple automorphism on it.
Finally, we define $RU_k(\Lambda,\epsilon)$ to be the subgroup generated by
$TU_k(\Lambda)$ and the flip map $\sigma_k:=\svec{0&1\\\epsilon& 0}\oplus\id_{H_\epsilon(\Lambda^{k-1})}$.
The quotient of the limit groups
$SU(\Lambda,\epsilon)=\lim_{k\ras\infty} SU_k(\Lambda,\epsilon)$ and
$RU(\Lambda,\epsilon)=\lim_{k\ras\infty} RU_k(\Lambda,\epsilon)$
is the abelian group
$
L^s_{2q+1}(\Lambda):=SU(\Lambda,\epsilon)/RU(\Lambda,\epsilon)
$.

In \cite{Kre99} $\ol{\Lambda}$ is defined as the set
of equivalence classes of pairs $(H_\epsilon(\Lambda^k),V)$ for
$k\in\bN$ where $V\subset \Lambda^{2k}$ is a based free direct summand of rank $k$.
The equivalence relation is given by stabilisation with trivial pairs
$(H_\epsilon(\Lambda^k), \Lambda^k\times\{0\})$ and an action of $RU(\Lambda,\epsilon)$ so that
two pairs $(H_\epsilon(\Lambda^k),V)$ and $(H_\epsilon (\Lambda^l),V')$ are
equivalent if there is a $\tau \in RU_n(\Lambda,\epsilon)$ such that
$\tau(V\oplus (\Lambda^{n-k}\times\{0\}))=V'\oplus (\Lambda^{n-l}\times\{0\})$.  We shall write $[H_\epsilon(\Lambda^k), V] \in l_{2q+1}(\Lambda)$ for the equivalence class represented by $(H_\epsilon(\Lambda^k), V)$.  The orthogonal sum of pairs induces an abelian monoid structure on $\ol{\Lambda}$ with group of units $\oL{\Lambda}$, the submonoid of equivalence classes of pairs where the form induced on $V$ is zero.

A pair $(H_\epsilon(\Lambda^k), V)$ is called {\bf elementary} if $V\oplus(\{0\}\times\Lambda^k)=\Lambda^{2k}$ (as based modules).  An element of $\ol{\Lambda}$ is called elementary if it has an elementary representative.  The elementary elements of $\ol{\Lambda}$ form a submonoid which we denote by $\mathcal{E}l_{2q+1}(\Lambda)$.

\subsection{A new definition of $\ol{\Lambda}$ via quasi-formations}
\label{olquasidef}
In this subsection we update Kreck's definition of $\ol{\Lambda}$ in a manner similar to Ranicki's
reformulation of Wall's original definition of the odd-dimensional $L$-groups by
{\bf formations}.  A formation is a triple $(M, \psi; F, G)$ consisting of
a simple form $(M, \psi)$ together with an ordered pair of simple Lagrangians $F$ and $G$.
A formation of the form $(H_\epsilon(F); F, F^*)$ is called {\bf trivial} and
a stable isomorphism of formations is an isometry of such triples after possible addition
of  trivial formations. There are two different ways of deriving $L_{2q+1}^s(\Lambda)$ from the set of stable
isomorphism classes of
formations (\cite{Ran80a}[\S5], \cite{Ran01a}[Remark 9.15]):
\begin{enumerate}
\item \label{eq1}
stabilisation by boundaries of forms,

\item \label{eq2}
introduction of the additional equivalence relation:
\begin{eqnarray*}
(M,\psi; F,G) \oplus (M,\psi; G,H) \sim (M,\psi; F,H).
\end{eqnarray*}
\end{enumerate}
We now give an alternative description of $\ol{\Lambda}$ in terms of generalised formations which we shall call {\it quasi-formations}. For $\ol{\Lambda}$ one has to be careful about the equivalence relation for quasi-formations because the extensions of the two possibilities above are not the same (Remark \ref{lcobrem}). We note that there is a related earlier approach of defining $\ol{\Lambda}$
via quasi-formations in the unpublished preprint \cite{Kre85}.

\begin{defi}
\label{quasiformdef}
\begin{enumerate}

\item
An {\bf \eq\ $(M,\psi; L, V)$} is a simple $\epsilon$-quadratic form
$(M,\psi)$ together with a simple Lagrangian $L$ and a based half rank
direct summand $V$.

\item
An \eq\ $(M,\psi; L, V)$ is an {\bf \eqf} if $V$ is a 
Lagrangian.\footnote{Strictly speaking, these formations should be called nonsingular following \cite{Ran73}.}
If in addition $V$ is a simple Lagrangian the formation is called
{\bf simple}.

\item
An {\bf isomorphism $f\co(M,\psi; L, V) \isora (M',\psi';L', V')$ of \eqs}
is an isometry $f\co(M,\psi) \isora (M',\psi')$ such that $f(L)=L'$, $f(V)=V'$ and such that the
induced isomorphisms $L\isora L'$ and $V\isora V'$ are simple.

\item
A {\bf trivial formation} is a \eqf\
$(P,P^*) := (H_\epsilon (P); P,P^*)$ for some based module $P$.

\item
The {\bf boundary of an asymmetric form $(K,\rho)$} is the
\eq\ $\delta(K,\rho)=(H_\epsilon(K);K, \svec{1\\\rho}K)$. An \eq\ is {\bf elementary}
if it is isomorphic to a boundary.



\item
Two \eqs\ are {\bf stably  isomorphic} if they are isomorphic after the addition of trivial formations.

\end{enumerate}
\end{defi}

\begin{defi}
\begin{enumerate}
\item
Let $l^{\rm new}_{2q+1}(\Lambda)$ be the unital abelian monoid of stable isomorphism classes of \eqs\ modulo the relation
\begin{eqnarray}
\label{shifteqn}
(M,\psi; K, L) \oplus (M,\psi; L, V) \sim (M,\psi; K,V)
\end{eqnarray}
where $K$ and $L$ both are simple Lagrangians.  The unit $0
\in l^{\rm new}_{2q+1}(\Lambda)$ is the equivalence class of
all trivial formations.  For an \eq\ $x = (M, \psi; L, V)$,
we shall write $[x] = [M, \psi; L, V] \in l_{2q+1}^{\rm
new}(\Lambda)$ for the element represented by $x$.

\item An element in  $l^{\rm new}_{2q+1}
(\Lambda)$ is called {\bf elementary} if it is represented by a boundary.  The elementary elements
form a submonoid $\mathcal{E}l^{\rm new}_{2q+1}(\Lambda)$.

\item
Let $L_{2q+1}^{\rm new}(\Lambda)\subset l_{2q+1}^{\rm new}(\Lambda)$ be the abelian group of all classes
represented by $\epsilon$-quadratic formations and let $L_{2q+1}^{s,\rm new}(\Lambda)$ be the subgroup of all
classes represented by simple $\epsilon$-quadratic formations.

\end{enumerate}
\end{defi}

\begin{rem}
\label{eqrem}
\begin{enumerate}
\item
Any \eq\ is isomorphic to an \eq\ of the type $(H_\epsilon(F); F, V)$.




\item Let $(M,\psi; L,V)$ be an \eq\ then there is a unique equivalence class of $s$-basis for $V^\perp$ such that the short exact sequence
\[
0 \lra V^\perp \lra M \stackrel{j^*\phi}{\lra} V^* \lra 0
\]
is based. Here $\phi=(1+T_\epsilon)\psi$ and $j\co V\hookrightarrow M$ is the inclusion.

\end{enumerate}
\end{rem}

\begin{lem}
\label{elemeqlem}
An element $x\in l^{\rm new}_{2q+1}(\Lambda)$ is elementary if and only if it has a representative
$(M, \psi; L, V)$ such that $M = L \oplus V$ (as based modules).
\end{lem}
\begin{proof}
\fat
For any \eq\ $(M, \psi; L, V)$ it is obvious that
\begin{eqnarray*}
[M, \psi; L^*, V] = [(M,\psi;L^*,L)\oplus(M, \psi; L, V)] = [M, \psi; L, V] \in l^{\rm new}_{2q+1}(\Lambda)
\end{eqnarray*}
Hence, we need to show that $x\in l^{\rm new}_{2q+1}(\Lambda)$ is elementary if and only if
$x=[M, \psi; L, V]$ such that $M = L^* \oplus V$.
Any elementary element of $l^{\rm new}_{2q+1}(\Lambda)$ is represented by a boundary $\delta(K,\rho)$
which clearly fulfills the above condition.

On the other hand let $(M, \psi; L, V)$ be an \eq\ such that $M = L^* \oplus V$.
\Wlog  we assume that $(M,\psi)=H_\epsilon(L)$ and that $L$ is free.
Let $\{b_1,\dots,b_n\}\subset L$ and $\{v_1,\dots,v_n\}\subset V$ be some
preferred bases. Then the basis transformation matrix $\svec{1 & X \\ 0 & Y}$ in respect to the bases
$\{b_1,\dots,b_n, v_1,  \dots,v_n\}$ and $\{b_1,\dots,b_n, b_1^*,\dots,b_n^*\}$ represents an element in $Z$.
Hence the component $y$ of
the inclusion $\svec{y\\x}\co V \hookrightarrow L\oplus L^*$ must be a simple isomorphism.
The isometry $\svec{y^{-1} & 0 \\ 0& y^{*}}$
of $H_\epsilon(L)$ induces an isomorphism between $(M, \psi; L, V)$ and a boundary.
\end{proof}

\begin{prop}
\label{newlprp}
There is an isomorphism of monoids
\begin{eqnarray*}
\eta\co \ol{\Lambda}   &\lra&  l^{\rm new}_{2q+1}(\Lambda)
\\
{[H=H_\epsilon(\Lambda^k),V]} &\mt& [H; \Lambda^k\times{0}, V]
\end{eqnarray*}
with
$\eta(\mathcal{E}l_{2q+1}(\Lambda)) = \mathcal{E}l^{\rm new}_{2q+1}(\Lambda)$,
$\eta(L^{}_{2q+1}(\Lambda)) = L^{\rm new}_{2q+1}(\Lambda)$
and $\eta(L^{s}_{2q+1}(\Lambda)) = L^{s,\rm new}_{2q+1}(\Lambda)$.  In particular  $L_{2q+1}^{new}(\Lambda)$ is the group of units of $l_{2q+1}^{new}(\Lambda)$.
\end{prop}

\begin{proof}
We first show that $\eta$ is a well-defined map \ie it is invariant under the equivalence relations used
to define $\ol{\Lambda}$.
Let $H=H_\epsilon(\Lambda^k)$ and $[H,V]\in\ol{\Lambda}$. Obviously, an isometry $\tau\in TU_k(\Lambda,\epsilon)$
induces an isomorphism between $(H; \Lambda^k\times\{0\}, V)$ and $(H; \Lambda^k\times\{0\}, \tau(V))$.
Now let $\sigma_k\in RU_k(\Lambda,\epsilon)$ be the flip map mentioned in \S\ref{origolsec}.
Let $x=(H;\sigma_k(\Lambda^k\times\{0\}), \Lambda^k \times \{0\})$.
By relation (\ref{shifteqn}), $x\oplus (H; \Lambda^k \times \{0\}, V)$
is equivalent to $(H;\sigma_k(\Lambda^k\times\{0\}), V)$
which in turn is isomorphic to
$(H;\Lambda^k\times\{0\}, \sigma_k(V))$.
Because
\[
x=(H_\epsilon(\Lambda), \{0\}\times\Lambda, \Lambda\times\{0\})\oplus
(H_\epsilon(\Lambda^{k-1}), \Lambda^{k-1}\times\{0\}, \Lambda^{k-1}\times\{0\})
\]
it represents zero in $l^{\rm new}_{2q+1}(\Lambda)$ which proves
that $[H; \Lambda^k \times \{0\}, V] = [H;\Lambda^k\times\{0\}, \sigma_k(V)] \in
l^{\rm new}_{2q+1}(\Lambda)$.

It is clear that $\eta$ is a monoid map so we complete the proof by constructing
an inverse homomorphism $\nu\co l^{\rm new}_{2q+1}(\Lambda)\ra \ol{\Lambda}$. Let
$(M,\psi;L,V)$ be an \eq\ such that $M$, $V$ and $L$ are free. Choose an isometry
$\alpha\co (M,\psi) \isora H = H_\epsilon(\Lambda^k)$ such that $\alpha(L) = \Lambda^k\times\{0\}$ and
such that the isomorphism $L\isora \Lambda^k\times\{0\}$ induced by $\alpha$ is simple.
We would like to define $\nu([M,\psi;L,V]) = [H, \alpha(V)]\in\ol{\Lambda}$ but in order to do so
we must show that this definition is independent of the various equivalence relations.
Firstly, a different choice of
$\alpha$ changes $(H,V)$ only by an action of an element in $TU_k(\Lambda,\epsilon)$.
Secondly, two  isomorphic quasi-formations are mapped to two pairs differing again by an
element of $TU_k(\Lambda,\epsilon)$.
Thirdly, trivial quasi-formations are mapped to trivial pairs.

At last, we have to show that $\nu$ is invariant under the relation (\ref{shifteqn}).
Let $(M,\psi; K, L)$ be a simple \eqf\ and $(M,\psi; L, V)$ an \eq.
Let $\alpha,\alpha'\co (M,\psi) \isora H:=H_\epsilon(\Lambda^k)$ be
two isometries such that $\alpha(K)=\alpha'(L)=\Lambda^k\times\{0\}$ and such that the
induced isomorphisms between the respective Lagrangians are simple.
Then $\phi:=\alpha{\alpha'}^{-1}\in SU_k(\Lambda, \epsilon)$.  Let
$W:=\alpha'(V)$.
By definition $\nu$ maps $(M,\psi;K,L)\oplus(M,\psi; L,V)$ to
$(H,\phi(\Lambda^k\times\{0\}))\oplus(H,W)$ and $(M,\psi; K, V)$ to $
(H, \phi(W))$.
Now observe that $(H,\phi(\Lambda^k\times\{0\}))\oplus(H, W)$
and $(H,\phi(W))\oplus(H,\Lambda^k\times\{0\})$ only differ
by an isometry $\tau=\svec{0&\phi\\\phi^{-1}&0}$ of $H\oplus H$.
Moreover, $\tau\in RU(\Lambda,\epsilon)$ because $\tau = (\sigma_1\oplus\cdots\oplus\sigma_1)
\circ (\phi\oplus\phi^{-1})$ vanishes in $L^s_{2q+1}(\Lambda)$.
This shows that $(M, \psi; K, V)$ and $(M, \psi; K, L) \oplus (M, \psi; L, V)$
are mapped to equivalent pairs.
It is clear that $\nu$ is additive and that $\eta$ and $\nu$ are inverse to each other.
Moreover, $\nu$ respects elementariness by Lemma \ref{elemeqlem}.
\end{proof}


\begin{rem} \label{Lunitrem}
Henceforth we shall identify $l^{\rm new}_{2q+1}(\Lambda)$ and $\ol{\Lambda}$, etc.  
\end{rem}

\begin{rem}\label{minusrem}
It is easy to check that there is a well-defined monoid
automorphism
\[ T \co  l_{2q+1}(\Lambda) \cong l_{2q+1}(\Lambda), ~~~[H, \psi; L, V] \mapsto [H, -\psi; L, V]. \]
\end{rem}

\begin{rem}[c.f. \cite{Kre99}{[p.773]}] \label{Ltorsionrem}
There is an exact sequence
\[
 0\lra L_{2q+1}^{s}(\Lambda) \lra L_{2q+1}^{}(\Lambda)  \stackrel{\tau}{\lra} \Wh(\Lambda)
\]
where $\tau([M,\psi; L,K])$ is the torsion of any isomorphism
$(M,\mathcal{B}) \isora (M,\mathcal{C})$ where $\mathcal{B}$ is
represented by the preferred $s$-bases of $M$ and $\mathcal{C}$ is
represented by the Hamiltonian $s$-bases of $M$ with respect to
$K$.  The group $L_{2q+1}(\Lambda)$ corresponds to case $C$ in
\cite{Wal99}[\S 17D].  As predicted there, the image of $\tau$
lies in the set of anti-self dual torsions.  We discuss the image of $\tau$ further in subsection
\ref{linearandsimplesubsec}.
\end{rem}
\section{Glueing quadratic forms together}
\label{gluesec}

The main theorem of this paper calculates subsets of $\ol{\Lambda}$
using isomorphisms between the boundaries of $\epsilon$-quadratic forms.
This section introduces the notions of boundaries and unions of possibly singular forms.

If $\Lambda=\bZ$ and the cokernel of a form is finite one can
define a linking form on this cokernel which is often described as
the boundary of the form (subsection \ref{linkingsubsec} or
\cite{Ran81}[\S3.4]). In general, the boundary of an
$\epsilon$-quadratic form is a refined version of a formation: a
{\bf split $\epsilon$-quadratic formation}.  If, for two (possibly
singular) $\epsilon$-quadratic forms $(V,\theta)$ and
$(V',\theta')$, there is an isomorphism $f\co \del(V, \theta)
\cong \del(V', -\theta')$ between their boundaries, one can glue
the forms together. The result is a nonsingular
$\epsilon$-quadratic form $(V,\theta)\cup_f (V',\theta')$.

\subsection{Formations and boundaries of forms}

We review some concepts from \cite{Ran81}[p.69ff and p.86ff].
\begin{defi} \label{formationdef}
\begin{enumerate}
\item
A {\bf simple split $\epsilon$-quadratic formation}
$\left(F,\left( \svec{\gamma\\\mu},\theta \right)G\right)$
is a simple \eqf\ $(H_\epsilon(F); F, \svec{\gamma\\\mu}G)$ together with an
element $\theta\in Q_{-\epsilon}(G)$ such that $\gamma^*\mu=\theta-\epsilon\theta^*$
where $\gamma\co G \ra F$ and $\mu\co G \ra F^*$ define the embedding of $G$ in $H_\epsilon(F)$.

\item
For a module $P$, the {\bf trivial split $\epsilon$-quadratic formation on $P$} is defined to be
$(P,P^*):=\left(P,\left(\svec{0\\1},0\right)P^*\right)$.

\item
The {\bf boundary of an $\epsilon$-quadratic form $(K,\psi)$}
is the simple split $(-\epsilon)$-quadratic formation
$\del(K,\psi)=\left(K,\left(\svec{1\\\psi-\epsilon\psi^*},\psi
\right)K\right)$

\item
An {\bf isomorphism of simple split $\epsilon$-quadratic formations}
\[
f=(\alpha,\beta,\nu) \co (F,\left( \svec{\gamma\\\mu},\theta \right)G)
\isora (F',\left( \svec{\gamma'\\\mu'},\theta' \right)G')
\]
is a triple consisting of simple isomorphisms $\alpha\in\Hom_\Lambda(F,F')$,
$\beta\in\Hom_\Lambda(G,G')$ and an element $\nu\in Q_{-\epsilon}(F^*)
$ such that:
\begin{enumerate}
\item $\alpha\gamma+\alpha(\nu-\epsilon\nu^*)^*\mu=\gamma'\beta \in
\Hom_\Lambda(G, F')$,
\item $\alpha^{-*}\mu =\mu'\beta \in \Hom_\Lambda(G, {F'}^*)$,
\item $\theta+\mu^*\nu\mu=\beta^*\theta'\beta \in Q_{-\epsilon}(G)$.
\end{enumerate}

\item
The {\bf boundary of an isometry $h\co(M,\psi)\isora(M',\psi')$ of simple $\epsilon$-quadratic forms}
is the isomorphism $\del h=(h,h,0)\co \del(M,\psi)\isora
\del(M',\psi')$.

\item
The {\bf composition} of two isomorphisms of simple split $\epsilon$-quadratic formations is
$(\alpha',\beta',\nu')\circ(\alpha,\beta,\nu) = (\alpha'\alpha,\beta'\beta, \nu+\alpha^{-1}\nu'\alpha^{-*})$.
The {\bf inverse} of an isomorphism $(\alpha,\beta,\nu)$ is $(\alpha^{-1},\beta^{-1},-\alpha\nu\alpha^*)$.
The {\bf identity} on a split $\epsilon$-quadratic formation $x$ is the isomorphism $(1,1,0)$.

\item
A {\bf homotopy of isomorphisms of simple split $\epsilon$-quadratic
formations
\[\Delta\co(\alpha,\beta,\nu)\simeq (\alpha',\beta',\nu')\co(F,\left
( \svec{\gamma\\\mu},\theta \right)G)\isora
(F',\left( \svec{\gamma'\\\mu'},\theta' \right)G')\]} is a homomorphism $\Delta\in
\Hom_\Lambda(G^*,F')$ such that:
\begin{enumerate}
\item ${\beta'}^{-*}-{\beta}^{-*} = {\mu'}^*\Delta \in \Hom_\Lambda
(G^*,{G'}^*)$,
\item ${\alpha'}-\alpha= \Delta\mu^* \in \Hom(F, {F'})$,
\item $\alpha'\nu'{\alpha'}^*-\alpha\nu\alpha^{*}
= (\epsilon\alpha'\gamma+\Delta\theta)\Delta^* \in Q_{-\epsilon}({F'}
^*)$.
\end{enumerate}

\item
A {\bf stable isomorphism of two simple split $\epsilon$-quadratic formations
$y$ and $z$} is an isomorphism $y \oplus (P,P^*)\cong z \oplus (Q,Q^*)$.

\item
Let $f_i\co x\oplus u_i \isora y\oplus v_i$ ($i=0,1$) be two isomorphisms of
simple split $\epsilon$-quadratic formations where $u_i$ and $v_i$ are
isomorphic to trivial formations. Then $f_0$ and $f_1$ are {\bf stably homotopic}
if there are based modules $P$, $Q$ and $R_i$ as well as isomorphisms
$g_i\co (P,P^*) \isora u_i\oplus (R_i,R_i^*)$ and $h_i\co v_i\oplus (R_i, R_i^*) \isora (Q,Q^*)$
such that there is a homotopy
\begin{eqnarray*}
\widetilde{f}_0 \simeq \widetilde{f}_1 \co x\oplus (P,P^*) \isora y\oplus (Q,Q^*)
\end{eqnarray*}
where $\widetilde{f}_i=(\id_y \oplus h_i) \circ (f_i \oplus \id_{(R_i, R_i^*)}) \circ (\id_x \oplus g_i)$.

\item Let $y$ and $z$ be simple split $\epsilon$-quadratic formations.
We denote the set of stable homotopy classes of stable isomorphisms from
$y$ to $z$ by $\Iso(y, z)$ and remark that $\Aut(y) : = \Iso(y, y)$ forms
a group under composition.
\end{enumerate}
\end{defi}

We will see the importance of homotopies in the next section since the isometry class of an \eqf\ obtained by gluing two forms together with an isomorphism of their boundary formations depends only on the homotopy class of the isomorphism. (Proposition \ref{unionlem}).

\begin{rem}
Both \cite{Ran01a}[\S6] and \cite{Ran80a}[\S3] explain how an \eqf\ gives rise to a {\bf short odd complex}
\ie a chain complex $d\co C_{q+1}\ras C_q$ together with an $\epsilon$-quadratic structure
$\psi \in Q_\epsilon(C)$. A (stable) isomorphism of \eqfs\ corresponds to a chain isomorphism
(chain equivalence) of the associated short odd complexes. Stable homotopies of stable
isomorphisms of \eqfs\ correspond to chain homotopies of the respective chain equivalences.
\end{rem}

The following technical lemmas give a better description of stable isomorphisms between boundary formations and the conditions under which such isomorphisms are homotopic.

\begin{lem}
\label{sthomlem}
Let $(V,\theta)$ and $(V',\theta')$ be $\epsilon$-quadratic forms and
let $\lambda=\theta+\epsilon\theta^*$ and $\lambda'=\theta'+\epsilon{\theta'}
^*$ be the underlying $\epsilon$-symmetric forms.

\begin{enumerate}
    \item Let $P$ and $P'$ be modules and $\alpha\co V\oplus P \isora V'\oplus P'$
    and $\beta\co V\oplus P^*\isora V'\oplus {P'}^*$ simple isomorphisms. Let
    $\nu\in Q_\epsilon({V'}^*\oplus {P'}^*)$.
Then
\begin{eqnarray*}
(\alpha,\beta,\nu) \co\del (V,\theta)\oplus (P,P^*) &\isora&
\del (V',\theta')\oplus (P',{P'}^*)
\end{eqnarray*}
is a stable isomorphism if and only if there are homomorphisms $a$, $a_1$, $a_3$, $b$, $b_1$ and $s$ such that
\begin{eqnarray}
\label{isonot1eq}
\alpha                  &=&\svec{a&a_1\\\epsilon b_1^*\lambda&a_3}  \co V\oplus P \isora V'\oplus P',
\\\nonumber
\beta^{-1}              &=&\svec{b&b_1\\a_1^*\lambda'&a_3^*}    \co V'\oplus {P'}^*\isora V\oplus P^*,
\\\nonumber
\alpha\nu\alpha^*&=&\svec{s&-\epsilon ab_1\\0&-b_1^*\theta b_1 }        \in Q_\epsilon({V'}^*\oplus {P'}^*),
\\\nonumber
1&=&ab + (s^*+\epsilon s)\lambda' \co V' \ra V',
\\\nonumber
a^*\lambda'&=&\lambda b \co V' \ra V^*,
\\\nonumber
\theta'&=&b^*\theta b +{\lambda'}^*s\lambda'\in Q_\epsilon(V').
\nonumber
\end{eqnarray}
\item
Two stable isomorphisms
\begin{eqnarray*}
f=(\alpha,\beta,\nu) \co\del (V,\theta)\oplus (P,P^*) &\isora&
\del (V',\theta')\oplus (P',{P'}^*)\\
\widetilde{f}=(\widetilde{\alpha},\widetilde{\beta},\widetilde{\nu})\co\del (V,\theta)
\oplus (Q,Q^*) &\isora& \del (V',\theta')\oplus (Q',{Q'}^*)
\end{eqnarray*}
are homotopic if and only if there is a $\Delta_1 \in\Hom_\Lambda(V^*, V')$ such that
$\widetilde{a} - a =\Delta_1 \lambda^*$, $\widetilde{b}- b=\Delta_1^*{\lambda'}^*$ and
$(-\epsilon \widetilde{a} +\Delta_1 \theta)\Delta_1^*=\widetilde{s}-s\in Q_\epsilon({V'}^*)$
where we use the notation of (\ref{isonot1eq}).
\end{enumerate}
\end{lem}

\begin{proof}
\begin{enumerate}
\item This follows straight from the definition.
\item \Wlog $P=Q$ and $P'=Q'$. Then use the homotopy $\svec{\Delta &
\widetilde{a}_1-a_1\\\widetilde{b}^*_1-{b}^*_1 & \widetilde{a}_3-a_3}$.\qedhere
\end{enumerate}
\end{proof}

\begin{lem}
\label{splitisolem}
\begin{enumerate}
\item
Given a simple $\epsilon$-quadratic form $(M,\psi)$  there is an isomorphism
$(1,\phi, -\phi^{-*}\psi\phi^{-1})\co \del(M,\psi) \isora (M,M^*)$
where $\phi=\psi+\epsilon\psi^*$.

\item
Let $(\alpha,\beta,\nu)\co (P,P^*) \isora (P,P^*)$ be an isomorphism of
trivial formations.
Then $\Delta=1-\alpha$ is a homotopy to the identity $(1,1,0)$.
\end{enumerate}
\end{lem}

\subsection{The union and splitting of forms}

\begin{defi}[\cite{Ran81}{[p.84ff]}]
\label{uniondef}
Let $(V,\theta)$ and $(V',\theta')$ be two $\epsilon$-quadratic forms and
let $f = (\alpha,\beta,\nu)\co \del(V,\theta)\oplus(P,P^*) \isora
\del (V', -\theta')\oplus(P',{P'}^*)$
be an isomorphism. Using the notation of Lemma \ref{sthomlem} (\ref{isonot1eq}),
we define the {\bf union} of $(V, \theta)$ and $(V', \theta')$ along $f$, denoted
$(V, \theta)\cup_f (V',\theta')$, to be the $\epsilon$-quadratic form
\[
(M, \psi) =\left(V\oplus {V'}^*, \svec{\theta&0\\\epsilon a&-s}\right).
\]

\end{defi}

The following Lemma lists the basic properties of the glueing construction.

\begin{lem}
\label{unionlem}
Let $(M, \psi) = (V,\theta)\cup_f (V',\theta')$ as in Definition \ref{uniondef}
and let $(M, \phi)$ be its symmetrisation.
\begin{enumerate}
\item
There is an exact sequence
$
0\ra V \stackrel{j}{\lra} M \stackrel{{j'}^*\phi}{\lra} {V'}^* \ra 0
$
where
$j=\svec{1\\0}\co(V,\theta)\ras (M,\psi)$ and
$j'=\svec{b\\-\lambda'}\co (V',\theta') \ras (M,\psi)$
are split injections and $\lambda'=\theta'+\epsilon{\theta'}^*$.

\item
The form $(M,\psi)$ is simple.

\item
Let $\widetilde{f}$ be a stable isomorphism $\del(V,
\theta) \cong \del(V',-\theta')$ which is stably homotopic to $f$, then the respective
unions are isometric relative to $(V, \theta)$.

\item
\label{unionlem3}
Let $k\co(V,\theta)\isora (W,\sigma)$ and $k'\co(V',\theta')\isora (W',
\sigma')$ be isometries.
Define the automorphism $\widetilde{f}=(\del k'\oplus \id_{(P', {P'}^*)}) \circ f
\circ (\del k^{-1}\oplus \id_{(P, {P}^*)})$.
Then there is an isomorphism $\svec{k&0\\0&{k'}^{-*}}\co (V,\theta)\cup_f(V',\theta')\cong (W,\sigma)\cup_{\widetilde{f}} (W',
\sigma')$.

\end{enumerate}
\end{lem}

\begin{proof}
\begin{enumerate}
\item
By Lemma \ref{sthomlem}, ${j'}^*\phi = \mat{0&1}$.

\item
Write $\alpha^{-1} = \svec{u & v\\x & y}\co V'\oplus P' \isora V\oplus P$.
Then, using Lemma \ref{sthomlem}, one computes
\[
\phi \circ \svec{b_1 v^* & \epsilon b\\ u^* & -\epsilon\lambda'} = \svec{ 1 & 0 \\ \epsilon (a b_1 v^* - t u^*) & 1}
\]
which shows that $\phi\co M\ras M^*$ is an isomorphism.
In order to show that it is simple we consider three chain maps $h\co C\ras C'$, $g\co D\ras C$
and $f\co E\ras D$ of based chain complexes given by
\begin{eqnarray*}
\xymatrix@C+60pt
{
C'_1=V'\oplus P'
\ar[r]^{d_{C'}=\svec{{\lambda'}^* & 0 \\ 0 & 1}}
&
C'_0={V'}^*\oplus P'
\\
C_1 =V\oplus P
\ar[r]^{d_C=\svec{{\lambda^*}  & 0 \\ 0 & 1}}
\ar[u]^{h_1=\alpha}
&
C_0=V^*\oplus P
\ar[u]_{h_0=\beta^{-*}}
\\
D_1=V\oplus V'
\ar[r]^{d_D=\svec{\lambda^* & 0 \\ a & 1}}
\ar[u]^{g_1=\svec{1 & 0 \\ 0 & 0}}
&
D_0=M^*=V^*\oplus V'
\ar[u]_{g_0=\svec{1 & 0 \\ 0 & 0}}
\\
E_1=V\oplus V'
\ar[r]^{d_E=\svec{1 & b \\ 0 & -\lambda'}}
\ar[u]^{f_1=\epsilon\cdot\id}
&
E_0=M=V\oplus{V'}^*
\ar[u]_{f_0=\phi}
}
\end{eqnarray*}
Obviously $h$ and $f$ are chain isomorphisms with torsions $\tau(h)=\tau(\alpha)-\tau(\beta^{-*})=0$
and $\tau(f)=-\tau(\phi)$ and $g$ is a simple equivalence.

There is a chain homotopy $\Delta \co h\circ g\circ f \simeq k$ given by
\begin{eqnarray*}
k_0 &=& \svec{0 & 1 \\ 0 & 0}\co E_0=V\oplus{V'}^* \ras C'_0={V'}^*\oplus P'
\\
k_1 &=& \svec{0 & -\epsilon \\ 0 & 0}\co                E_1=V\oplus V'          \ras C'_1=V'\oplus P'
\\
\Delta &=& \svec{\epsilon a & -t^* \\ b_1^*\lambda & b_1^*a^*}\co
E_0=V\oplus{V'}^* \ras C'_1=V'\oplus P'.
\end{eqnarray*}
This shows that $\tau(\phi)=-\tau(h\circ g\circ f)=-\tau(k)=0$.

\item
Given $\Delta=\svec{\Delta_1&\Delta_2\\\Delta_3& \Delta_4}\co V^*\oplus P^* \ras V'\oplus {P'}^*$
defining a homotopy
\[
\Delta\co f\simeq \widetilde{f}\co \del(V,\theta)\oplus (P,P^*)\isora
\del(V',-\theta')\oplus (P',{P'}^*)
\]
there is an isometry
\[
\svec{1&-\Delta_1^*\\0&1}\co (V,\theta)\cup_f(V',-\theta') \isora (V,\theta)\cup_{\widetilde{f}}(V',-\theta').\qedhere
\]
\end{enumerate}
\end{proof}

\begin{prop}
\label{bdryextprp}
Let $(V,\theta)$ and $(V',\theta')$ be $\epsilon$-quadratic forms and
let $f\co\del (V,\theta) \oplus (P,P^*) \isora \del (V',\theta')
\oplus (P',{P'}^*)$ be an isomorphism. If $(M,\psi)=(V,\theta)\cup_f
(V',-\theta')$ and $(M',\psi')= H_\epsilon(V')$ then there is an isometry
\[h\co(V,\theta)\oplus(M',\psi')\isora (V', \theta')\oplus (M,\psi)\]
such that $\del h$ is stably homotopic to $f$.
\end{prop}

\begin{proof}
Using the notation of Lemma \ref{sthomlem} we define the isometry
\[
h=-\svec{0&1&0\\1&b&0\\0&-\lambda'&1}\svec{1&0&0\\-a&1&-s^*\\0&0&1}\co
(V,\theta)\oplus (M',\psi')\isora (V',\theta')\oplus (M,\psi)
\]
The isomorphism $h$ is simple because it is the composition of
triangular matrices with only ones in the diagonal and permutation
matrices. By Lemma \ref{splitisolem} (i) there are natural
isomorphisms $g\co\del(M,\psi) \isora (M,{M}^*)$ and $g'\co
\del(M',\psi') \isora (M',{M'}^*)$. Using Lemma \ref{sthomlem} one
proves that $(\del\id_{V'}\oplus g) \circ \del h \circ
(\del\id_{V}\oplus {g'}^{-1})$ is stably homotopic to f.
\end{proof}

The facts contained in the following proposition are mentioned without proof in \cite{Ran81}[p.86] but since an explicit description of the maps occurring in the proposition plays a crucial role in our results we give a detailed proof.

\begin{prop}
\label{orthoglueprp}
Let $(M,\psi)$ be a simple $\epsilon$-quadratic form with $\phi=(1+T_\epsilon)\psi$.
Let $j\co(V,\theta)\hookrightarrow (M,\psi)$ be a split inclusion of $\epsilon$-quadratic forms.
Let $(V^\perp,\theta^\perp)$ be the induced quadratic form.
Then there is a stable isomorphism $f_j$ between the boundaries of $(V,\theta)$ and $(V^\perp,-\theta^\perp)$
and an isometry
\[
r_j\co (M,\psi) \isora (V,\theta) \cup_{f_j} (V^\perp,\theta^\perp).
\]
Moreover, the isomorphism $f_j$ is well-defined up to homotopy and
$f_j$ and $r_j$ are natural with respect to isometries of such pairs of forms.
\end{prop}

\begin{proof}
Let $\lambda=\theta+\epsilon\theta^*$ and $\lambda^\perp=\theta^\perp+
\epsilon(\theta^\perp)^*$.
The short exact sequence
\[
0\lra V \stackrel{j}{\lra} M \stackrel{(j^\perp)^*\phi}{\lra} (V^\perp)^* \lra 0
\]
is based. Let $\sigma\in\Hom_\Lambda((V^\perp)^*, M)$
be any section so that $(j^\perp)^*\phi \sigma=\id_{(V^\perp)^*}$.
The isomorphism
\[
h = \svec{1&0&0\\j^\perp &\sigma&j}
\svec{-\sigma^*\phi^* j & 1 & -\sigma^*\psi^* \sigma \\ 0&0&1 \\ 1&0&0}
\co
V\oplus V^\perp\oplus (V^\perp)^* \isora V^\perp\oplus M
\]
is obviously simple and even an isometry $h\co(V,\theta)\oplus (M',\psi') \isora (V^
\perp,-\theta^\perp)\oplus (M,\psi)$
with $(M',\psi')=H_\epsilon(V^\perp)$. We write $\phi'=\psi'+\epsilon{\psi'}^*$.
Using Lemma \ref{splitisolem} we obtain an isomorphism
\begin{eqnarray*}
f_j&=&
\left((1,1,0)\oplus(1,\phi, -\phi^{-*}\psi\phi^{-1})\right)
\circ\del h\circ
\left((1,1,0)\oplus(1, {\phi'}^{-1}, {\phi'}^{-*}\psi'{\phi'}^{-1})
\right)
\\
&&\co \del(V,\theta)\oplus(M',{M'}^*)\isora \del (V^\perp,-\theta^\perp)\oplus (M,M^*)
\end{eqnarray*}
Then $r_j :=\mat{j & -\epsilon\sigma} \co (V,\theta)\cup_f (V^\perp,\theta^\perp) \isora(M,\psi)$ is an isometry.

Now we analyze the effect of different choices of $\sigma$ and $\psi$ on
$f_j$.
Let $\widetilde{\sigma}$ be another section and $\widetilde{\psi}$ another
representative of
$[\psi]\in Q_\epsilon(M)$. There are homomorphisms $l\in\Hom_\Lambda((V^\perp)^*,
V)$ and $\kappa\in\Hom_\Lambda(M,M^*)$
such that $\widetilde{\sigma} - \sigma= j l$ and $\widetilde{\psi}-\psi=\kappa-\epsilon
\kappa^*$.
We construct an isometry $\widetilde{h}$ and a stable isomorphism $\widetilde{f_j}$
using $\widetilde{\sigma}$ and $\widetilde{\psi}$ as before.
Then there is a homotopy
\begin{eqnarray*}
\Delta &=& \svec
{
-l^*        & \epsilon x            & 0\\
-j^\perp l^*    & \epsilon (j^\perp x+ jl)  & 0
}\co
\del h \simeq \del \widetilde{h}
\\&&
\co  \del\left((V,\theta)
\oplus (M',\psi')\right)\isora
\del\left( (V^\perp,-\theta^\perp)\oplus (M,\psi)\right)
\end{eqnarray*}
where $x=-\widetilde{\sigma}^*\widetilde{\psi}^*\widetilde{\sigma}+\sigma^*\psi^* \sigma$.
It follows that $\widetilde{f_j}\simeq f_j$.

At last, we discuss naturality. Let $g\co(M,\psi) \isora (\bar M, \bar \psi)$ be an isometry and let
$\bar V=g(j(V))$.
Let $\bar\theta$ be the induced quadratic form on $\bar V$, \etc
Choose the section $\bar \sigma= g \sigma g^*$. Construct $\bar h$, $\bar
f$, \etc as before.
Then
\[
(g\oplus g)\circ h \circ (g^{-1}\oplus g^{-1}\oplus g^*)=\bar h.
\]
It is easy to see that
\[
(1,\bar\phi,-\bar\phi^{-*}\bar\psi\bar\phi^{-1})=
(g,g^{-1*},0)\circ (1,\phi, -\phi^{-*}\psi\phi^{-1}) \circ (g^{-1}, g^
{-1},0).
\]
Putting these facts together we have
\begin{eqnarray}
\label{transformfeqn}
\bar f_j&=& (\del g\oplus (g, g^{-1*}, 0))\circ f \circ (\del g^{-1}\oplus (g^{-1}, g^{*},0))\co
\\\nonumber
&&\del(\bar V,\bar \theta)\oplus(\bar M', \bar {M'}^*)
\isora
\del(\bar V^\perp, -\bar \theta^\perp)\oplus (\bar M, \bar M^*). 
\end{eqnarray}
\end{proof}

\begin{ex}
\label{formationex}
Let $z=(M,\psi;K,L)$ be a possibly non-simple $\epsilon$-quadratic formation. We would like to compute $f_j$ associated to the inclusion of the lagrangian $j\co L\hookrightarrow (M,\psi)$.  We can assume that there is a possibly non-simple isomorphism
$g\co L^* \isora L^*$ such that $[g]=-[g^*] \in\Wh(\Lambda)$ and $(M,\psi)= \left(L\oplus L^*, \svec{0&g\\0&0}\right)$.
Therefore $j=\svec{1\\0}$, $j^\perp = \svec{g^{-*}\\0}$, $\sigma=\svec{0\\1}$ and
$h=\svec{-g^* & 1 & 0 \\ 0 & g^{-*} & 0 \\ 0 & 0 & 1}$. We compute
\begin{eqnarray*}
f_j &=&
\left(  \svec{1 & 0 & 0 \\ 0 & 1 & 0 \\ 0 & 0 & 1},
                \svec{1 & 0 & 0 \\ 0 & 0 & g \\ 0 & \epsilon g^* & 0 },
                \svec{0 & 0 & 0 \\ 0 & 0 & 0 \\ -\epsilon g^{-1} & 0 & 0}
                \right)
                \circ (h,h,0) \circ
\left(  \svec{1 & 0 & 0 \\ 0 & 1 & 0 \\ 0 & 0 & 1},
                \svec{1 & 0 & 0 \\ 0 & 0 & \epsilon \\ 0 & 1 & 0 },
                \svec{0 & 0 & 0 \\ 0 & 0 & 0 \\ \epsilon & 0 & 0}
                \right)
\\&=&
\left(  \svec{-g^* & 1 & 0 \\ 0 & g^{-*} & 0 \\ 0 & 0 & 1},
                \svec{-g^* & 1 & 0 \\ 0 & 0      & g \\ 0 & \epsilon & 0},
                \svec{0    & 0 & 0 \\ 0 & 0      & 0 \\ \epsilon g^{-2} & 0 & 0}
                \right)
\circ
\left(  \svec{1 & 0 & 0 \\ 0 & 1 & 0 \\ 0 & 0 & 1},
                \svec{1 & 0 & 0 \\ 0 & 0 & \epsilon \\ 0 & 1 & 0 },
                \svec{0 & 0 & 0 \\ 0 & 0 & 0 \\ -\epsilon & 0 & 0}
                \right)
\\&=&
\left(  \svec{-g^* & 1 & 0 \\ 0 & g^{-*} & 0 \\ 0 & 0 & 1},
                \svec{-g^* & 0 & \epsilon \\ 0 & g & 0 \\ 0 & 0 & 1},
                \svec{0 & 0 & 0 \\ 0 & 0 & 0 \\ \epsilon(g^{-2}-1) & 0 & 0}
                \right)
\end{eqnarray*}
By Lemma \ref{sthomlem} $f_j$ is stably homotopic to the stable isomorphism
\begin{eqnarray*}
\left( -g^* \oplus g^{-*}, -g^* \oplus g^{-*}, 0 \right) \co
\del (L, 0)\oplus (L, L^*) \isora \del (L, 0)\oplus (L, L^*).
\end{eqnarray*}
In particular, if $z$ is simple, we can choose $g=\id_{V^*}$ and therefore $f_j$
is stably homotopic to the boundary of the identity $\id_V\in\Aut(V,0)$.
\end{ex}

\begin{ex}
\label{elemex}
Let $\delta(K,\rho) = (H_\epsilon(K);K, \svec{1\\ \rho}K)$ be the boundary of an asymmetric form
$(K,\rho)$.
Using the notation of the proof of Proposition \ref{orthoglueprp} with $(M, \psi) = (M', \psi') = H_\epsilon(K)$ we have
\begin{eqnarray*}
j&=&\svec{1\\\rho}\co K\ra M=K\oplus K^*,\\
j^\perp&=&\svec{1\\-\epsilon \rho^*} \co K \ra M=K\oplus K^*,\\
\sigma &=&\svec{0\\1}\co K^*\ra K\oplus K^*.
\end{eqnarray*}
The isometry $h$ from Proposition \ref{orthoglueprp} is
\[
h=\svec{
-1&1&0\\
0&1&0\\
(\rho+\epsilon \rho^*)&-\epsilon \rho^*&1
}\co
(K, \theta)\oplus H_\epsilon(K) \isora (K,\theta)\oplus H_\epsilon(K)
\]
where $(K, \theta)$ is the $\epsilon$-quadratic form with $\theta=[\rho]\in Q_\epsilon(K)$.
There is a homotopy
\begin{eqnarray*}
\Delta&\co&\del h \simeq \del(-\id_K \oplus \id_{H_\epsilon(K)})\co
\del((K,\theta)\oplus H_\epsilon(K)) \cong
\del((K,\theta)\oplus H_\epsilon(K)),\\
\Delta&=&\svec{0&0&-1\\0&0&0\\-\epsilon&0&\epsilon c^*}\co K^*
\oplus K^*\oplus K\ras K\oplus K\oplus K^*
\end{eqnarray*}
and therefore the stable isomorphism $f_j$ is homotopic to
$\del(-\id_K)$.
\end{ex}

\section{The structure of $\ol{\Lambda}$}
\label{structuresec}
In this section we prove our main theorem about the structure of $\ols{\Lambda}$.

\subsection{The map $b\co \ols{\Lambda} \ra \mathcal{F}_{2q}^{\rm zs}(\Lambda)\times \mathcal{F}_{2q}^{\rm zs}(\Lambda)$}

Let $x=(M,\psi;F,V)$ be an $\epsilon$-quadratic quasi-formation and let $j\co V\hookrightarrow M$
and $j^\perp\co V^\perp \hookrightarrow M$ be the inclusions of $V$ and
its annihilator where $V^{\perp}$ is $s$-based as in Remark \ref{eqrem} (ii).  We shall always write $(V,\theta)$ and $(V^\perp, -\theta^\perp)$ for the induced $\epsilon$-quadratic forms $\theta=j^*\psi j$ and $\theta^\perp= (j^\perp)^*\psi(j^\perp)$.  We call $(V,\theta)$ and $(V^\perp, -\theta^\perp)$ the {\bf boundaries of $x$} and record their basic properties in the following


\begin{prop}
\label{prepprp}
Let $x=(M,\psi;F,V)$ and $x'=(M',\psi';F',V')$ be \eqs\ with boundaries $(V,\theta)$,
$(V^\perp, -\theta^\perp)$, $(V',\theta')$ and $({V'}^\perp, -{\theta'}^\perp)$ respectively. Then
\begin{enumerate}

\item
\label{prepprphyperit}
there is an isometry
$(V,\theta)\oplus H_\epsilon(F) \cong (V^\perp,-\theta^\perp)\oplus H_\epsilon(F)$,

\item
if $[x]=[x']\in \ols{\Lambda}$ then
$(V,\theta)\oplus(P,0)\cong (V',\theta')\oplus (P',0)$
and $(V^\perp,\theta^{\perp})\oplus(P,0)\cong ({V'}^\perp,{\theta'}^\perp) \oplus (P',0)$
for some based modules $P$, $P'$,

\item
\label{prepprpelemit}
if $x$ is elementary then $(V,\theta)\cong(V^\perp,-\theta^\perp)$,

\item
$\Rad(V,\theta)=\Rad(V^\perp,-\theta^\perp)$ and $\rk(V)=\rk(V^\perp)= \frac{1}{2}\rk(M)$.
\end{enumerate}
\end{prop}
\begin{proof}
\begin{enumerate}
\item Follows from Propositions \ref{orthoglueprp} and
\ref{bdryextprp}.

\item
This statement follows from the fact that an isomorphism of quasi-formations induces
 isometries of its boundary forms and that the boundaries of \eqfs\ are zero forms.

\item
By definition $x$ is isometric to the boundary of an asymmetric form $(K,\rho)$. Therefore
$(V,\theta)\cong (K,[\rho])$ and $(V^\perp,\theta^\perp)\cong (K,[-\rho])$.

\item
By definition $\Rad(V, \theta) = V \cap V^{\perp}= \Rad(V^{\perp}, -\theta^{\perp})$.
The second equality follows from the decomposition $M \cong V \oplus V^{\perp}$ in Proposition 3.9.
\qedhere
\end{enumerate}
\end{proof}

\begin{defi} \label{bdefi}
An immediate consequence of (ii) above is that there is a unital  monoid map
\begin{eqnarray*}
b \co\ols{\Lambda} &\lra& \mathcal{F}_{2q}^{\rm zs}(\Lambda)\times \mathcal{F}_{2q}^{\rm zs}(\Lambda)
\\
{[M,\psi;F,V]} &\mt& ([V,\theta], [V^\perp, -\theta^\perp]).
\end{eqnarray*}
\end{defi}
\noindent
We record the essential properties of $b$ in the following

\begin{cor}
\label{bcor}
The monoid maps $b \co\ols{\Lambda} \ra \mathcal{F}_{2q}^{\rm zs}(\Lambda)\times \mathcal{F}_{2q}^{\rm zs}(\Lambda)$
\noindent
and $b_{\mathcal{E}} := b|_{\mathcal{E}\ol{\Lambda}}$ satisfy
\begin{enumerate}
\item
${\rm Im}(b) = \{ ([w], [w'])\,|\, [w] + [H_\epsilon(\Lambda^r)] = [w'] + [H_\epsilon(\Lambda^r)] \text{~for some~} r \}$,

\item
$b_{\mathcal{E}} \co \oels{\Lambda} \isora \Delta(\mathcal{F}_{2q}^{\rm zs}(\Lambda)) : = \left\{ \left([w],[w]\right)\, |\, [w]\in \mathcal{F}_{2q}^{\rm zs}(\Lambda) \right\}$,

\item
$b^{-1}((0,0)) = L_{2q+1}^{}(\Lambda)$.

\end{enumerate}
\end{cor}

\begin{proof}
\begin{enumerate}
\item
One inclusion follows from Proposition \ref{prepprp} (i).  So let $(W,\sigma)$ and $(W',\sigma')$ be $\epsilon$-quadratic forms and let $Q$ be a based module such that
there exists an isometry
$$h \co (W,\sigma)\oplus H_\epsilon(Q) \cong (W',\sigma')\oplus H_\epsilon(Q).$$
Applying Lemma \ref{unionlem} (iv) we see that the form
$$(M, \psi) := ((W, \sigma) \oplus H_\epsilon(Q)) \cup_{\del h} ((W', -\sigma') \oplus -H_\epsilon(Q))$$
is isometric to the trivial double and hence hyperbolic. The boundary of $H_\epsilon(Q)$ is trivial and so $\del h$ is stably homotopic to some stable isomorphism $f$ between $\del(W,\sigma)$ and $\del(W',\sigma')$. Thus
$$(M, \psi) \cong ((W, \sigma) \cup_ f (W', -\sigma')) \oplus H_\epsilon(Q \oplus Q).$$
Now consider the \eq\ $x := (M, \psi; L, j(W) \oplus Q \oplus Q)$ where $j$ is the map from Lemma \ref{unionlem} and $L$ is some arbitrary Lagrangian.  By construction $b([x]) = ([W,\sigma], [W',\sigma'])$.

\item
By Proposition \ref{prepprp} (\ref{prepprpelemit} and by Example \ref{elemex}, 
$b_{\mathcal{E}}\left(\oels{\Lambda}\right) = \Delta(\mathcal{F}_{2q}^{\rm zs}(\Lambda))$.

Now assume that $b\left([\delta(K,\rho)]\right) = b\left([\delta(K',\rho')]\right)$
for asymmetric forms $(K,\rho)$ and $(K',\rho')$.  That is, there are based modules $X$ and $X'$ and an isometry
$h\co (Y,[\nu])=(K,[\rho])\oplus (X,0) \isora (Y',[\nu'])=(K',[\rho'])\oplus (X',0)$. There is an
$\chi\in\Hom_\Lambda(Y, Y^*)$ such that $h^*\nu'h-\nu=\chi-\epsilon \chi^*\in\Hom_\Lambda(Y, Y^*)$.
This isometry induces an isomorphism
\[
\svec{h&0 \\ 0&h^{-*}}\svec{1&0\\\chi-\epsilon \chi^* & 1}\co (H_\epsilon(Y);Y^*,\svec{1\\ \nu}Y) \isora (H_\epsilon(Y');{Y'}^*,\svec{1\\ \nu'}Y')
\]
Adding trivial formations and the relation (\ref{shifteqn}) shows that $[\delta(Y,\nu)]=[\delta(Y',\nu')]\in\ols{\Lambda}$.
Because boundaries of zero-forms vanish in the $l$-monoid,
$[\delta(K,\rho)]=[\delta(Y,\nu)]=[\delta(Y',\nu')]=[\delta(K',\rho')] \in \ols{\Lambda}$.

\item
If $x  = (M, \psi; L, V)$ is a formation such that $(V, \theta) = (V, 0)$ then $V^{\perp} = V$ (possibly with a different basis but that does not concern us).  Hence $(V, \theta)$ is a zero form if and only if $(V^{\perp}, -\theta^{\perp})$ is a zero form.  It is now a matter of definition that $L_{2q+1}(\Lambda) = b^{-1}((0, 0))$.
\end{enumerate}
\end{proof}

\begin{defi} \label{e(v)defi}
Given a $0$-stabilised form $[v]$, we write $e([v])$ for the
unique element of $\mathcal{E}\ol{\Lambda}$ such that $b(e([v])) =
([v], [v])$.  In fact $e([v]) = [\delta(V, \rho)]$ from Definition
\ref{quasiformdef} where $v = (V,  \theta)$ and $\rho$ is a
representative of $\theta$.
\end{defi}

\subsection{Boundary isomorphisms}
\label{boundaryisosec}

Recall for quadratic forms $v$ and $v'$ of the same rank we write $v \sim v'$ if $[v] + [H_\epsilon(\Lambda^k)] = [v'] + [H_\epsilon(\Lambda^k)]$ for some $k$.  If $v \sim v'$ then $([v], [v']) \in {\rm Im}(b)$ and we define the set
\[ l_{2q+1}(v, v') : = b^{-1}([v], [v']) \subset \ol{\Lambda}\]
which is the focal point of our main theorem.  Recall also that a quadratic form $w$ has a boundary $\del w$ which is a split formation and that $\Iso(\del v, \del v')$ denotes the set of homotopy classes of stable isomorphisms from $\del v$ to $\del v'$.  If $v \sim v'$ then for some module $Q$, $\del (v \oplus (Q, 0)) \cong \del (v \oplus (Q, 0))$.  In this subsection we gather the results from sections \ref{form&qformsec} and \ref{gluesec} to calculate $l_{2q+1}(v, v')$ in terms of the classical $L$-groups and an appropriate elaboration of ${\rm Iso}(\del v, \del v')$.

\begin{defi}
\label{bautdef}
Given $\epsilon$-quadratic forms $(V,\theta)$ and $(V',\theta')$
we define the {\bf boundary isomorphism set $\bIso((V,\theta),(V',\theta'))$}
to be the set of orbits of the group action
\begin{eqnarray*}
(\Aut(V,\theta)\times \Aut(V', \theta'))\times\Iso(\del(V,\theta), \del(V',\theta'))
& \lra & \Iso(\del(V,\theta), \del(V', \theta'))\\
((g,h), [f]) & \mt & [\del h \circ f \circ\del g^{-1}].
\end{eqnarray*}
When $(V, \theta) = (V', \theta')$ we define $\bAut(V,\theta):=\bIso((V,\theta),(V,\theta))$ and let $1 \in \bAut(V,\theta)$ be the orbit containing the isomorphism $\partial\id_{(V',0)}=(\id_V,\id_V,0)$.
\end{defi}

\begin{rem} \label{bIsorem}
Isometries $k\co (V, \theta) \isora (W, \sigma)$ and $l\co (V', \theta') \isora (W', \sigma')$ give rise to an identification
$\bIso((V, \theta), (V', \theta')) \cong \bIso((W, \sigma), (W', \sigma'))$ via $[f] \mapsto [\del l \circ f \circ \del k^{-1}]$.  Evidently this identification is independent of the isometries chosen.  We shall often make this sort of identification without comment.
\end{rem}

\begin{defi}
\label{deltadefi}
For an \eq\ $x=(M, \psi; L, V)$ let $\delta(x) \in \bIso((V, \theta), (V^{\perp}, -\theta^{\perp}))$ be the homotopy class of the stable isomorphism $f_j$ defined in Proposition \ref{orthoglueprp}.   We note that $\delta(x)$ does not depend upon $L$.
\end{defi}

\begin{ex}
\label{delex}
Let $x = (M, \psi; L, V)$ be an \eq.
\begin{enumerate}
\item
If $x$ is an elementary then by Lemma \ref{prepprp} and Example \ref{elemex} $\delta(x) = 1 \in \bAut(V,\theta)$.  
\item 
If $x$ is a simple $\epsilon$-quadratic formation then by Example \ref{formationex} $\delta(x) = 1 \in \bAut(V,\theta)$.
\item If $y = x \oplus \delta(M, \rho)$ where $\rho$ represents $\psi$ then by Proposition \ref{bdryextprp} $\delta(y) = 1 \in \bAut((V, \theta) \oplus (M, \psi))$.
\end{enumerate}
\end{ex}

As elements of $\ols{\Lambda}$ are stable equivalence classes of
\eqs, we need to stabilise the boundary isomorphism set in order
to convert $\delta(x)$ into an invariant of $[x] \in
\ol{\Lambda}$.  For any based module $Q$ we have the stabilisation
map
\[\Aut(v) \lra \Aut(v \oplus (Q,0)),\quad g \mapsto g \oplus \id_Q.\]
We use this map and the analogous map for $v'$ in the definition of the stabilisation map
\begin{eqnarray*}
s_Q \co \bIso(v, v') \lra \bIso(v\oplus (Q, 0), v' \oplus (Q, 0)),\quad
{[f]} \mt {[ f\oplus  \del \id_Q]}.
\end{eqnarray*}
The set of all based modules $(\Lambda^k,\mathcal{B})$ with the relation $\leq$ is a directed poset (Definition \ref{baseddef}).
In that way the maps $s_{(\Lambda^k,\mathcal{B})}$ define a directed system of sets which leads to the following definition.

\begin{defi}
\label{sbIsodefi}
Let $v \sim v'$ be $\epsilon$-quadratic forms.
\begin{enumerate}
\label{sbautdef}
\item
The {\bf stable boundary isomorphism set} is
\[
\sbIso(v,v') := \varinjlim_{Q=(\Lambda^k,\mathcal{B})}\bIso(v\oplus (Q,0), v'\oplus (Q,0)).
\]

\item
When $v = v'$ we have the {\bf stable boundary automorphism set}
\[\sbAut(v) := \sbIso(v,v) \text{~~with~~}1 = [\del\id_V] \in \sbAut(v).\]

\item
There is an obvious stabilisation map
\begin{eqnarray*}
s\co \bIso(v,v') \lra \sbIso(v,v'),
\quad[f] \mapsto [f].
\end{eqnarray*}

\end{enumerate}
\end{defi}


\begin{lem}
\label{stableinjlem}
For any based module $Q$ the stabilisation map
\[
s_Q\co \bAut(v) \lra \bAut(v\oplus (Q,0)),\quad [f] \mt [f\oplus \del \id_Q]
\]
satisfies $s_Q^{-1}(1) = 1$.  Consequently for $s\co\bAut(v) \ra \sbAut(v)$, $s^{-1}(1) = 1$.
\end{lem}

\begin{proof}
Let $f$ be an automorphism of $\del v \oplus (P,P^*)$ such that
$f\oplus \del \id_Q= 1\in \bAut(v\oplus (Q,0))$. After possibly enlarging $P$
there is an isometry $G$ of $v\oplus(Q,0)$ and a homotopy
\begin{eqnarray*}
&&\Delta\co f\oplus \del \id_Q\simeq  \del G \oplus \id_{(P,P^*)}\co\\
&&\qquad \del v \oplus \del(Q,0)\oplus (P,P^*) \isora
\del v \oplus \del(Q,0)\oplus (P,P^*).
\end{eqnarray*}
The definition of homotopy shows, firstly, that
\[
G=\svec{G_V&0\\G_{21}& \id_Q}\co W\oplus Q \isora W\oplus Q
\]
(where $G_V$ is an isometry of $w$) and that, secondly,
$\Delta$ induces a stable homotopy between $f$ and $\del G_V$.
\end{proof}

We next describe the maps which allow us to compute $l_{2q+1}(v, v')$.  Firstly, abusing notation, we write
\[\rho\co  L_{2q+1}(\Lambda) \ra l_{2q+1}(v, v')\]
for the action of $L_{2q+1}(\Lambda) \ni [z]$ on $\ols{v, v'} \ni [x]$, $\rho([x], [z]) = [x] + [z]$.  Secondly there is the map
\[
\delta\co \ols{v, v'}  \lra \sbIso(v, v')
\]
which is defined as follows.  Given $[x] \in \ols{v, v'}$ choose a representative $x = (M, \psi; F, W)$ where for notational reasons we have written $W$ for the second summand in place of the usual $V$.  If $(W, \sigma)$ is the induced form on $W$ then by definition $[W, \sigma] = [v]$ and $[W^{\perp}, -\sigma^{\perp}] = [v']$.   Applying Definition \ref{deltadefi}, we have $\delta(x) \in {\rm bIso}((W, \sigma), (W^{\perp}, -\sigma^{\perp}))$.  It follows that there are modules $Q$ and $P$ and isomorphisms 
$k\co (W, \sigma) \oplus (Q, 0) \isora v \oplus (P, 0)$ and 
$l\co (W^{\perp}, -\sigma^{\perp}) \oplus (Q, 0) \cong v' \oplus (P, 0)$.  
We define
\[ \delta([x]) := [\del l \circ (\delta(x) \oplus \del {\rm id}_Q) \circ \del k^{-1}]  \in {\rm sbIso}(v, v').\]

We now show that $\delta$ is well-defined.  By Remark \ref{bIsorem}, different choices of $l$ and $k$ don't effect $\delta([x])$.  The construction of $f_j$ is well-defined up to homotopy. By the naturality of Proposition \ref{orthoglueprp} (in particular
equation (\ref{transformfeqn}) from the proof) an isomorphism of \eqs\ doesn't
change $\delta([x])$ either.  Lastly, we have to analyse the effect of adding trivial formations and the relation (\ref{shifteqn}).   Since $\delta(x \oplus x') = \delta(x) \oplus \delta(x')$ and by Example \ref{delex}[ii] $\delta(x') = 1$ for any simple formation $x'$ we see that adding simple formations, and in particular trivial formations, does not alter $\delta([x])$.  As we remarked above, the definition of $\delta(x)$ does not depend upon the Lagrangian in the quasi-formation $x$ and so $\delta$ is invariant under the relation (\ref{shifteqn}) which only alters Lagrangians.

We turn to the preliminaries required to determine the image of $\delta$.
Gluing quadratic forms together defines a map
\[
\kappa\co \bIso(v, v') \lra L^s_{2q}(\Lambda), \quad [f] \mt [\kappa(f)]
\]
where $\kappa(f)=v\cup_f(-v')$.
By Lemma \ref{unionlem}, $[\kappa(f)] \in L_{2q}^s(\Lambda)$ doesn't change if one takes another representative
for $[f] \in \bIso(v, v')$. Stabilisation of $f$ with the identity on another zero form will only add
a hyperbolic form to $\kappa(f)$. Hence $\kappa$ extends to a well-defined map $\kappa\co  \sbIso(v,v') \ra L_{2q}^s(\Lambda)$.

\begin{thm}\label{generalmainthm}
Let $v \sim v'$ be $\epsilon$-quadratic forms with $v \sim v'$. There is an ``exact'' sequence of sets
\[
\oLs{\Lambda}\stackrel{\rho}{\lra} l_{2q+1}(v, v') \stackrel{\delta}{\lra} \sbIso(v, v') \stackrel{{\kappa }}{\lra} L^s_{2q}(\Lambda)
\]
by which we mean that the orbits of $\rho$ are the fibres of $\delta$ and $\im(\delta)={\kappa}^{-1}(0)$.
\end{thm}

The case $v = v'$ is of particular interest.  By combining Theorem \ref{generalmainthm} and Example \ref{elemex} we obtain the following

\begin{cor}
\label{maincor}
For any $\epsilon$-quadratic form $v$ there is an exact sequence
\[
\oLs{\Lambda}\stackrel{\rho}{\lra} \ols{v, v} \stackrel{\delta}{\lra} \sbAut(v) \stackrel{\kappa}{\lra} L^s_{2q}(\Lambda)
\]
in the following sense: the orbits of the action $\rho$ are
precisely the fibres of the map $\delta$ and
$\im(\delta)=\kappa^{-1}(0)$.  Moreover $\delta([x]) = 1 \in
\sbAut(v)$ if and only if $[x]$ is elementary modulo the action of
$L_{2q+1}^s(\Lambda)$.
\end{cor}

\begin{cor}
\label{neatcor}
Let $x=(M,\psi; L, V)$ be an \eq.  Then $[x] \in \ols{\Lambda}$ is elementary modulo $\oLs{\Lambda}$
if and only if there is a module $P$ such that
$(V,\theta) \oplus (P, 0) \cong (V^\perp, -\theta^\perp) \oplus (P, 0)$ and
$\delta(x \oplus (H_\epsilon(P); P, P^* )) = 1 \in \bAut((V,\theta) \oplus (P, 0))$,
\end{cor}

\begin{proof}
Follows from Lemma \ref{stableinjlem}, Proposition \ref{prepprp} and Corollary \ref{maincor}.
\end{proof}


\begin{proof}[Proof of Theorem \ref{generalmainthm}]
By definition $\delta$ is invariant under the action of $\oLs{\Lambda}$.  Now let $[x], [x'] \in l_{2q+1}(v, v')$ be such that $\delta([x]) = \delta([x'])$.  Choose equal rank representatives
$(M,\psi; F,V)$ and $(M',\psi'; F',V')$ for $[x]$ and $[x']$ and let $f_j$ and $f_{j'}$ be the associated boundary isomorphisms for the embeddings $j\co(V,\theta) \hra (M,\psi)$ and $j'\co(V',\theta') \hra (M',\psi')$ as defined in Lemma \ref{orthoglueprp}.  The equality $\delta([x])=\delta([x'])$ implies that there are modules $P$ and $P'$, isometries $k\co(V,\theta)\oplus (P,0) \isora (V',\theta')\oplus (P',0)$
and $l\co(V^\perp,-\theta^\perp)\oplus (P,0) \isora ({V'}^\perp,{-\theta'}^\perp)\oplus (P',0)$
and a stable homotopy $\Delta$ between $f_{j'}\oplus \id_{\del(P',0)}$ and
$\del l \circ \left( f_j\oplus \id_{\del(P,0)}\right) \circ \del k^{-1}$.

Lemmas \ref{unionlem} and \ref{orthoglueprp} yield the following commutative diagram.
\begin{eqnarray*}
\xymatrix@C+60pt
{
V\oplus P
\ar[r]^-{j\oplus \svec{\id_P\\0}}
&
(M,\psi)\oplus H_\epsilon(P)
\ar[d]^{r_j\oplus \id_{H_\epsilon(P)}}_\cong
\\
V\oplus P
\ar[r]^-{\svec{\id_V\\0}\oplus\svec{\id_P\\0}}
\ar@{=}[u]
\ar@{=}[d]
&
\left( (V,\theta)\cup_{f_j}-(V^\perp,\theta^\perp)\right) \oplus
H_\epsilon(P)
\ar[d]_\cong
\\
V\oplus P
\ar[r]^-{\svec{1&0\\0&1\\0&0\\0&0}}
\ar[d]_{k}^\cong
&
\left((V,\theta)\oplus (P,0)\right)\cup_{f_j\oplus \id_{\del(P,0)}}
-\left((V^\perp, \theta^\perp)\oplus (P,0)\right)
\ar[d]^{\svec{1&\Delta_1\\0&1}\svec{k&0\\0&{l}^{-*}}}_\cong
\\
V'\oplus P'
\ar[r]^-{\svec{1&0\\0&1\\0&0\\0&0}}
\ar@{=}[d]
&
\left((V',\theta)\oplus (P',0)\right)\cup_{f_{j'}\oplus \id_{\del(P',0)}}
-\left( ({V'}^\perp,{\theta'}^\perp)\oplus (P',0)\right)
\ar[d]_\cong
\\
V'\oplus P'
\ar[r]^-{\svec{1&0\\0&1\\0&0\\0&0}}
\ar@{=}[d]
&
\left( (V',\theta')\cup_{f_{j'}}-({V'}^\perp,{\theta'}^\perp)\right) \oplus
H_\epsilon(P')
\ar[d]^{r_{j'}\oplus \id_{H_\epsilon(P')}}_\cong
\\
V'\oplus P'
\ar[r]^-{j'\oplus \svec{\id_{P'}\\0}}
&
(M',\psi')\oplus H_\epsilon(P')
}
\end{eqnarray*}
The composition of the right hand isometries yields an isomorphism
\[
g\co((M,\psi) \oplus H_\epsilon(P);L \oplus P, V\oplus P) \isora ((M',\psi') \oplus H_\epsilon(P');g(L \oplus P), V' \oplus P')
\]
Therefore $[x]=[x']+[z]\in\ol{\Lambda}$ with
\[z=[(M',\psi') \oplus H_\epsilon(P'); L' \oplus P', g(L \oplus P)] \in\oLs{\Lambda}.\]

Finally, the composition of $\kappa$ and $\delta$ maps an \eq\ $x=(M,\psi; F, V)$
to $[(M, \psi)] = 0 \in L^s_{2q}(\Lambda)$ and therefore $\kappa\circ\delta$ is trivial.  In the other direction, let $v  = (V, \theta)$ and $v' = (V', \theta')$ be $\epsilon$-quadratic forms with $v \sim v'$ and let $f \co  \del v \oplus (P,P^*) \isora \del v' \oplus (P,P^*)$ be a stable isomorphism between their boundaries such that $\kappa([f])=0$.  This means that the form $\kappa(f) = v \cup_f(-v')$ is stably hyperbolic \ie there are based modules $P$ and $Q$ such that $\kappa(f) \oplus H_\epsilon(P) \cong H_\epsilon(Q)$.  But $\kappa(f) \oplus H_\epsilon(P) = \kappa(f \oplus \del\id_{P})$.  It follows that $\delta([x]) = [f] \in {\rm sbIso}(v, v')$  where $x$ is the quasi-formation $x = (\kappa(h \oplus \id_{P}); L, V \oplus P)$ for any Lagrangian $L \subset \kappa(f \oplus \id_{P})$. But $b([x]) = ([v], [v'])$ and so $\delta\co l_{2q+1}(v, v') \ra \kappa^{-1}(0)$ is onto.\qedhere

\end{proof}

\subsection{The Grothendieck group of $\ol{\Lambda}$}
 
In this subsection we prove that for every $[x] \in \ol{\Lambda}$ there is an integer $k$ such that $[x] + e([H_\epsilon(\Lambda^k)])$ is elementary
(see Definiton \ref{e(v)defi}).  
This is an algebraic analogue of \cite{Kre99}[Theorem 2] that can also
be used to find an alternative proof of that theorem.

We begin with a Lemma about the action of $L_{2q+1}^s(\Lambda)$ on $\ol{\Lambda}$.  
Following the original definition of $L_{2q+1}^s(\Lambda)$ we define
$z(\alpha) = (H_\epsilon(F); F, \alpha(F))$ for $\alpha\in\Aut(H_\epsilon(F))$. Every simple $\epsilon$-quadratic 
formation can be represented in this manner up to isometry.

\begin{lem}
\label{L-actionlem}
Let $[x] \in \ol{\Lambda}$ be represented by $x=(H_\epsilon(F); F, W)$ and suppose that an isometry $\alpha\in\Aut(H_\epsilon(F))$ restricts to an isometry of $W$.  Then $[z(\alpha)] \in L_{2q+1}^s(\Lambda)$ acts trivially on $[x]$.
\end{lem}

\begin{proof}
Such an isometry $\alpha$ is also an isomorphism
$x\isora (H_\epsilon(F);\alpha(F),W)$ and
therefore
$
[z(\alpha)]+[x] = [(H_\epsilon(F);F,\alpha(F))\oplus(H_\epsilon(F);\alpha(F),W)] =[H_\epsilon(F);F,W].
$
\end{proof}

\begin{cor}
\label{L-actioncor}
If $[x] \in \ol{\Lambda}$ is represented by $x=(M, \psi; L, V)$ and $(V, \theta) \cong H_{\epsilon}(F) \oplus (V', \theta')$ splits off a hyperbolic summand then every element of $L_{2q+1}^s(\Lambda)$ which can be represented by a formation $z=(H_{\epsilon}(F); F, G)$ acts trivially on $[x]$.
\end{cor}

\begin{proof}
There is a decomposition $(M, \psi) = H_\epsilon(F) \oplus (M', \psi')$ such that $H_\epsilon(F)\subset (V,\theta)$ and $(V', \theta') \subset (M', \psi')$.  
There is an isometry $\alpha_0 \in \Aut( H_\epsilon(F))$ such that $[z] = [z(\alpha_0)]$.  
We extend $\alpha_0$ to $\alpha\in \Aut(M, \psi)$ where $\alpha = \alpha_0 \oplus \id_{M'}$.  
Evidently $\alpha$ satisfies the hypothesis of Lemma \ref{L-actionlem} and so $[z(\alpha_0)] = [z(\alpha)]$ acts trivially on $[x]$.
\end{proof}

\begin{prop}
\label{stablyelemprop}
For every element $[x] \in \ol{\Lambda}$, there is a positive integer $k$ such that $[x] + e([H_\epsilon(\Lambda^k)])$ is elementary.
\end{prop}

\begin{proof}
Write $x=(H_\epsilon(L); L, V)$.  It follows immediately from Example \ref{delex}[iii] and the definitions that $b([x] + e([H_\epsilon(L)])) \in \Delta(\mathcal{F}_{2q}^{\rm zs}(\Lambda))$ and that $\delta([x] + e([H_\epsilon(L)])) = 1 \in {\rm sbAut}((V, \theta) \oplus H_{\epsilon}(L))$.  Hence by Corollary \ref{maincor} there is a $[z] \in L_{2q+1}^s(\Lambda)$ such that $([x] + e([H_\epsilon(L)])) + [z] $ is elementary.  Now $z$ can be chosen to be of the form $(H_\epsilon(F); F, G)$ and by Corollary \ref{L-actioncor} $[z] \oplus e([H_\epsilon(F)]) = e([H_\epsilon(F)])$.  It follows that $[x] \oplus e([H_\epsilon(F\oplus L)])$ is elementary.
\end{proof}

\begin{cor} \label{grothgrpcor}
The monoid homorphisms $\mathcal{F}_{2q}^{\rm zs}(\Lambda) \cong \mathcal{E}\ol{\Lambda} \hra \ol{\Lambda}$ induce isomorphisms of the respective Grothendieck groups
\[ {\rm Gr}(\mathcal{F}_{2q}^{\rm zs}(\Lambda)) \cong {\rm Gr}(\mathcal{E}\ol{\Lambda}) \cong {\rm Gr}(\ol{\Lambda}).\]
\end{cor}

\begin{proof}
Let $i\co  \mathcal{E}\ol{\Lambda} \hra \ol{\Lambda}$ denote the inclusion.  The induced homomorphism ${\rm Gr}(i)\co {\rm Gr}(\mathcal{E}\ol{\Lambda}) \ra {\rm Gr}(\ol{\Lambda})$ is onto by Proposition \ref{stablyelemprop}.  On the other hand the monoid isomorphism
$b_{\mathcal{E}} \co\mathcal{E}\ol{\Lambda} \cong \Delta(\mathcal{F}_{2q}^{\rm zs}(\Lambda))$
factors as $b_{\mathcal{E}} = b \circ i$ and this shows that ${\rm Gr}(i)$ is injective.
\end{proof}

\section{Calculations in special cases}
\label{calcsec}
In this we calculate ${\rm sbAut}(v)$ and ${\rm sbIso}(v, v')$ in special situations.  The first subsection concerns ${\rm sbAut}(v)$ when $v$ is the sum on linear and simple forms.  In the next subsection we compute ${\rm sbIso}(v, v')$ when $v$ and $v'$ become nonsingular after localisation.  In this case $\del v$ and $\del v$ are quadratic linking forms.  In the final subsection we compute the monoid $l_{3}(\bZ)$ and describe the set $l_5(\bZ)$.

\subsection{On $l_{2q+1}(v, v)$ for linear and simple quadratic forms}
\label{linearandsimplesubsec}
Recall that an $\epsilon$-quadratic form $v=(V, \theta)$ is linear if $\theta +\epsilon \theta^* = 0$ and simple if $\theta +\epsilon \theta^*\co  V \ra V^*$ is a simple isomorphism.  Prior to our first definition, we warn the reader that torsions and in particular torsions of non-simple isometries will play a key role in this subsection and thus, {\rm isomorphisms and isometries are not assumed to be simple in this subsection.}

\begin{defi}
Give an $\epsilon$-quadratic form $(V, \theta)$, let $\Aut^h(V, \theta)$ be the group of all isometries of $(V, \theta)$, simple or not.  Let
\[
Z^1(\Wh(\Lambda)) :=\{[h] \in \Wh(\Lambda)\, |\, [h] = -[h^*] \}
\]
and let ${\rm UWh}(\Lambda)  \subset Z^1(\Wh(\Lambda))$ be the subgroup of all torsions $\tau(h) \in {\rm Wh}(\Lambda)$ where $h \in \Aut^h(H_\epsilon(L))$ for some hyperbolic form.
\end{defi}

We begin with the trivial case.
\begin{lem} \label{bzerolem}
Suppose that $v = (N, 0)$ is a zero form.  Then $\sbAut(v) = Z^1(\Wh(\Lambda))$ and identifying $L_{2q+1}(\Lambda) = l_{2q+1}(v, v)$, the exact sequence of Corollary \ref{maincor} for $v$,
maps onto a fragment of the Ranicki-Rothenberg sequence,
\begin{eqnarray*}
\xymatrix
{
&
L_{2q+1}^s(\Lambda)
\ar@{=}[d]
\ar@{^{(}->}[r]^-\rho
& 
\oL{\Lambda} 
\ar@{->>}[d]
\ar[r]^-\delta
&
Z^1(\Wh(\Lambda))
\ar@{->>}[d]
\ar[r]^-\kappa
&
L_{2q}^s(\Lambda) 
\ar@{=}[d]
& 
\\
\cdots
\ar[r]
&
L_{2q+1}^s(\Lambda) 
\ar[r]
&
L_{2q+1}^h(\Lambda) 
\ar[r]
& 
\widehat H^1(\Wh(\Lambda)) 
\ar[r]
& 
L_{2q}^s(\Lambda)
\ar[r] 
&
\cdots,
}
\end{eqnarray*}
where $\widehat H^1(\Wh(\Lambda)) = Z^1(\Wh(\Lambda))/\{[h] - [h]^*\, |\, [h] \in \Wh(\Lambda)\}$, $L_{2q+1}(\Lambda)$ maps onto the unbased odd-dimensional surgery obstruction group $L_{2q+1}^h(\Lambda)$ (see \cite{Ran73}).  Moreover, $\delta = -\tau^*$ where $\tau$ is the map in Remark \ref{Ltorsionrem} and the image of $\delta$ is ${\rm UWh}(\Lambda)$ so there is a short exact sequence
\[0 \dlra{} L_{2q+1}^{s}(\Lambda) \dlra{} L_{2q+1}(\Lambda) \dlra{\delta} {\rm UWh}(\Lambda) \dlra{} 0.\]
\end{lem}

\begin{proof}
Using the definitions and Lemma \ref{sthomlem} one sees that there is an isomorphism
\begin{eqnarray*}
\delta^T_{(N,0)}\co \Aut(\partial(N,0)) &\isora& \{ h\in GL(N)\, |\, [h]=-[h^*]\in\Wh(\Lambda)\}
\\
{[(\alpha,\beta,\nu)]}&\mt& [\alpha|_N\co N \isora N]
\end{eqnarray*}
where $GL(N)$ is the group of all isomorphisms $N \cong N$, simple or not.  Since $\Aut(N, 0) = \{h \in GL(N) | [h] = 0 \in \Wh(\Lambda) \}$ we obtain after stabilisation that
$\sbAut(N,0) \cong Z^1(\Wh(\Lambda))$.  The map onto the Ranicki-Rothenberg sequence follows from the definitions as does the identity $\delta = -\tau^*$.   Finally, the identification of the image of $\delta$ comes from the fact that $L_{2q+1}(\Lambda) \cong U(\Lambda, \epsilon)/RU(\Lambda, \epsilon)$ where $U(\Lambda,\epsilon)$, the stable unitary group, is defined just as $SU(\Lambda, \epsilon)$ in subsection \ref{origolsec} minus the requirement that isometries need be simple.
\end{proof}

Now let ${\rm U'Wh}(\Lambda)  \subset Z^1(\Wh(\Lambda))$ be the subgroup of all torsions $\tau(h) \in {\rm Wh}(\Lambda)$ where $h \in \Aut^h(H^\epsilon(L))$ for some symmetric hyperbolic form.  The main result of this subsection is the following

\begin{prop}\label{simplelinearprop}
If $v = (N, \eta) \oplus (M, \psi)$ is the sum of a linear form $(N, \eta)$ and simple form $(M, \psi)$ and if ${\rm UWh}(\Lambda) = {\rm U'Wh}(\Lambda)$ then $L_{2q+1}(\Lambda)$ acts transitively on $l_{2q+1}(v, v)$.
\end{prop}

In the general case $v =  (N, \eta) \oplus (M, \psi)$, $\del v = \del (N, \eta) \oplus \del  (M, \psi) $ and by Lemma \ref{splitisolem}, $\del(M,\psi)\cong(M,M^*)$ is isomorphic to a trivial split formation.  
It follows from 
Lemma \ref{sthomlem} that there is a homomorphism
\[
\delta^T_{v}\co  \Aut(\del v) \ra Z^1(\Wh(\Lambda)),\quad
[(\alpha, \beta, \nu)] \mapsto [a] 
\]
where $a = \alpha|_N\co  N \cong N$ as in the notation of Lemma \ref{sthomlem}.
Evidently $\delta^T_{v}$ is well-defined after stabilisation. For any isometry $h \in\Aut(N, \eta)$ of the linear form $(N, \eta)$, $\delta^T_{(N, \eta)}(\del h) = 0$. Hence $\delta^T_{(N, \eta)}$ induces a map $\sbAut(N, \eta) \ra Z^1(\Wh(\Lambda))$.

Next consider any automorphism of $v$, $h=\svec{a&b\\c&d}\co N\oplus M \isora N\oplus M$, and observe that $h$ must preserve $N$, the radical of $\phi =\psi + \epsilon \psi^*$.  It follows that $c=0$, that $a$ is a possibly non-simple isometry of $(N, \eta)$, that $d$ is a possibly non-simple isometry of the symmetric form $(M, \phi)$ and that $[a] = -[d] \in Z^1(\Wh(\Lambda))$.

\begin{lem}
\label{surnonslem}
Let $(N, \eta)$ be a linear form and $(M, \psi)$ a simple form.  Then
\[
i \co \sbAut(N,\eta) \lra \sbAut((N, \eta) \oplus (M,\psi)),
\quad [f] \mt [f \oplus \del\id_M]
\]
is surjective and injective on the fibres of $\delta^T_{(N, \eta)}$.
\end{lem}

\begin{proof}
The surjectivity of $i$ follows immediately from the fact that $\del(M, \psi)$ is trivial.  For the injectivity of $i$, let $f$ and $f'$ be two stable automorphisms of $\del((N,\eta)\oplus (Q, 0))$ representing elements $[f]$ and $[f']$ in $\sbAut(N,\eta)$ for a based module $Q$.  We assume that $Q = Q_0 \oplus Q_1$ where $Q_1$ is a based module of rank greater than the rank of $M$ such that $f|_{\del(Q_1,0)} = f'|_{\del(Q_1,0)} = \del {\rm Id}_{Q_1}$.  Assume that $\delta_{(N, \eta)}^T(f) = \delta_{(N, \eta)}^T(f')$ and that $$i(f)=i(f') \in 
\bAut((N,\nu) \oplus (Q,0) \oplus (M,\psi) \oplus (P,0))$$
for some based module $P$.  This means that there are isometries $h, h' \in \Aut((N,\nu) \oplus (Q,0) \oplus (M,\psi) \oplus (P,0))$
such that
\[
\partial h \circ (f \oplus \id_{\partial (M,\psi)} \oplus \id_{\partial(P,0)}) \circ \partial h' \simeq
f' \oplus \id_{\partial (M,\psi)} \oplus \id_{\partial(P,0)}.
\]
As above, the isometries $h$ and $h'$ give possibly non-simple isometries $a,a'$ of $(N,\eta) \oplus (Q,0) \oplus (P,0)$.  We claim that $h$ and $h'$ can be chosen so that $a$ and $a'$ are simple and hence $a, a' \in \Aut((N,\eta) \oplus (Q,0) \oplus (P,0))$.  It then follows that
\[
\partial a \circ (f \oplus \id_{\partial(P,0)}) \circ \partial a' \simeq f' \oplus\id_{\partial(P,0)}
\]
and therefore $[f]=[f'] \in\sbAut(N,\eta)$.

We demonstrate the claim as follows: firstly $\tau(a) = \delta_{(N, \eta)}^T(\del a)$ and similarly for $a'$, so the equality $\tau(a) + \del_{(N, \eta)}^T(f) + \tau(a') = \delta_{(N, \eta)}^T(f')$ gives rise to $\tau(a) = -\tau(a')$.   Now let $g$ be an isomorphism of $N \oplus Q \oplus P$ such that $g|_{N \oplus Q_0 \oplus P}= {\rm Id}$ and $\tau(g) = -\tau(a)$ (which is possible since the rank of $Q_1$ is larger than the rank of $M$ and from above we see that there is an isomorphism $d\co  M \isora M$ with $\tau(d^{-1}) = \tau(a)$).  It follows that $\tau(ag) = \tau(g^{-1}a') =0$ and so we replace $h$ and $h'$ by $(g \oplus {\rm Id}_M) \circ h$ and $(g^{-1} \oplus {\rm Id}_M) \circ h'$.
\end{proof}

\begin{proof}[Proof of Proposition \ref{simplelinearprop}]
Consider to begin the case where $v = (N, \eta)$ is linear and let $x = (H_\epsilon(L); L, N)$ represents $[x] \in l_{2q+1}(v, v)$.  By Example \ref{delex} we see that $\delta([x]) \in \sbAut(N, \eta)$ maps to $1 \in \sbAut((N, \eta) \oplus H_\epsilon(L))$ under the map $i$ of Lemma \ref{surnonslem}.  Now the injectivity part of Lemmas \ref{surnonslem} shows that $\delta^T_{(N, \eta)}(\delta([x]))$ is the remaining obstruction to $\delta([x])$ being equal to $1$.  Moreover, the considerations just prior to Lemma \ref{surnonslem} show that $\delta^T_{(N, \eta)}(\delta([x]))$ is equal to the torsion of an isometry of the symmetric form $H^\epsilon(L)$.  Applying Lemma \ref{bzerolem} and the assumption that ${\rm UWh}(\Lambda) = {\rm U'Wh}(\Lambda)$ we see that there exists a $[z_1] \in L_{2q+1}^{}(\Lambda)$ such that $\delta([x] + [z_1]) = 1 \in \sbAut(v)$.  By Corollary \ref{maincor} there is a $[z_2] \in L_{2q+1}^s(\Lambda)$ such that $[x] + [z_1] + [z_2]$ is elementary.

The general case for $v = (N, \eta) \oplus (M, \psi)$ now follows immediately from Lemma \ref{surnonslem}.
\end{proof}

\begin{cor}
\label{fieldcor}
If $\Lambda$ is a field with $\cha\Lambda\neq 2$ or $\Lambda=\bZ/2\bZ$ and we set $Z = \widetilde K^1(\Lambda)$, then all elements of $\ol{\Lambda}$ are elementary.
\end{cor}

\begin{proof}
Let $x=(M,\psi; L,V)$ be an \eq\ over $\Lambda$ and let $(V,\theta)$ and
$(V^\perp,\theta^\perp)$ be the boundaries of $x$.  By proposition \ref
{prepprp} $(V, \theta) \oplus (M, \psi) \cong (V^{\perp}, -\theta^{\perp}) \oplus (M, \psi)$ and so by Witt's cancellation theorem for $\Lambda \neq \bZ/2$ and by \cite{Bro72} Theorem III.1.14 for $\Lambda = \bZ/2$, we deduce that $(V,\theta)\cong (V^\perp,-\theta^\perp)$.  Hence $[x] \in l_{2q+1}((V,\theta), (V,\theta))$.  Moreover, as $\Lambda$ is a field, $(V,\theta)$ is the orthogonal sum of a linear form and a nonsingular form. Now Proposition \ref{simplelinearprop} states that $L_{2q+1}(\Lambda)$ acts transitively on $l_{2q+1}((V, \theta), (V, \theta))$ but by \cite{Ran77} $L_{2q+1}(\Lambda) = 0$ and we are done.
\end{proof}


We conclude the subsection with a simple example.

\begin{ex}
\label{L_2tauex}
Let $\Lambda=\bZ$ and $\epsilon=-1$.  Then $\kappa\co \sbAut(\bZ, 1)
\ra L_2(\bZ)=\bZ/2\bZ$ is a bijection with $\kappa([{\rm Id}_{\bZ}, {\rm Id}_{\bZ},1])$ having Arf invariant $1$.   The sequence of Corollary \ref{maincor} gives
\[0 \dlra{\cong} l_3(\bZ, 1) \dlra{} \sbAut(\bZ, 1) \dlra {\cong} L_2(\bZ).\]

\end{ex}

\subsection{Linking forms}
\label{linkingsubsec}
In this section we show how to calculate boundary isomorphism sets of forms which become nonsingular after localisation.  To avoid complications with torsions we assume throughout the section that $\Wh(\Lambda) = \{ 0 \}$ (see Remark \ref{torsionrem}).  However, torsions can be interlaced with what follows using \cite{Ran81}[Chapter 3.7].  The following technical lemma applies for all the forms considered later in the section.

\begin{lem}
\label{bijbisolem}
Let $(V,\theta)$ and $(V', \theta,)$ be $\epsilon$-quadratic
forms with stably isomorphic boundaries and injective symmetrisations $\lambda=\theta+\epsilon\theta^*\co V \ras V^*$ and
$\lambda' = \theta'+\epsilon\theta'^* \co V' \ras V'^*$.  Then
\[s\co \bIso((V,\theta), (V', \theta')) \cong \sbIso((V,\theta), (V', \theta')). \]
\end{lem}

\begin{proof}
It suffices to show that
\[s_Q \co\bIso((V,\theta), (V', \theta')) \cong \bIso((V,\theta) \oplus (Q, 0), (V', \theta') \oplus (Q, 0))\]
is an isomorphism for any based module $Q$.  So let $f=(\alpha,\beta,\nu)\co y \isora y'$ be an isomorphism with
\begin{eqnarray*}
y &=& (F,\left( \svec{\gamma\\\mu},\psi \right)G)
=\del(V,\theta) \oplus (P,P^*)  \oplus \del(Q,0)
\\
y' &=& (F',\left( \svec{\gamma'\\\mu'},\psi' \right)G')
=\del(V',\theta') \oplus (P',P^{'*}) \oplus \del(Q,0).
\end{eqnarray*}
We write $\beta=(b_{ij})_{1\le i,j\le 3} \co V \oplus P^*\oplus Q \isora V' \oplus {P'}^* \oplus Q$ and
similarly $\alpha=(a_{ij})$, $\nu=(n_{ij})$. 
The equation $\alpha^{-*}\mu=\mu'\beta$ entails that $b_{13}$ and $b_{23}$ are zero.  Hence $b_{33}$ and $\widetilde{\beta}=\svec{b_{11}&b_{12}\\b_{21}&b_{22}}$
are isomorphisms.
Evaluating $\alpha(\gamma+(\nu -\epsilon\nu^*)^*\mu)=\gamma'\beta$ yields the
vanishing of $a_{13}$ and $a_{23}$ and that $a_{22}=b_{22}$. Therefore
$\widetilde{\alpha}=\svec{a_{11}&a_{12}\\a_{21}&a_{22}}$ is an isomorphism.
Setting $\widetilde{\nu}=\svec{n_{11}&n_{12}\\n_{21}&n_{22}}$ one observes that
$C(f) := (\widetilde{\alpha},\widetilde{\beta},\widetilde{\nu})$ is an isomorphism
$\del(V, \theta) \oplus (P, P^*) \cong \del(V', \theta') \oplus (P', {P'}^*)$.

Moreover, one may also check that if
$\Delta$ is a homotopy from $f$ to another isomorphism $f' = (\alpha', \beta', \nu')$, then
$C(\Delta) :=  \svec{\Delta_{11}&\Delta_{12}\\\Delta_{21}&\Delta_{22}}$ is a homotopy from $C(f)$ to $C(f')$.
It also clear that if $g\co y\isora y$ and $g'\co y'\isora y'$ are isomorphisms,
then $C(g' \circ f \circ g) = C(g') \circ C(f) \circ C(g)$.
If $h$ is an isometry of $(V, \theta) \oplus (Q, 0)$ and ${\rm pr}_V\co V \oplus Q \ras V$ the projection then
$h_V := {\rm pr}_V \circ h|_V$ is an isometry of $(V, \theta)$ and $C(\del h) = \del (h_V)$.

We now show the injectivity of $s_Q$.  Let $f_0, f_1 \in \Iso(\del(V, \theta), \del(V', \theta'))$ be such that
$s_Q([f_0]) = s_Q([f_1]) \in \bIso((V, \theta)\oplus (Q, 0), (V', \theta') \oplus(Q, 0))$.
Then there are isometries $h \in \Aut((V, \theta) \oplus (Q, 0))$ and $h' \in \Aut((V', \theta') \oplus (Q, 0))$
such that $f_0 \oplus \del\id_Q$ is stably homotopic to $\del h' \circ (f_1 \oplus \del\id_Q) \circ \del h$.
Using the fact that $f_i = C(f_i \oplus \del\id_Q)$ and applying $C$ to the homotopic isomorphisms above we see that
$f_0$ is homotopic to $C(\del h') \circ f_1 \circ C(\del h) = \del h'_{V'} \circ f_1 \circ \del h_V$ and hence
$[f_0] = [f_1] \in \bIso((V,\theta), (V', \theta'))$.

We now turn to the surjectivity of $s_Q$.  Given $f$ as in the beginning,
consider the automorphism $g = f \circ(C(f) \oplus \del \id_Q)^{-1}$ and,
reusing notation, write $g = (\alpha, \beta, \nu)$ where $\alpha = (a_{ij})$, $\beta = (b_{ij})$ and $\nu = (n_{ij})$.
Calculation gives that
\begin{eqnarray*}
(a_{ij})&=&\svec{1&0&0\\0&1&0\\a_{31}&a_{32}&a_{33}}\\
(b_{ij})&=&\svec{1&0&0\\0&1&0\\b_{31}&a_{32}&b_{23}}\\
(n_{ij})&=&\svec{0&0&n_{13}\\0&0&0\\0&n_{32}&n_{33}}
\in Q_\epsilon(V\oplus Q\oplus P^*)
\end{eqnarray*}
Now the isomorphism $h=\svec{\id_V&0\\-a_{33}^{-1}a_{31}&a_{33}^{-1}}$ is an isometry of
$(V,\theta)\oplus (Q,0)$.  Replacing $g$ by $g \circ (\del h \oplus \id_{(P,P^*)})$ reduces to the case
the case where
$a_{33}=b_{33}=\id_Q$ and $a_{31}=0$. From the equation $\alpha(\gamma+
(\nu+\epsilon\nu^*)^*\mu)=\gamma\beta$
it follows that $n_{32}=\epsilon b_{32}$ and $b_{31}=n_{13}^*\lambda$.
Then $\Delta=\svec{
0 & 0               & n_{13} \\
0 & 0   & b_{32}^* \\
0 & -a_{32}             & -\epsilon n_{33}^*+a_{32}b_{32}^*}$ is a homotopy from
$g \circ (\del h \oplus \id_{(P,P^*)})$ to the identity.
Thus $f$ is homotopic to $(C(f) \oplus \id_Q) \circ (\del h^{-1} \oplus \id_{(P,P^*)})$ which means
that $[f] = s_Q[C(f)]$ and so $s_Q$ is surjective.
\end{proof}

%
%
%

Let $S\subset\Lambda$ be a central and multiplicative subset and denote the localisation of $\Lambda$ away
from $S$ by $S^{-1}\Lambda$.  If $P$ is a $\Lambda$-module then $S^{-1}P := S^{-1}\Lambda \otimes_{\Lambda} P$
is the induced $S^{-1}\Lambda$ module.
First we repeat some definitions from \cite{Ran81}[Chapters 3.1 and 3.4].

\begin{defi}
\label{multsetdef}
\begin{enumerate}
\item
Let $P$ and $Q$ be \fg free modules.
A homomorphism $f\in \Hom_\Lambda(P,Q)$ is called an
{\bf $S$-isomorphism} if
$
S^{-1}f := f \otimes_\Lambda \id_{S^{-1}\Lambda} \in \Hom_{S^{-1}\Lambda}(S^{-1}P,S^{-1}Q)
$
is an isomorphism.

\item
A {\bf $(\Lambda, S)$-module $M$} is a $\Lambda$-module $M$ such that
there is an exact sequence
$
0\ra P \stackrel{d}{\ra} Q \ra M \ra 0
$
where $d$ is an $S$-isomorphism.

\item
An {\bf $\epsilon$-symmetric linking form $(M,\phi)$ over $
(\Lambda,S)$} is a $(\Lambda,S)$-module $M$
together with a pairing $\phi \co M \times M \ras S^{-1}\Lambda/
\Lambda$ such that
$\phi(x,-)\co M \rightarrow S^{-1}\Lambda/\Lambda$ is $\Lambda$-linear for all $x\in M$ and
$\phi(x,y)=\epsilon\overline{\phi(y,x)}$ for all $x,y\in M$.

\item
A {\bf split $\epsilon$-quadratic linking form $(M,\phi,\nu)$ over
$(\Lambda,S)$} is an
$\epsilon$-symmetric linking form $(M,\phi)$ over $(\Lambda,S)$
together with a map
$\nu \co M \ras Q_{\epsilon}(S^{-1}\Lambda/\Lambda)$
such that for all $x,y\in M$ and $a\in\Lambda$
\begin{enumerate}
\item $\nu(ax)=a\nu(x)\bar a \in Q_\epsilon(S^{-1}\Lambda/\Lambda)$

\item $\nu(x+y)-\nu(x)-\nu(y)=\phi(x,y)\in Q_\epsilon(S^{-1}\Lambda/\Lambda)$

\item $(1+T_\epsilon)\nu(x) = \phi(x,x)\in S^{-1}\Lambda/\Lambda$
\end{enumerate}

\item An {\bf isometry} between $\epsilon$-quadratic linking forms $(M_0,\phi_0,\nu_0)$ and $(M_1,\phi_1,\nu_1)$
is a $\Lambda$-module isomorphism $f\co M_0 \isora M_1$ such that
$\phi_0(x, y) = \phi_1(f(x),f(y))$ and $\nu_0(x) = \nu_1(f(x))$ for all $x, y \in M_0$.

\item We write $\Iso_S((M_0,\phi_0,\nu_0), (M_1,\phi_1,\nu_1))$ and $\Aut_S(M,\phi,\nu)$ for, respectively, the {\bf set of isometries} between $\epsilon$-quadratic linking forms and the {\bf group of automorphisms} of an $\epsilon$-quadratic linking form.
\end{enumerate}
\end{defi}

\begin{defi}
\label{Sboundarydef}
Let $(V,\theta)$ be an $\epsilon$-quadratic form and let $\lambda=\theta+ \epsilon\theta^*$.
\begin{enumerate}
\item
The form $(V,\theta)$ is {\bf $S$-nonsingular} if $\lambda$ is an $S$-isomorphism.

\item  The {\bf $S$-boundary} of an $S$-nonsingular form $(V, \theta)$ is the 
split $\epsilon$-quadratic linking form
$\del_S (V,\theta) := (\coker\lambda,\phi,\nu)$ given by
\begin{align*}
\phi\co \coker\lambda \times \coker\lambda
&\lra S^{-1}\Lambda/\Lambda,
&(x,y)
&\mt \dfrac{x(z)}{s}
\\
\nu\co \coker\lambda
&\lra Q_\epsilon(S^{-1}\Lambda/\Lambda),
&y
&\mt \dfrac{1}{\bar s} \theta(z,z) \dfrac{1}{s}
\end{align*}
with $x,y\in V^*$, $z\in V$, $s\in S$ such that $sy=\lambda(z)$.   
We call the pair $(\coker \lambda, \phi)$ the {\bf symmetric $S$-boundary} of $(V, \theta)$.

\item
The {\bf boundary of an isometry $h\co(V,\theta)\isora (V',\theta')$ of
$S$-nonsingular $\epsilon$-quadratic forms} is the isometry
$\del_S h:= [h^{-*}]\co \del_S(V,\theta) \isora \del_S(V', \theta')$.
\end{enumerate}
\end{defi}

Let $(V, \theta)$ and $(V', \theta')$ be $S$-nonsingular forms such that $\del(V, \theta) \cong \del(V', \theta')$.
There is a group homomorphism
\[\begin{array}{ccc}
q \co \Iso(\del(V, \theta), \del(V', \theta')) & \lra & \Iso_S(\del_S(V, \theta), \del_S(V', \theta')) \\
\left[(\alpha,\beta,\nu)\right] &\mt & \left[\alpha^{-*}\right]
\end{array}\]
relating the isomorphisms of quasi-formations and the associated linking forms.
In order to investigate this map we start by considering the case where $(V, \theta) = (V', \theta')$.

\begin{prop}
\label{linkautprp}
Let $(V, \theta)$ be an $S$-nonsingular form, let $\lambda = \theta + \epsilon \theta^*$ and
let $K_\theta$ be the kernel of the map
\[
\widehat{Q}^{-\epsilon}(V^*)\lra \widehat{Q}^{-\epsilon}(V),\quad \nu \mt \lambda^*\nu\lambda.
\]
There is an exact sequence of groups
\[
0 \lra K_\theta \lra \Aut(\del (V,\theta)) \stackrel{q}{\lra} \Aut_S(\del_S (V,\theta)) \ra 0.\]
\end{prop}

\begin{proof}
Define the homomorphism
\begin{align*}
\kappa\co K_\theta &\lra \Aut(\del (V,\theta)),\quad
\chi  \mt (\id_V,\id_V,\chi).
\end{align*}
The surjectivity of $q$ follows from the proof of \cite{Ran81}[Proposition 3.4.1] and
\cite{Ran80a}[Proposition 1.5]. Let $(\alpha,\beta,\nu)$ be a stable isomorphism of
$\del (V,\theta)$ which is  in the kernel of $q$.
Then $\alpha^{-*}= 1+\mu\Delta^*$ for some $\Delta\in\Hom_\Lambda
(V^*,V)$. Because of the equation
$\alpha^{-*}\mu=\mu\beta$ it follows that $\beta=1+\Delta^*\mu$.
Hence $\Delta$ is a homotopy between $(\alpha,\beta,\nu)$ and $(1,1,\nu)$ for some
$\nu\in Q_{\epsilon}(V^*\oplus P^*)$. But by the definition of homotopy \ref{formationdef}, we see that $\nu$ must have trivial symmetrisation and satisfy $\lambda^* \nu \lambda = 0$.  Thus $(1,1,\nu)\in\im(\kappa)$.  The injectivity of $\kappa$ follows quickly from Lemma \ref{sthomlem} and the fact that $\lambda$ is injective.
\end{proof}

\begin{cor}
\label{linkisocor}
Let $(V, \theta)$ and $(V', \theta')$ be $S$-nonsingular forms with $\del(V, \theta) \cong \del(V', \theta')$.  Then there is an exact sequence
\[
0 \lra K_\theta \lra \Iso(\del(V, \theta), \del(V', \theta')) \stackrel{q}{\lra} \Iso_S(\del_S(V, \theta), \del_S(V', \theta')) \lra 0\]
in the sense that there is a free action of $K_\theta$ on $\Iso(\del(V, \theta), (V', \theta'))$ with orbits the fibres of $q$.
\end{cor}

\begin{proof}
The automorphism group $\Aut(\del(V, \theta))$ acts freely and transitively on the set $\Iso(\del(V, \theta), \del(V', \theta'))$ by precomposition.  Restricting to $\kappa(K_\theta) \subset \Aut(\del(V, \theta))$ gives the required action which, is $[(\alpha, \beta, \nu)] + \chi = [(\alpha, \beta, \nu + \chi)]$.   Similarly  the group $\Aut_S(\del_S(V, \theta))$ acts freely and transitively on $\Iso_S(\del_S(V, \theta), \del_S(V', \theta'))$.  Moreover, the map $q$ maps the first action to the latter with kernel $K_\theta$ and the Corollary now follows.
\end{proof}

We now define a ``linking boundary isomorphism set" for $S$-nonsingular forms and relate it to the full boundary isomorphism set.

\begin{defi}
Let $(V, \theta)$ and $(V', \theta')$ be $S$-nonsingular forms.  The {\bf linking boundary isomorphism set $\bIso_S((V,\theta), (V', \theta'))$} is the
set of orbits of the group action
\begin{eqnarray*}
(\Aut(V,\theta) \times \Aut(V', \theta'))\times \Aut_S(\del_S(V,\theta))
&\ra& \Iso_S(\del_S(V,\theta), \del_S(V', \theta'))\\
(h,g,f) &\mt& \del_S h \circ f \circ\del_S g^{-1}.
\end{eqnarray*}
When $(V', \theta') = (V, \theta)$ we have the {\bf linking boundary automorphism set} $\bAut_S(V, \theta) : = \bIso_S((V, \theta), (V, \theta))$.  The orbit of the identity map is denoted by $1 \in \bAut_S(V,\theta)$.
\end{defi}

The map $q$ above
induces a map
\[q_{\rm b} \co \bIso((V, \theta), (V', \theta')) \ra \bIso_S((V, \theta), (V', \theta')), \quad [f] \ra [q(f)]. \]
As $(V,\theta)$ and $(V', \theta')$ are $S$-nonsingular $\epsilon$-quadratic forms the maps
$\lambda=\theta+\epsilon\theta^*$ and $\lambda' = \theta' + \epsilon \theta^{'*}$ are injective otherwise $S^{-1}\lambda$ and $S^{-1}\lambda'$ could not be isomorphisms.  Hence we may apply Lemma \ref{bijbisolem} to deduce that
\[
s\co \sbIso((V, \theta), (V', \theta')) \isora \bIso((V, \theta), (V', \theta')).
\]

\begin{prop} \label{bIsosprop}
Let $(V,\theta)$ and $(V', \theta')$ be $S$-nonsingular $\epsilon$-quadratic forms with $\del(V, \theta) \cong \del(V', \theta')$.
There is a surjection
\[ \sbIso((V,\theta), (V', \theta')) \stackrel{q_{\rm b} \circ s^{-1}}{\lra} \bIso_S((V,\theta), (V', \theta')) \]
such that $K_\theta$ is mapped onto {\em but not necessarily into} its fibres.
\end{prop}

\begin{proof}
The surjectivity of $q_{\rm b} \circ s^{-1}$ follows from the surjectivity of $q_{\rm b}$.
Let $[f_0], [f_1] \in \bIso((V, \theta), (V', \theta'))$ be such that $q_b([f_0]) = q_b([f_1])$.
Then there are isometries $h$ of $(V, \theta)$ and $h'$ of $(V', \theta')$ such that
$q(f_0) = \del_S h' \circ q(f_1) \circ \del_S h$.  But since $q(\del h) = \del _S h$ and similarly for $h'$,
we have that $q(f_0) = q(\del h' \circ f_1 \circ \del h)$ and now by Corollary \ref{linkisocor} it follows
that there is a $\chi \in K_\theta$ such that $f_0 \circ \kappa(\chi) = \del h' \circ f_1 \circ \del h$,
completing the proof.
\qedhere
\end{proof}

\begin{rem}
We remind the reader that $K_\theta \subset \hQ^{-\epsilon}(V^*)$, that the group $\hQ^{-\epsilon}(V^*)$ has exponent $2$ in general and that $\hQ^{-\epsilon}(V^*)$ may vanish, for example when $\Lambda = \bZ[\pi]$, $w = 0$ and $\epsilon = 1$.
\end{rem}

\subsection{On $l_{2q+1}(\bZ)$}
\label{lZsec}


We begin with the $+$-quadratic case and $l_1(\bZ)$ which is in
general very complex but stably very simple.  Every $+$-quadratic
form over $\bZ$, $w = (W, \sigma)$, determines and is determined by
its symmetrisation, $\bar w = (W, \sigma + \sigma^*)$, which is an
even symmetric bilinear form.  Moreover each $w$ splits uniquely
up to isomorphism as $(W, \sigma) \cong (V, \theta) \oplus (U, 0)$
where $v = (V, \theta)$ is nondegenerate: that is, if $\bar v =
(V, \lambda)$ then $\lambda \co  V \ra V^*$ is injective.  It follows
that $\mathcal{F}_{0}^{\rm zs}(\bZ)$ can be identified with the
set of isomorphism classes of nondegenerate symmetric bilinear
forms.

Following \cite{Nik79}, we call nondegenerate even
symmetric bilinear forms lattices.  A lattice $(V, \lambda)$ is
called indefinite if $\lambda(x, x) = 0$ for some $x \in V - \{ 0
\}$.  Of course, the classification of lattices, and in particular
definite lattices, is an extremely rich and complicated subject.
Classically, two lattices are said to belong to the same genus if
they become isomorphic when tensored with the $p$-adic integers for
each prime $p$.  We record here just two basic facts.

\begin{prop} \label{latticeprop}
\begin{enumerate}
\item The set of isomorphism classes of lattices of a fixed rank is finite but may be arbitrarily large.
\item If $\bar v$ and $\bar v'$ are stably isometric, then they belong to the same genus.
\end{enumerate}
\end{prop}

\begin{proof}
Part (i) can be found in \cite{HuMi73}[II.1.6].  The second part follows from \cite{Nik79}[ Corollary 1.9.4] which states that the rank, signature and induced quadratic linking form on the boundary determine the genus of a lattice.  But these invariants are agree for $\bar v$ and $\bar v'$ if and only if they agree for $\bar v \oplus H_+(\bZ^k)$ and $\bar v' \oplus H_+(\bZ^k)$.
\end{proof}

Turning to $l_{1}(\bZ)$, we fix a nondegenerate quadratic form $v$ and consider
\[
l_{1}(v) = \bigcup_{v' \sim v} l_{1}(v, v') \subset  l_{1}(\bZ).
\]
Proposition \ref{latticeprop} (i), implies that the above union may be taken over a finite set of $v'$ and by Theorem \ref{generalmainthm} each $l_{2q+1}(v, v')$
sits in the exact sequence
\[
L_1(\bZ)=0 \dlra{\rho} l_{1}(v, v') \dlra{\delta} {\rm sbIso}(v, v') \dlra{\kappa} L_0(\bZ)=\bZ.
\]
Since the signature of twisted doubles $w \cup_g -w$ of quadratic forms is zero (see \cite{Ran98} §42), $\kappa$ is the zero map.

Setting $S:=\bZ - \{0\}$ we see that every nondegenerate form is $S$-nonsingular and applying Lemma \ref{bijbisolem} we may conclude that
\[
l_{1}(v, v') = \sbIso(v, v') = \bIso_S(v, v').
\]
This set is finite because it is a quotient of the finite group of isomorphisms between the boundary quadratic linking forms: $\Iso(\del_S v, \del_S v')$.  However even for very simple forms $v$ the set $\sbAut(v)$ can be arbitrarily large.

\begin{ex} \label{sbautex}
Let $(\bZ,n)$ be the quadratic form where
$n = p_1^{m_1} \dots p_k^{m_k}  \in \bZ$ is a product of $k$ distinct odd primes powers $p_i^{m_i}$ with $m_i \geq 1$.  Then $\Aut_{S}(\del(\bZ,n))\cong\Aut(\bZ/2n\bZ, 1)$ and
$(\bZ/2n\bZ, 1) \cong (\bZ/2\bZ, 1) \bigoplus_{i=1}^k (\bZ/p_i^{m_i}\bZ, 1)$ has an automorphism group isomorphic
to $(\bZ/2\bZ)^k$ containing $\Aut(\bZ, n) = \{ \pm 1 \}$ as a normal subgroup.  Therefore $\sbAut(\bZ, n) \cong (\bZ/2\bZ)^{k-1}$.
\end{ex}

%

Summarizing our discussion we have

\begin{prop} \label{l_1(Z)finprop}
For each $+$-quadratic form $v$ over $\bZ$ the set $l_1(v)$ is finite but there are $v$ for which $\{ [v'] \,| \, [v'] \sim [v] \}$ or ${\rm sbAut}(v)$ is arbitrarily large.
\end{prop}

We now consider the question for which nondegenerate quadratic forms $v$ does strict cancellation hold?  That is, for which forms $v$ is  $l_{4k+1}(v)=\{ e([v]) \}$?  From the discussion of Theorem \ref{generalmainthm} above this is equivalent to asking whether $v \sim v'$ implies that $v \cong v'$ and, if so, whether ${\rm sbAut}(v) = \{1 \}$.

In \cite{Nik79} Nikulin explicitly raised very similar questions concerning symmetric bilinear forms and we report translations of his results in parts (i) and (ii) of the following proposition.  Here ${\rm rk}(V)$ denotes the rank of a free abelian group $V$, $l(G)$ denotes the minimal number of generators of a finite abelian group $G$ and $l_p(G) = l(G_p)$ where $G_p$ is the $p$-primary component of $G$ for a prime $p$.

\begin{prop}
\label{cancellationforVZ+1prop}
Let $v=(V, \theta)$ be a nondegenerate quadratic form and let $(G,\phi)$ be the associated symmetric boundary (Definition \ref{Sboundarydef}). Then strict cancellation holds for $v$ if
any of the following conditions hold.

\begin{enumerate}
\item
The symmetric form $(V, \theta + \theta^*)$ is indefinite and satisfies

\begin{enumerate}
\item
$\rk(V) \geq l_p(G) +2$ for all primes $p \neq 2$,

\item
if $\rk(V) = l_2(G)$ then the symmetric boundary associated to
$\left( \bZ^2, \svec{0 & 2 \\ 0 & 0} \right)$
is a summand of the $2$-primary component of $(G, \phi)$.
\end{enumerate}

\item
The symmetric form $(V, \theta + \theta^*)$ is isomorphic to one of the classical lattices $E_8$, $E_7$, $E_6$, $D_5$ or $A_4$.
\item
The quadratic form $v$ is isomorphic to $(\bZ, p)$ for any prime $p$.

\end{enumerate}
\end{prop}

\begin{proof}
The proof of parts (i) and (ii) requires only that we translate some results of Nikulin concerning lattices and the isomorphisms of their boundaries.  Specifically Theorem 1.14.2 and Remark 1.14.6. \cite{Nik79} assert for any lattice $(V, \theta + \theta^*)$ in part (i) or (ii) respectively that all lattices in the same genus as $(V, \theta + \theta^*)$ are isometric to $(V, \theta + \theta^*)$ and that ${\rm sbAut}(v) = \{ 1 \}$.  Applying Proposition \ref{latticeprop} (ii) we are done.

For part (iii), note that any form which is stably isometric to $(\bZ, p)$ is a rank $1$ form $v'$ with $\del_S(v') \cong (\bZ/2p, 1)$ and so $v'$ must be isomorphic to $(\bZ, p)$.  Hence $l_{5}(v) = l_5(v, v)$.  Moreover, $\Aut(v) \cong \bZ/2 \cong \Aut(\del_S(v))$ and so $\bAut(v) = \{1 \}$.


\end{proof}

\begin{rem} \label{Zcancelrem}
The case (i) includes the case when $v = w \oplus H_+(\bZ)$ splits off a single hyperbolic plane.
\end{rem}


In the skew-quadratic case the monoid $l_{3}(\bZ)$ is as simple as one can hope.

\begin{prop} \label{leml_3Zelem}
All elements of $l_{3}(\bZ)$ are elementary.
\end{prop}

\begin{proof}
Let $x = (M, \psi; L, V)$ be a skew-symmetric quasi-formation representing $[x] \in l_{4k+3}(\bZ)$.  
As $L_{3}(\bZ) = 0$ it suffices to show that $V$ has a Lagrangian complement in $M$.  As in Example \ref{delex}[iii], we set $y = x \oplus \delta(M, \rho)$ where $\rho$ represents $\psi$ and obtain that $\delta(y) = 1 \in {\rm bAut}(V \oplus M, \theta \oplus \psi)$.  Now by Lemma \ref{unionlem}[iii] it follows that $V \oplus M$ has a Lagrangian complement in $(M,\psi ; L, V) \oplus \delta(M, \rho)$.   We now wish to apply the penultimate section (pp. 742-3) of the proof of Corollary 4 in \cite{Kre99} to conclude that $V$ has a Lagrangian complement in $M$: observe that $\delta(M, \rho)$ corresponds to Kreck's $(H_\epsilon(\Lambda^{2k}), H_\epsilon(\Lambda^k))$ under the isomorphism of Proposition \ref{newlprp}.
However we note that the argument there relies on \cite{Kre99}[Proposition 9] and for the case $\epsilon = -1$, $\Lambda = \bZ$ the proof of Proposition 9 ignores the quadratic refinement.  We may therefore conclude only that $V$ has complement $W$ in $M$ on which $\phi = \psi - \psi^*$ vanishes.  Let $\mu\co M \ras Q_{-1}(\bZ)=\bZ/2\bZ$ be the quadratic refinement corresponding to $\psi$ as in Remark \ref{splitformrem}. The restriction $\mu|_W$ is a homomorphism $W \ras \bZ/2\bZ$.  Thus we may choose bases of
$\{v_1, \dots , v_k\}$ and $\{w_1, \dots , w_k\}$ of $V$ and $W$
respectively so that
\[
\phi(v_i, w_j) = \delta_{ij},
\qquad
\phi(w_i, w_j) = 0,
\quad
\mu(w_i) = \delta_{1i}
\]
where $\delta_{ij} = 0$ if  $i \neq j$ and $1$ if $i=j$.  As
$\{v_1, \dots, v_k, w_1, \dots, w_k\}$ forms a basis for the
hyperbolic form $(M, \psi)$ we see that $\mu(v_1) = 0$.  This is
because $\{w_1, v_1\}$ spans a nonsingular sub-form of $(M,
\psi)$ which we call $(M_0, \psi_0)$, and thus $(M, \psi) \cong
(M_0, \psi_0)\oplus (M_0^{\perp}, \psi_0^{\perp})$. But $\{w_2,
\dots w_k\}$ spans a Lagrangian in $(M_0^{\perp}, \psi_0^{\perp})$
and so the Arf invariant of this form vanishes.  Hence the Arf
invariant of $(M_0, \psi_0)$ vanishes but this is $\mu(w_1) \cdot
\mu(v_1)$ and so $\mu(v_1) = 0$.  We therefore let $K$ be the
Lagrangian with basis $\{w_1 + v_1, w_2, \dots, w_k\}$ and observe
that $K$ is a Lagrangian complement for $V$.
\end{proof}

\begin{cor} \label{corl_3Z}
There is a sequence of monoid isomorphisms
\[
\mathcal{E}l_{3}(\bZ) = l_{3}(\bZ)
\dlra{b}
\Delta(\mathcal{F}_{2}^{zs}(\bZ))
\dlra{\cong}
\mathcal{F}_{2}^{\rm zs}(\bZ).
\]
\end{cor}

\section{Strict cancellation and absolute stable rank}
Recall that strict cancellation holds for a $0$-stabilised form $[v]$ if for any quasi-formation $x = (M, \psi;L,V)$ with $[V, \theta] = [v]$, $[x] \in \ol{\Lambda}$ is elementary.  Below we define the absolute stable rank of a ring $\Lambda$, ${\rm asr}(\Lambda)$, and in this section $\Lambda$ is always a ring with finite absolute stable rank: e.g. $\Lambda = \bZ[\pi]$ for $\pi$ polycyclic-by-finite.  We shall prove that strict cancellation holds for any $0$-stabilised form $[v]$ over $\Lambda$ if $[v] = [w] \oplus [H_\epsilon(\Lambda^{k})]$ and $k \geq {\rm asr}(\Lambda) + 1$.  

The proof follows from Corollary \ref{maincor} and the cancellation and transitivity theorems of \cite{MKV88} which hold for rings with finite absolute stable rank: if $(V,\theta)$ contains a hyperbolic of rank ${\rm asr}(\Lambda) +1$ then any stable isometry from $(V,\theta)$ to another form can be replaced by an isometry that induces the same boundary isomorphism. In addition, $\oLs{\Lambda}$ acts trivially on $[x]$.

Under these circumstances we show that $[x]\in\ol{\Lambda}$ is
elementary as follows.  By Proposition \ref{prepprp} $(V,\theta)$ and
$(V^\perp,\theta^\perp)$ are stably isometric and the cancellation theorem mentioned above states that
they are already isometric.  Therefore $[x] \in \ol{(V,\theta), (V,\theta)}$.

The second obstruction to $[x]$ being elementary, $\delta([x])$, is represented by the stable
isomorphism $f_j$ of Proposition \ref{orthoglueprp} which is the boundary of an isometry
$h\co (V,\theta)\oplus H_\epsilon(K) \isora (V^\perp,-\theta^\perp)\oplus H_\epsilon(K)$
by Proposition \ref{bdryextprp}. Again, the cancellation theorem allows us to replace
$h$ by an isometry $(V,\theta)\isora (V^\perp,-\theta^\perp)$ without changing the boundary isomorphism.
Hence $\delta([x]) = 1$ and $[x]$ is elementary modulo the action of $\oLs{\Lambda}$.
But the action of $\oLs{\Lambda}$ on $[x]$ is trivial and therefore $[x]$ is indeed elementary.

The topological consequence of this algebra is Theorem \ref{thm1} which is a cancellation result for 
stably diffeomorphic manifolds with polycyclic-by-finite fundamental group which split off enough $S^q \times S^q$ connected summands.


We now introduce the absolute stable rank of Magurn, Van der Kallen and Vaserstein which is
a generalisation of concepts of Bass and Stein.

\begin{defi}[\cite{MKV88}, \cite{ReGa67}]
\begin{enumerate}
\item
Let $S\subset\Lambda$. Let $J(S)$ denote the intersection of all maximal left ideals of $\Lambda$ which contain $S$. The ideal $J(0)$ is the
{\bf Jacobson radical of $\Lambda$}.

\item
The {\bf absolute stable rank of $\Lambda$, $\asr(\Lambda)$,} is the
minimum of
all integers $n$ with the following property:
for all $(n+1)$-pairs $
(a_i)_{0\leq i\leq n}$
in $\Lambda$ there is an $n$-pair $(t_i)_{0\leq i\leq n-1}$ in $
\Lambda$ such that
$a_n\in J(a_0+t_0a_n,\dots, a_{n-1}+t_{n-1}a_n)$. If no such $n$
exists we set $\asr(\Lambda)=\infty$.

\end{enumerate}
\end{defi}

%
%
%

Important examples of rings with finite absolute stable rank are the group rings
of polycyclic-by-finite groups.
\begin{defi} [\cite{Sco64}{[\S7.1]}] \label{polydefi}
A group $\pi$ is {\bf polycyclic-by-finite} if
there is a subnormal series
\footnote{A series is subnormal if
all $\pi_{i-1}\subset \pi_i$ are normal subgroups. It is not
necessary that the $\pi_i$ are normal subgroups of $\pi$.}
$1=\pi_0 \triangleleft \pi_1 \triangleleft \cdots \triangleleft \pi_s=
\pi$
such that $\pi_k/\pi_{k-1}$ is either cyclic or finite. The number of
infinite cyclic factors
is an invariant of $\pi$ called the {\bf Hirsch number $h(\pi)$}.
\end{defi}

\begin{thm} \label{asrthm}
Let $\pi$ be a polycyclic-by-finite group. Then $\asr(\bZ[\pi])\leq 2
+h(\pi)$.
\end{thm}
\begin{proof}
Let $\Lambda=\bZ[\pi]$. The ring $\Lambda$ is right-Noetherian
due to \cite{McRo87}[\S 1.5].
Consecutively, we apply \cite{Sta90}[Theorem A(i)], \cite{ReGa67}[(e)]
and \cite{Smi73} in order to show that
$
\asr(\Lambda)\leq 1+\Kdim(\Lambda/J(0))\leq 1+\Kdim(\Lambda) = 2
+h(\pi)
$ where $\Kdim$ denotes the Krull-dimension.
\end{proof}

We now introduce some important concepts related to cancellation of hyperbolic forms.

\begin{defi}[\cite{MKV88}, \cite{Bas73}{[I.5.1]}]
Let $(V,\theta)$ be an $\epsilon$-quadratic form and $\lambda=\theta+
\epsilon\theta^*$.
\begin{enumerate}
\item
The {\bf Witt-index $\ind(V,\theta)$} is the largest integer $k$
such that there exists a sub form of $(V,\theta)$ isometric to $H_
\epsilon(\Lambda^k)$.
(Then there is a decomposition $(V,\theta)\cong(V',\theta')\oplus H_
\epsilon(\Lambda^k)$ by \cite{Bas73}[I.3.2].)

\item
A vector $v\in V$ is {\bf $(V,\theta)$-unimodular} if there is a $w
\in V$ with $\lambda(v,w)=1$.

\item
A {\bf symplectic basis $S=(e_i, f_i)_i$ of an $\epsilon$-quadratic
form $H_\epsilon(P)$}
is an ordered basis of $P\oplus P^*$ with $\theta(u,v)=0$ for all $u,v
\in S$ except for $\theta(e_i,f_i)=1$
for all $i$.

\item
A {\bf hyperbolic pair $(e,f)$ of $(V,\theta)$} is a
symplectic basis of some hyperbolic subform of $(V,\theta)$.

\item
Let $u,v \in V$ and $a\in\Lambda$ be such that $u$ is 
$(V,\theta)$-unimodular, $\lambda(u,v)=0\in\Lambda$,
$\theta(u,u)=0\in Q_\epsilon(\Lambda)$ and
$\theta(v,v)=[a]\in Q_\epsilon(\Lambda)$. Then the {\bf (orthogonal)
transvection $\tau_{u,a,v}$}
is the homomorphism
\begin{eqnarray*}
V \lra V,\quad x \mt x+u\lambda(v,x)-\epsilon v \lambda(u,x)-\epsilon
ua\lambda(u,x).
\end{eqnarray*}
\end{enumerate}
\end{defi}

\begin{prop}
\label{transvprp}
We use the notation of the preceding definitions.
\begin{enumerate}
\item
All transvections are isometries of $(V,\theta)$.

\item
$\tau_{u,a',v'}\circ\tau_{u,a,v}=\tau_{u,a'+\lambda(v',v)+a,v}$.

\item
$\tau_{u,a,v}^{-1}=\tau_{u,\lambda(v,v)-a,-v}$.

\item
$[\partial\tau_{u,a,v}] = 1 \in \Aut(\partial(V,\theta))$.

\item
Let $(V,\theta)=H_\epsilon(\Lambda^k)$ and let $S:=\{e_i, f_i\}_{1\leq i\leq k}$
be the canonical symplectic basis.  Then $\tau_{u,a,v}\in RU_k(\Lambda,\epsilon)$ for all $u\in S$.

\end{enumerate}
\end{prop}
\begin{proof}
The first three statements are proved in \cite{Bas73}[I.5].
With the help of Lemma \ref{sthomlem} one computes that
$\Delta=-\epsilon u v^* + vu^*+uau^*$ defines a homotopy
$\del \tau_{u,a,v} \simeq \del \id_{(V,\theta)}$. For the last claim we write
$v=\sum_{i=1}^k a_i e_i +\sum_{i=1}^k b_i f_i$.
There is a decomposition
\[
\tau_{u,a,v}=
\prod_{i=1}^k\tau_{u,0,a_i e_i}\circ\prod_{i=1}^k\tau_{u,0,b_i f_i}
\circ \tau_{u,x,0}
\]
and each of the factors is in $RU_k(\Lambda,\epsilon)$.
\end{proof}

\begin{thm}[\cite{MKV88}{[Theorem 8.1]}]
\label{mkvmainthm}
Let $(V,\theta)$ be an $\epsilon$-quadratic form with $k:=\ind(V, \theta)\geq \asr(\Lambda)+2$.
Let $S:=\{e_i, f_i\}_{1\leq i \leq k}$ be a symplectic basis of a
hyperbolic sub form of $(V,\theta)$. Let $v,v'\in V$ be $(V,\theta)$-unimodular
with $\theta(v,v)=\theta(v',v')\in Q_\epsilon(\Lambda)$. Then there
exists an $f\in \Aut(V,\theta)$ mapping $v$ to $v'$ 
which is a product of transvections of
the form $\tau_{u,a,v}$ with $u\in S$.
\end{thm}

An analysis of \cite{MKV88}[Corollary 8.2] shows
that cancellation holds for forms with high enough Witt-index.

\begin{cor}
\label{wittcanccor}
Let $f\co(V,\theta)\oplus H_\epsilon(L) \isora (V',\theta')\oplus H_\epsilon(L)$
be an isometry of $\epsilon$-quadratic forms where $\ind(V,\theta)\geq \asr(\Lambda)+1$.
Then there is an isometry $f'\co(V,\theta) \isora (V',\theta')$ such that
$\partial f$ and $\partial f'$ are stably homotopic.
\end{cor}

\begin{proof}
We can assume that $(V,\theta)=(W,\sigma)\oplus H_\epsilon(K)$ where
$\rk(K) > \asr(\Lambda)$.
By the proof of \cite{MKV88}[Corollary 8.2], there are products
of transvections, $\sigma$ and $\sigma'$, such that
\begin{eqnarray*}
\sigma'\circ f \circ \sigma = \svec{g & 0 \\ 0 & h} \co
(V,\theta)\oplus H_\epsilon(L) \isora (V',\theta')\oplus H_\epsilon(L).
\end{eqnarray*}
This isometry has the same boundary as $f$ by Proposition \ref{transvprp}[iii].  Assume for a moment that there is a possibly non-simple isometry $\widetilde{h} \co H_\epsilon(K) \isora H_\epsilon(K)$
with $\tau(\widetilde{h}) =\tau(h)$. Then
$g \circ (\id_W \oplus \widetilde{h}^{})$ is a simple isometry
with the same boundary as $f$.

It remains to find $\widetilde{h}$. If $\rk(L) \leq \rk(K)$ we simply use
$h$ plus the identity on $H_\epsilon(L)/H_\epsilon(K)$.
Otherwise, we can again compose $h$ with transvections such that the result is
\begin{eqnarray*}
\svec{h' & 0 \\ 0 & h_2} \co H_\epsilon(L) = H_\epsilon(L') \oplus H_\epsilon(K) \isora
H_\epsilon(L') \oplus H_\epsilon(K).
\end{eqnarray*}
If $\rk(L') \leq \rk(K)$ then take $\widetilde{h} = h_2 \circ (\id\oplus {h'})$.
However, if $\rk(L') > \rk(K)$ then $h'_1 = h' \circ (\id\oplus h_2) \co H_\epsilon(L') \isora H_\epsilon(L')$
is a possibly non-simple isometry with $\tau(h'_1)=\tau(h)$. We repeat this
process for $h'$ until we arrive at the desired non-simple isometry
$\widetilde{h}$.
\end{proof}

A second consequence of Theorem \ref{mkvmainthm} is the following


\begin{cor}
\label{rutranscor}
Let $(V,\theta) = H_\epsilon(\Lambda^k)$ with $k\geq \asr(\Lambda)+2$.
Then $RU_\epsilon(\Lambda^k)$ acts transitively
on all bases of all hyperbolic planes in $(V,\theta)$.
\end{cor}
\begin{proof}
Let $\lambda=\theta+\epsilon\theta^*$ the underlying $\epsilon$-
symmetric form of $(V,\theta)$.
Let $\{e,f\}$ be a hyperbolic pair and let $\{e_i,f_i\}_{1\leq i\leq
k}$ be the standard symplectic bases of $(V,\theta)$.
By Theorem \ref{mkvmainthm} and Lemma \ref{transvprp} there is a $
\sigma\in RU_\epsilon(\Lambda^k)$ with
$\sigma(e)=e_1$. Write $\sigma(f)=a e_1+b f_1 + v$ where $a,b\in
\Lambda$ and $v$ is in the span of $e_2,\dots,e_k,f_2\dots f_k$.  If follows from $\lambda(\sigma(e),\sigma(f))=\lambda(e,f)=1$ and $\theta(f,f)=0
$ that $b=1$ and $\theta(v,v)=[-a]\in Q_\epsilon(\Lambda)$.
One easily computes that $\rho=\tau_{e_1,-\epsilon \overline{a}, -
\epsilon v}$ sends $f_1$ to $\sigma(f_1)$ and fixes $e_1$.
Hence $\rho^{-1}\sigma\in RU_\epsilon(\Lambda^k)$ maps $e$ to $e_1$
and $f$ to $f_1$.
\end{proof}

\begin{thm}
\label{cancpolthm}
Let $x=(H_\epsilon(L); L, V)$ be an \eq\ and let $v = (V,\theta)$ be the
induced form.
If $\ind(V,\theta)\geq \asr(\Lambda)+1$ then $[x]\in\ol{\Lambda}$ is elementary.
\end{thm}

\begin{proof}
Due to Corollary \ref{wittcanccor} and Proposition \ref{prepprp}, there is an
isometry 
\[
k'\co (V^\perp,-\theta^\perp)\isora (V,\theta).
\]
Therefore $\delta(x)=[\partial k'\circ f_j] \in \bAut(V, \theta)$ 
where $f_j$ is the stable isomorphism obtained by
applying Proposition \ref{orthoglueprp}
to the inclusion $j\co (V,\theta)\hookrightarrow H_\epsilon(L)$.
By Proposition \ref{bdryextprp} and Corollary \ref{wittcanccor}
there is an isometry $h\co (V,\theta) \isora (V^\perp, -\theta^\perp)$
such that $\del h\simeq f$. Hence $\delta(x)=\del(k'h)=1 \in
\bAut(V,\theta)$ and $[x]$ is elementary modulo the action of $L_{2q+1}^s(\Lambda)$ by Corollary \ref{maincor}.

Now we show that $L_{2q+1}^s(\Lambda)$ acts trivially on $[x]$. Let $k$ be the rank of $L$ and let $j=k-\asr (\Lambda)-1$.  Let $(e_i,f_i)_{1 \leq i \leq k}$ be
the standard symplectic basis of $H_\epsilon(L)$.  W.l.o.g the hyperbolic form spanned by
$(e_i,f_i)_{j \leq i \leq k}$, call it $H$, lies in $V$.  We denote by $H^\perp$ the orthogonal complement of $H$ in $H_\epsilon(L)$
(\ie the span of $(e_i,f_i)_{1\leq i < j}$)
and by $V'$ the orthogonal complement of $H$ in $V$.
Clearly $V=V'\oplus H$ and $V'\subset H^\perp$ by \cite{Bas73}[I.3.2].

Finally, we use an argument from the proof of \cite{Kre99}[Theorem 5]
to complete the proof.  Let $[z]\in L_{2q+1}^{s}(\Lambda)$. Then $[x+z]=[(H_\epsilon(L);L,\alpha(V)]$ for some $
\alpha\in \Aut(H_\epsilon(L))$.
Inductive application of Corollary \ref{rutranscor} shows that there
is a $\beta\in RU_\epsilon(L)$
with $\beta\alpha(e_i)=e_i$ and $\beta\alpha(f_i)=f_i$ for $1\leq i<j$.
Hence $\beta\alpha$ is the identity on $H^\perp$ and therefore on $V'$.
Moreover, $\beta\alpha \in \Aut(H_\epsilon(\Lambda))$ must map $H$ to
itself and therefore
$\beta\alpha(V)=V$. Therefore $[H_\epsilon(L);L,\alpha(V)]=[H_\epsilon
(L);L,V]\in\ol{\Lambda}$.
\end{proof}

Finally recall the defintion of $h'(\pi, q)$ from Theorem \ref{thm1}.

\begin{cor}
\label{polycyccor}
Let $\Lambda = \bZ[\pi]$ be the group ring of a polycyclic-by-finite group $\pi$ and let $[v]$ be a $0$-stabilised form.  If $[v] = [w] \oplus [H_\epsilon(\Lambda^k)]$ for $k \geq h'(\pi,q)$, then strict cancellation holds for $[v]$.
\end{cor}

\begin{proof}
If $\pi$ is infinite then $h'(\pi, q) = h(\pi) + 3 \geq {\rm asr}(\bZ[\pi]) + 1$ by Theorem \ref{asrthm} and we apply Theorem \ref{cancpolthm}.  If $\pi$ is trivial then we apply Remark
\ref{Zcancelrem} and Proposition \ref{leml_3Zelem}.  If $\pi$ is
non-trivial but finite, then the proof follows along the same
lines as the proof of Proposition \ref{leml_3Zelem} except that now
there is no gap to be filled in the proof of \cite{Kre99}[Proposition 9].
\end{proof}

\noindent
\textsc{School of Mathematical Sciences, University of Adelaide, Australia, 5005.}\\ \\
{\it E-mail addresses:} \texttt{diarmuidc23@gmail.com,  sixt@mathi.uni-heidelberg.de}

\end{document}